\documentclass[journal]{IEEEtran}  %

\usepackage{ifthen}
\newboolean{longver}
\setboolean{longver}{true}

\usepackage{amsmath}
\usepackage{amssymb}
\usepackage{amsthm}
\usepackage{enumitem}
\usepackage{xfrac}

\usepackage{tikz}
\usepackage{pgfplots}
\pgfplotsset{compat=1.17}
\usetikzlibrary{fit}
\usepackage{dsfont}
\usetikzlibrary{shapes,arrows}
\usetikzlibrary{patterns}

\usepackage{graphicx}
\usepackage{epsfig} 
\usepackage{epstopdf}
\usepackage{mathptmx} 
\usepackage{times} 
\usepackage{amsmath} 
\usepackage{amssymb}  
\usepackage[utf8]{inputenc}
\usepackage[english]{babel}
\usepackage{fancyhdr}
\usepackage{dsfont}

\usepackage{lipsum}
\ifCLASSOPTIONcompsoc
\usepackage[caption=false, font=normalsize, labelfont=sf, textfont=sf]{subfig}
\else
\usepackage[caption=false, font=footnotesize]{subfig}
\fi

\usepackage{color}

\newtheorem{thm}{Theorem}
\newtheorem{defn}{Definition}
\newtheorem{prop}{Proposition}
\newtheorem{lem}{Lemma}
\newtheorem{rem}{Remark}
\newtheorem{cor}{Corollary}

\newtheorem{cex}{Counterexample}

\newtheorem{ex}{Example}

\newcommand{\mcl}[1]{\mathcal{#1}}

\newcommand{\R}{\mathbb{R}}
\newcommand{\x}{\mathbf{x}}

\newcommand{\N}{\mathbb{N}}

\newcommand{\norm}[1]{\lVert{#1}\rVert}

\newcommand{\eps}{\varepsilon}

\title{\LARGE \bf
Sublevel Set Approximation in The Hausdorff and Volume Metric with Application to Path Planning and Obstacle Avoidance
}

\author{Morgan Jones%
	\thanks{M. Jones is with the Department of Automatic Control and Systems Engineering,
	The University of Sheffield, Amy Johnson Building, Mappin Street, Sheffield, S1 3JD. e-mail: {\tt \small morgan.jones@sheffield.ac.uk} }
}

\begin{document}

\maketitle
\thispagestyle{plain}
\pagestyle{plain}


\begin{abstract}
Under what circumstances does the ``closeness" of two functions  imply the ``closeness" of their respective sublevel sets? In this paper, we answer this question by showing that if a sequence of functions converges strictly from above/below to a function, $V$, in the $L^\infty$ (or $L^1$) norm then these functions yield a sequence sublevel sets that converge to the sublevel set of $V$ with respect to the Hausdorff metric (or volume metric). Based on these theoretical results we propose Sum-of-Squares (SOS) numerical schemes for the  optimal outer/inner polynomial sublevel set approximation of various sets, including intersections and unions of semialgebraic sets, Minkowski sums, Pontryagin differences and discrete points. We present several numerical examples demonstrating the usefulness of our proposed algorithm 
including approximating sets of discrete points to solve machine learning one-class classification problems and  approximating Minkowski sums to construct C-spaces for computing optimal collision-free paths for Dubin's car.
\end{abstract}

%
\section{Introduction} \label{sec: intro}

The notion of a \textit{set} is a fundamental object of the modern mathematical toolkit, simply defined as a collection of elements. Sets are ubiquitous throughout control theory and are the language we use to represent frequently encountered concepts such as uncertainty~\cite{guan2013uncertainty,zhang2020partition}, regions of attraction~\cite{jones2021converse,henrion2013convex}, reachable states~\cite{trodden2016one,jones2020polynomial}, invariant sets~\cite{korda2014convex}, attractors~\cite{schlosser2022converging,jones2023converse}, admissible inputs~\cite{trodden2020actuation}, feasible states~\cite{scibilia2011feasible}, etc. Even for simple, low-dimensional problems, such sets can contain vast complexities and be numerically challenging to manipulate and analyze. As modern engineering systems become increasingly complex, we can expect that this problem will be exacerbated as the sets we encounter become ever more unwieldy. In this context, we present several fundamental results for the approximation of sets, along with associated implementable numerical schemes based on Sum-of-Squares (SOS) programming.


For some given set $X \subset \R^n$, the goal of this paper is to find an outer or inner set approximation of $X$. Specifically, we would like to compute another set $Y \subset \R^n$ such that $X \subseteq Y$ (outer approximation) or $Y \subseteq X$ (inner approximation) and $Y \approx X$, where the approximation is defined with respect to some set metric. 

Previous attempts at solving this set approximation problem were based on the fact that the volume of an ellipsoid, $\{x \in \R^n: x^\top A x <0 \}$, is proportional to $\det(A^{-1})$.  Then an equivalent optimization problem can be constructed for computing outer ellipsoid approximations by minimizing the convex objective function $-\log \det(A)$~\cite{magnani2005tractable}. This approach of determinant maximization has been heuristically generalized to the problem of computing outer SOS polynomial sublevel approximations~\cite{ahmadi2016geometry,jones2019using}. Alternative heuristic matrix trace maximization schemes have also been proposed~\cite{dabbene2013set}. The seminal work of~\cite{henrion2009approximate} solved the related problem of approximating an integral over a semialgebraic set using the moment approach, showing that the dual to this problem can give outer semialgebraic set approximations. 

 More recently, in the special case of star convex semialgebraic sets, a heuristic approach based on sublevel set scaling has been proposed in~\cite{guthrie2022inner}, demonstrating impressive numerical set approximation results based on SOS programming. An approach based on $L^1$ maximization, similar to some of the schemes proposed in this paper, has been proposed in~\cite{cotorruelo2022sum,dabbene2017simple}. However, the work of ~\cite{cotorruelo2022sum} only considers the specific problem of approximating Pontryagin differences and lacks convergence guarantees and~\cite{dabbene2017simple} only considers the specific problem of approximating semialgebraic in the volume metric. In~\cite{lasserre2015tractable} sets defined by quantifiers are approximated in the volume metric to arbitrary accuracy.  In this paper we extend~\cite{lasserre2015tractable} to the volume metric approximation of more general sets, represented as sublevel sets of integrable functions. Moreover, we propose a new fundamental result for the approximation sets in the Hausdorff metric. Other works that have dealt with the problem of set approximation in the Hausdorff metric include the excellent work of~\cite{jongen2008smoothing} that considered the problem of mollifying feasible constraints and helped to inspire the proof of the Hausdorff set approximations in this paper.


The main contributions of this paper is to show that when $X \subset \R^n$ can be written as a sublevel set of some function $V$ over a compact set (or only finite Lebesgue measurable in the case of volume approximation) $\Lambda \subset \R^n$. the following holds:
\begin{itemize}
	\item Approximating $V$ in the \textbf{$L^\infty$ norm} from above by some function $J$ yields a sublevel set $Y=\{x \in \Lambda : J(x)<\gamma \}$ that provides an arbitrarily accurate inner approximation of $X$ in the \textbf{Hausdorff metric}.
		\item Approximating $V$ in the \textbf{$L^1$ norm} from above by some function $J$ yields a sublevel set $Y=\{x \in \Lambda : J(x)<\gamma \}$ that provides an arbitrarily accurate inner approximation of $X$ in the \textbf{volume metric}.
\end{itemize}
The above sublevel set approximation results provide us with guiding principles for the design of numerical procedures for set approximation. Firstly, given a set, $X$, we must write $X$ as a sublevel set of some function $V$. Fortunately, this requirement is not difficult to satisfy for many of the sets encountered in control theory and in this work we demonstrate this by presenting $V$ explicitly for intersections and unions of semialgebraic sets, Minkowski sums, Pontryagin differences and discrete points. Although not considered in this work, our framework can still be applied when $V$ is not known explicitly, rather it is implicitly defined such as when $V$ is a Lyapunov function~\cite{jones2021converse} or value function~\cite{jones2020polynomial}. After finding $V$, we must find a function $J$ that minimizes $||V-J||$ such that $V(x) \le J(x)$, where $||\cdot||$ is the $L^1$ or $L^\infty$ norm. By introducing auxiliary variables we show that this optimization problem can be lifted to a convex optimization problem and solved by tightening the inequality constraints to SOS constraints for many set approximation problems with Volume/Hausdorff metric convergence guarantees.

To summarize our contribution, to the best of the authors knowledge, only the works of~\cite{dabbene2017simple,lasserre2015tractable} provide numerical schemes with volume metric convergence guarantees for the approximation semialgebraic sets and sets defined by quantifiers respectively. In contrast our work expands on these ideas and provides numerical schemes with convergence guarantees in \textbf{both} the Hausdorff and volume metric for a \textbf{larger class of sets} such as unions of semialgebraic sets, Minkowski sums, Pontryagin differences and discrete points. Moreover, the main results of the paper, stated in Theorems~\ref{thm: uniform convegence mplies H convergence} and~\ref{thm: close in L1 implies close in V norm strict sublevel set}, are independent of SOS programming, allowing exploration of alternative numerical schemes such as those found in~\cite{kamyar2015polynomial}.

\ifthenelse{\boolean{longver}}{}{\vspace{-0.25cm}}
\subsection{Application to Path Planning and Obstacle Avoidance}
\ifthenelse{\boolean{longver}}{}{\vspace{-0.15cm}}
    \ifthenelse{\boolean{longver}}{%
	 The path planning of autonomous systems is the computational problem of finding the sequence of inputs that moves a given object from an initial condition to a target set while avoiding obstacles. Computing optimal paths is a well researched subject, with a broad range of practical uses ranging from the navigation of UAVs~\cite{pham2020complete} to the precise movements of robotic manipulators~\cite{michaux2023can}.
	    }{%
	    
	    }

 Many popular algorithms for solving the path planning problem, such as $A^*$ or Dijkstra, do not inherently account for the size and shape of the object that travels along the path. A common method for overcoming this problem is to transform the workspace of the problem to a Configuration-space (C-space), where in C-space the object is represented by a single point and obstacles are enlarged to account for the loss of size and shape of the object. The enlarged obstacles in C-space are given by the Minkowski sum of the sets representing the object and obstacles. In special convex cases the Minkowski sum can be found as a closed form solution~\cite{ruan2022closed}, however,  unfortunately there is no analytical expression for the Minkowski sum of two general sets. In the absence of an analytical expression for Minkowski sums we rely on numerical approximations. Numerical schemes have previously been proposed for the case of ellipsoid objects and convex objects~\cite{ruan2022efficient} and also for more general sets using SOS and $\log \det$ heuristics~\cite{guthrie2022closed}. In this paper we apply our proposed SOS based numerical scheme  for set approximation, with convergence guarantees, to the problem of approximating Minkowski sums in order to construct C-spaces. 

\ifthenelse{\boolean{longver}}{}{\vspace{-0.15cm}}
\section{Notation} \label{sec: notation}

For $A \subset \R^n$ we denote the indicator function by $\mathds{1}_A : \R^n \to \R$ that is defined as $\mathds{1}_A(x) = \begin{cases}
& 1 \text{ if } x \in A\\
& 0 \text{ otherwise.}
\end{cases}$ For $B \subseteq \R^n$,  $\mu(B):=\int_{\R^n} \mathds{1}_B (x) dx$ is the Lebesgue measure of $B$. We denote the Hausdorff metric (given in Eq.~\eqref{eqn: H metric}) by $D_H$ and the volume metric (given in Eq.~\eqref{eqn: volume metric}) by $D_V$. For two sets $A,B \in \R^n$ we denote $A/B= \{x \in A: x \notin B\}$. For $x \in \R^n$ we denote $||x||_p= \left( \sum_{i =1}^n x_i^p \right)^{\frac{1}{p}}$. For $\eta>0$ and $y \in \R^n$ we denote the set $B_\eta(y)= \{x\in \R^n : ||x-y||_2< \eta\}$. We say $f: \Omega \to \R$ is such that $f \in L^1(\Omega,\R)$ if $||f||_{L^1(\Omega,\R)}:=\int_\Omega |f(x)| dx < \infty$. We define the $L^\infty$ norm as $||f||_{L^\infty(\Omega,\R)}:=\sup_{x \in \Omega} |f(x)|$. We denote the space of polynomials $p: \R^n \to \R$ by $\R[x]$ and polynomials with degree at most $d \in \N$ by $\R_d[x]$. We say $p \in \R_{2d}[x]$ is Sum-of-Squares (SOS) if there exists $p_i \in \R_{d}[x]$ such that $p(x) = \sum_{i=1}^{k} (p_i(x))^2$. We denote $\sum_{SOS}^d$ to be the set of SOS polynomials of at most degree $d \in \N$ and the set of all SOS polynomials as $\sum_{SOS}$. For two sets $A,B \subset \R^n$ we denote the Minkowski sum by $A \oplus B$ (defined in Eq.~\eqref{eq:mink sum}) and Pontryagin difference by $A \ominus B$ (defined in Eq.~\eqref{eq: Pontry diff}).\ifthenelse{\boolean{longver}}{ If $M$ is a subspace of a vector space $X$ we denote equivalence relation $\sim_M$ for $x,y \in X$ by $x \sim_M y$ if $x-y \in M$. We denote quotient space by $X \pmod M:=\{ \{y \in X: y \sim_M x \}: x \in X\}$.    }{
    }

\ifthenelse{\boolean{longver}}{}{\vspace{-0.2cm}}
\section{Sublevel Set Approximation} \label{sec:sublevel set approx}
\ifthenelse{\boolean{longver}}{}{\vspace{-0.1cm}}
Given a set $X \subset \R^n$, the goal of this paper is to solve \ifthenelse{\boolean{longver}}{he following problem,}{}
\ifthenelse{\boolean{longver}}{}{\vspace{-0.1cm}}
\begin{align} \label{opt: set approx}
Y^* \in \arg \inf_{Y \in \mcl C} D_S(X,Y),
\end{align}
where $D_S$ is some set metric providing a notion of how close set $Y^*$ is to $X$ and $\mcl C$ is some set of feasible sets.

The decision variables in Opt.~\eqref{opt: set approx} are sets, which are uncountable objects. This poses a challenge as it is difficult to search and hence optimize over sets effectively. To make such set optimization problems tractable, we need to find ways to parameterize our decision variables. The approach we take to overcome this problem is to consider problems where both our target set, $X$, and our decision variable, $Y$, are sublevel sets of functions, taking the following forms $X=\{x \in \Lambda: V(x)<\gamma\}$ and $Y=\{x \in \Lambda: J(x)< \gamma\}$, where $\Lambda \subset \R^n$ and $\gamma \in \R$. Then, rather than attempting to directly solve the ``geometric" set optimiziation problem in Eq~\eqref{opt: set approx}, we consider the following associated ``algebriac" optimization problem, 
\ifthenelse{\boolean{longver}}{}{\vspace{-0.1cm}}
\begin{align} \label{opt:  function approx}
J^* \in \arg \inf_{J \in \mcl F} D_F(V,J),
\end{align}
 where $D_F$ is some function metric providing a notion of the distance between the functions $V$ and $J$ and $\mcl F$ is some finite dimensional function space.

The decision variables of Opt.~\eqref{opt:  function approx} are functions and hence can be easily optimized over by parametrizing the decision variables using basis functions. We next turn our attention to the question of when does solving Opt.~\eqref{opt:  function approx} yield a close solution to Opt.~\eqref{opt: set approx}? More specifically, if $X=\{x \in \Lambda: V(x)<  \gamma \}$, $Y=\{x \in \Lambda: J(x)<  \gamma \}$ and $J \approx V$ when is it true that $X \approx Y$? In the following subsections we answer this question by showing that, for a given function, $V:\R^n \to \R$, a compact set, $\Lambda \subset \R^n$, and a sequence of functions, $\{J_d\}_{d \in \N}$, the following hold:
\begin{enumerate}[label=(\Alph*)]
	\item  If $J_d(x) \ge V(x)$ for all $x \in \Lambda$ and $J_d \to V$ as $d \to \infty$ in the $L^\infty(\Lambda,\R^n)$ norm then for any $\gamma>0$ we have that $\{x \in \Lambda: J_d(x)< \gamma\} \to \{x \in \Lambda: V(x)< \gamma\}$ with respect to the Hausdorff metric (defined in Eq.~\eqref{eqn: H metric}).
	\item If $V(x) \le J_d(x)$ for all  $x \in \Lambda$  and $J_d \to V$ as $d \to \infty$ in the $L^1(\Lambda,\R^n)$ norm then for any $\gamma>0$ we have that $\{x \in \Lambda: J_d(x) < \gamma\} \to \{x \in \Lambda: V(x) < \gamma\}$ in the volume metric (defined in Eq.~\eqref{eqn: volume metric}). 
\end{enumerate}
\ifthenelse{\boolean{longver}}{}{\vspace{-0.2cm}}
\subsection{Sublevel Set Approximation In The Hausdorff Metric}
\ifthenelse{\boolean{longver}}{}{\vspace{-0.1cm}}
For sets $A,B \subset \R^n$, we denote the Hausdorff metric as $D_H(A,B)$, defining
\begin{equation} \label{eqn: H metric}
\qquad \qquad  D_H(A,B):=\max\{ H(A,B), H(B,A)  \} ,
\end{equation}
where $H(A,B):=\sup_{x \in A} D(x,B)$ and $D(x,B):=\inf_{y \in B}\{\norm{x-y}_2\}$.

\ifthenelse{\boolean{longver}}{
	We now show that uniform convergence of functions yields sublevel approximation in the Hausdorff metric.
	    }{
	    }
\begin{thm} \label{thm: uniform convegence mplies H convergence}
	Consider a compact set $\Lambda \subset \R^n$, a function $V : \Lambda \to \R$, and a family of functions $\{J_d \}_{d \in \N}$ that satisfies the following properties:
\begin{enumerate}
	\item For any $ d \in \N$ we have $V(x) \le J_d(x)$ for all $x \in \Lambda$.
	\item $\lim_{d \to \infty} \sup_{x \in \Lambda} |V(x) -J_d(x)| =0$.
\end{enumerate}
Then for all $\gamma \in \R$ we have that,
\begin{align} \label{sublevel sets close in H}
\lim_{d \to \infty}	D_H \bigg(\{x \in \Lambda: J_d(x)< \gamma\},\{x \in \Lambda: V(x)< \gamma\} \bigg)=0.
\end{align}
\end{thm}
\begin{proof} Throughout this proof we use the following notation $X_d:=\{x \in \Lambda: J_d(x)< \gamma\}$ and $X^*:=\{x \in \Lambda: V(x)< \gamma\}$. 
	
Now, recall from Eq.~\eqref{eqn: H metric} that the Hausdorff metric for two sets  $X_d, X^* \subset \R^n$ is defined as the maximum of two terms, $D_H(X_d,X^*):=\max\{ H(X_d,X^*), H(X^*,X_d)  \} $, where $H(X_d,X^*):=\sup_{x \in X_d} D(x,X^*)$ and $D(x,X^*):=\inf_{y \in X^*}\{\norm{x-y}_2\}$. This naturally leads us to split the remainder of the proof into two parts. In Part~1 of the proof we show that the first term is such that $H(X_d,X^*)=0$ for all $d \in \N$. In Part~2 of the proof we show that for all $\eps>0$ there exists $N \in \N$ such that for all $d>N$ the second term is such that $H(X^*,X_d)< \eps$. Then Parts~1 and~2 of the proof can be used together to show Eq.~\eqref{sublevel sets close in H}, completing the proof.
	
	\underline{\textbf{Part 1 of proof:}} In this part of the proof we show $H( X_d,X^*)=0$ for all $d \in \N$. Since $V(x) \le J_d(x) $ for all $x \in \Lambda$ it follows that $X_d \subseteq X^* $ for any $d \in \N$. Thus for all $d \in \N$ and $x \in  X_d$ we have that $D(x,X^*)=0$. Therefore it clearly follows $H(X_d,X^*)=\sup_{x \in X_d} D(x,X^*)=0$ for all $d \in \N$.
	
	\underline{\textbf{Part 2 of proof:}} In this part of the proof we show that for all $\eps>0$ there exists $N \in \N$ such that for all $d>N$ we have $H(X^*, X_d)< \eps$. For contradiction suppose the negation, that there exists $\delta>0$ such that
	\begin{align*}
	H(X^*, X_d) > \delta \text{ for all } d \in \N.
	\end{align*}
	Then for each $d \in \N$ there exists $x_d \in X^*$ such that 
	\begin{align} \label{pfeq:x_d delta}
	D(x_d,X_d) \ge \delta \text{ for all } d \in \N.
	\end{align}
	
	 Now, $X_d \subseteq \Lambda$ and $\Lambda$ is compact. Therefore the sequence $\{x_d\}_{d \in \N} \subseteq \Lambda$ is bounded. Thus, by the Bolzano Weierstrass Theorem\ifthenelse{\boolean{longver}}{ (Thm.~\ref{thm: Bolzano} found in Appendix~\ref{sec: appendix miscaleneous})}{}, there exists a convergent subsequence $\{y_n\}_{n \in \N} \subseteq \{x_d\}_{d \in \N}$. Let us denote the limit point of $\{y_n\}_{n \in \N}$ by $y^* \in \Lambda$. 
	 
	 Since $\{x_d\}_{d \in \N}$ satisfies Eq.~\eqref{pfeq:x_d delta} and  $\{y_n\}_{n \in \N} \subseteq \{x_d\}_{d \in \N}$ it follows
	 \ifthenelse{\boolean{longver}}{}{\vspace{-0.35cm}}
	 \begin{align} \label{pfeq:y_n delta}
	 D(y_n,X_n) \ge \delta \text{ for all } n \in \N.
	 \end{align}

	 Moreover, because $x_d \in X^*$ for all $d \in \N$ and $\{y_n\}_{n \in \N} \subseteq \{x_d\}_{d \in \N}$ it follows $y_n \in X^*$ for all $n \in \N$. Hence, $\eps_n:=\gamma - V(y_n)  >0$ for all $n \in \N$. 
	 
	 On the other hand, since $\lim_{d \to \infty} \sup_{x \in \Lambda} |V(x) -J_d(x)| =0$ it follows for each $n \in \N$ there exists $N_n \in \N$ such that
	 \begin{align} \label{pfeqn: J_d(x_n) - V(x_n)}
	 J_d(y_n) - V(y_n)< \eps_n \text{ for all } d> N_n.
	 \end{align}
	 Since $\eps_n := \gamma - V(y_n)  >0$, it follows by Eq.~\eqref{pfeqn: J_d(x_n) - V(x_n)} that
	 	 \begin{align*} 
	 J_d(y_n) < \gamma \text{ for all } d> N_n,
	 \end{align*}
	 which implies $y_n \in X_d$ for all $d> N_n$.

	 Now, since $y_n \to y^*$ there exists $N \in \N$ such that $||y_n - y^*||_2< \delta/4$ for all $n >N$. Fixing $n >N$ and selecting $ d>\max\{N,N_n\}$ (as in Eq.~\eqref{pfeqn: J_d(x_n) - V(x_n)} so $y_n \in X_d$ for $ d>N_n$) we have that
	 \ifthenelse{\boolean{longver}}{}{\vspace{-0.4cm}   }
	 \begin{align*}
	 \qquad & D(y_d, X_d)= \inf_{x \in X_d} ||y_d - x ||_2 \\
	 & \qquad \le ||y_d - y_n||_2 \le ||y_d - y^*||_2 + ||y_n - y^*||_2 \le \delta/2,
	 \end{align*}
	 contradicting Eq.~\eqref{pfeq:y_n delta}. Thus it follows that for all $\eps>0$ there exists $N \hspace{-0.05cm}  \in \hspace{-0.05cm} \N$ such that for all $d>N$ we have $H(X^*, X_d)< \eps$. \end{proof}
 \ifthenelse{\boolean{longver}}{
 	\begin{rem} \label{rem: counterexamples}
 		The conditions that $\Lambda$ is compact, sublevel sets are strictly defined, $V(x) \le J_d(x)$ and $\lim_{d \to \infty} \sup_{x \in \Lambda} |V(x) -J_d(x)| =0$ in Thm.~\ref{thm: uniform convegence mplies H convergence} cannot be relaxed. In Section~\ref{subsec: appendix counter examples} we have presented several counterexamples showing that if any of these conditions are relaxed then it is not necessarily true that the sublevel sets of $J_d$ converge to the sublevel set of $V$ in the Hausdorff metric. 
 	\end{rem}
 	    }{
 	    }

\subsection{Sublevel Set Approximation in The Volume Metric} \label{subsec: vol approx}

For Lebesgue measurable sets $A,B \subset \R^n$, we denote the volume metric as $D_V(A,B)$, where
\begin{equation} \label{eqn: volume metric}
\qquad \qquad  D_V(A,B):=\mu\bigg( (A/B) \cup (B/A) \bigg),
\end{equation}
recalling from Sec.~\ref{sec: notation} that we denote $\mu(A)$ as the Lebesgue measure of the set $A \subset \R^n$. 

\begin{thm} \label{thm: close in L1 implies close in V norm strict sublevel set}
	Consider a Lebesgue measurable set $\Lambda \subset \R^n$ with finite Lebesgue measure, a function $V \in L^1(\Lambda, \R)$, and a family of functions $\{J_d \in L^1(\Lambda, \R): d \in \N\}$ that satisfies the following properties:
	\begin{enumerate}
		\item For any $ d \in \N$ we have $V(x) \le J_d(x) $ for all $x \in \Lambda$.
		\item $\lim_{d \to \infty} ||V -J_d||_{L^1(\Lambda, \R)} =0$.
	\end{enumerate}
	Then for all $\gamma \in \R$ we have that
	\begin{align} \label{strict sublevel sets close}
		\lim_{d \to \infty}	D_V \bigg(\{x \in \Lambda: J_d(x) < \gamma \} , \{x \in \Lambda: V(x) < \gamma \} \bigg) =0.
	\end{align} 
\end{thm}
\ifthenelse{\boolean{longver}}{
	\begin{proof}
		Let us denote $\tilde{V}(x)=-V(x)$ and $\tilde{J}_d(x)= - J_d(x)$. It follows that $\tilde{J}_d(x) \le \tilde{V}(x)$ for all $x \in \Lambda$ and $\lim_{d \to \infty} ||\tilde{V} -\tilde{J}_d||_{L^1(\Lambda, \R)} =0$. Therefore, by Prop.~\ref{prop: close in L1 implies close in V norm} {(found in the appendix)} it follows that for any $\gamma \in \R$ we have that,
		\begin{align} \label{pp1}
			\lim_{d \to \infty}	D_V \bigg(\{x \in \Lambda : \tilde{V}(x) \le \gamma\}, \{x \in \Lambda : \tilde{J}_d(x) \le \gamma\} \bigg) =0.
		\end{align}
		
		Now, $\Lambda=\{x \in \Lambda : V(x) < \gamma\} \cup \{x \in \Lambda : V(x) \ge \gamma\} = \{x \in \Lambda : V(x) < \gamma\} \cup \{x \in \Lambda : \tilde{V}(x) \le -\gamma\}$. Therefore
		\begin{align} \label{1}
			\{x \in \Lambda : \tilde{V}(x) \le -\gamma \} = \Lambda / \{x \in \Lambda : V(x) < \gamma \},
		\end{align}
		and by a similar argument
		\begin{align} \label{2}
			\{x \in \Lambda : \tilde{J}_d(x) \le -\gamma \} = \Lambda / \{x \in \Lambda : J_d(x) < \gamma \}.
		\end{align}
		
		Moreover, $V(x) \le J_d(x)$ it follows 
		\begin{align} \label{set con 1}
			\{x \in \Lambda : V(x) < \gamma\} \subseteq \{x \in \Lambda : J_d(x) < \gamma\},
		\end{align}
		and hence 
		\begin{align} \label{set con 2}
			\Lambda /\{x \in \Lambda : J_d(x) < \gamma\} \subseteq \Lambda/\{x \in \Lambda : V(x) < \gamma\},
		\end{align}
		
		Thus, by Lem.~\ref{lem: D_V is related to vol} {(found in the appendix)} and Eqs~\eqref{1} \eqref{2} \eqref{set con 1} \eqref{set con 2} , we have that
		\begin{align} \label{pp2}
			& D_V \bigg(\{x \in \Lambda : V(x) < \gamma\}, \{x \in \Lambda : J_d(x) < \gamma\} \bigg)\\ \nonumber
			& =\mu(\{x \in \Lambda : V(x) < \gamma\})-\mu(\{x \in \Lambda : J_d(x) < \gamma\}) \\ \nonumber
			& =\bigg(\mu(\Lambda)-\mu(\{x \in \Lambda : J_d(x) < \gamma\}) \bigg) \\ \nonumber 
			& \qquad -\bigg( \mu(\Lambda)  - \mu(\{x \in \Lambda : V(x) < \gamma\})\bigg) \\ \nonumber
			&=\mu(\Lambda/ \{x \in \Lambda : {J}_d(x) < \gamma\})-\mu(\Lambda/ \{x \in \Lambda : {V}(x) < \gamma\}) \\ \nonumber 
			& = D_V \bigg( \Lambda/ \{x \in \Lambda : {V}(x) < \gamma\}, \Lambda/ \{x \in \Lambda : {J}_d(x) < \gamma\} \bigg) \\ \nonumber
			& = D_V \bigg(\{x \in \Lambda : \tilde{V}(x) \le -\gamma\}, \{x \in \Lambda : \tilde{J}_d(x) \le -\gamma\} \bigg).
		\end{align}
		Now by Eqs~\eqref{pp1} and~\eqref{pp2} it follows that Eq.~\eqref{strict sublevel sets close} holds.
	\end{proof}
	    }{
	    \begin{proof}
	    	Let us denote $\tilde{V}(x)=-V(x)$ and $\tilde{J}_d(x)= - J_d(x)$. It follows that $\tilde{J}_d(x) \le \tilde{V}(x)$ for all $x \in \Lambda$ and $\lim_{d \to \infty} ||\tilde{V} -\tilde{J}_d||_{L^1(\Lambda, \R)} =0$. Therefore, by Prop.~7 from~\cite{jones2020polynomial} it follows that for any $\gamma \in \R$ we have that,
	    	\begin{align} \label{pp1}
	    		\lim_{d \to \infty}	D_V(\{x \in \Lambda : \tilde{V}(x) \le \gamma\}, \{x \in \Lambda : \tilde{J}_d(x) \le \gamma\} ) =0.
	    	\end{align}
	    	
	Now, $\Lambda=\{x \in \Lambda : V(x) < \gamma\} \cup \{x \in \Lambda : V(x) \ge \gamma\} = \{x \in \Lambda : V(x) < \gamma\} \cup \{x \in \Lambda : \tilde{V}(x) \le -\gamma\}$. Therefore
\begin{align} \label{1}
	\{x \in \Lambda : \tilde{V}(x) \le -\gamma \} = \Lambda / \{x \in \Lambda : V(x) < \gamma \},
\end{align}
and similarly $\{x \in \Lambda : \tilde{J}_d(x) \le -\gamma \} = \Lambda / \{x \in \Lambda : J_d(x) < \gamma \}$. 

Moreover, since $V(x) \le J_d(x)$ it follows 
\begin{align} \label{set con 1}
&	\{x \in \Lambda : V(x) < \gamma\} \subseteq \{x \in \Lambda : J_d(x) < \gamma\},\\ \label{set con 2}
&	\Lambda /\{x \in \Lambda : J_d(x) < \gamma\} \subseteq \Lambda/\{x \in \Lambda : V(x) < \gamma\},
\end{align}


		Thus, by  Lem.~1 from~\cite{jones2019using}, if $B \subseteq A$ are finite Lebesgue measurable then
		$D_V(A,B) =\mu(A/B)= \mu(A)- \mu (B)$, and Eqs~\eqref{1}  \eqref{set con 1} \eqref{set con 2} , we have
\begin{align} \label{pp2}
	& D_V (\{x \in \Lambda : V(x) < \gamma\}, \{x \in \Lambda : J_d(x) < \gamma\} )\\ \nonumber
	& =\mu(\{x \in \Lambda : V(x) < \gamma\})-\mu(\{x \in \Lambda : J_d(x) < \gamma\}) \\ \nonumber
	& =\bigg(\mu(\Lambda)-\mu(\{x \in \Lambda : J_d(x) < \gamma\}) \bigg) \\ \nonumber 
	& \qquad -\bigg( \mu(\Lambda)  - \mu(\{x \in \Lambda : V(x) < \gamma\})\bigg) \\ \nonumber
	&=\mu(\Lambda/ \{x \in \Lambda : {J}_d(x) < \gamma\})-\mu(\Lambda/ \{x \in \Lambda : {V}(x) < \gamma\}) \\ \nonumber 
	& = D_V ( \Lambda/ \{x \in \Lambda : {V}(x) < \gamma\}, \Lambda/ \{x \in \Lambda : {J}_d(x) < \gamma\} ) \\ \nonumber
	& = D_V (\{x \in \Lambda : \tilde{V}(x) \le -\gamma\}, \{x \in \Lambda : \tilde{J}_d(x) \le -\gamma\} ).
\end{align}
	    	Now by Eqs~\eqref{pp1} and~\eqref{pp2} it follows that Eq.~\eqref{strict sublevel sets close} holds.
	    \end{proof}
	    
	    }

\section{Numerical Sublevel Set Approximation} \label{sec: SOS programs for set approx}
 Given  some set $X \subset \R^n$ we would like to approximate, Theorems~\ref{thm: uniform convegence mplies H convergence} and~\ref{thm: close in L1 implies close in V norm strict sublevel set} illuminate the following steps we must take:
\begin{enumerate}
	\item Write the set $X$ as a sublevel set of some function $V$.
	\item Approximate the function $V$ by a uniformly bounding function in the $L^\infty$ norm or $L^1$ norm (depending on whether or not the set approximation is required to be with respect to the Hausdorff or volume metric).
\end{enumerate}
Step 1 is always viable since by \ifthenelse{\boolean{longver}}{Prop.~\ref{prop: exietence of smooth sublevel set} (found in the {appendix})}{Prop.~1 from~\cite{schlosser2021converging}}, it follows that for any compact set $X \subset \R^n$ there always exists a smooth function $V$ such that $X=\{x \in \R^n: V(x) \le 0\}$. Sometimes it is possible to complete Step~2 when $V$ is not known analytically, such as for value functions~\cite{jones2020polynomial} or Lyapunov functions~\cite{jones2021converse}. However, for numerical implementation in this work we will only consider cases where $V$ is known explicitly. Later, in the following subsections, we will present several analytical expressions of $V$ for various classes of sets including intersections and unions of semialgebraic sets, Minkowski sums, Pontryagin differences and discrete points. Before proceeding to these subsections we next briefly discuss the general approach we take to solving Step 2, approximating $V$ by a uniformly bounding function in either the $L^\infty$ or $L^1$ norm.
\paragraph{An optimization problem for $L^\infty$ approximation}
For a given function $V$ let us consider the problem of approximating $V$  uniformly from above in the $L^\infty$ norm, \ifthenelse{\boolean{longver}}{}{\vspace{-0.05cm}}
\begin{align} \label{opt: Linf 1}
	J_d^* \in &\arg \inf_{J_d \in \R_d[x]} \sup_{x \in \Lambda} |V(x)-J_d(x)|\\ \nonumber
&	\text{such that } V(x) \le J_d(x) \text{ for all } x \in \Lambda.
\end{align}

Unfortunately, Opt.~\eqref{opt: Linf 1} is a bi-level optimization problem, having a nested supremum inside the objective function. This makes the problem challenging to solve as our numerical implementation is based on SOS programming (although alternatives to SOS exist~\cite{kamyar2015polynomial}) that requires the coefficients of the polynomial decision variable, $J_d \in \R_d[x]$, to appear linearly in the objective function. Fortunately, it is possible to lift the problem to a convex problem, where all unknown coefficients appear linearly in the constraints and objective function, by introducing extra decision variables in the following way,
\ifthenelse{\boolean{longver}}{}{\vspace{-0.05cm}}
\begin{align} \label{opt: Linf 2}
	J_d^* \in &\arg \inf_{J_d \in \R_d[x],
	P_d \in \R_d[x], \gamma \in \R} \gamma \\ \nonumber
	&	\text{such that } P_d(x) \le V(x) \le J_d(x) \text{ for all } x \in \Lambda,  \\ \nonumber
	& \hspace{1.5cm} J_d(x)-P_d(x)< \gamma \text{ for all } x \in \Lambda.
\end{align}
\paragraph{An optimization problem for $L^1$ approximation} For a given function $V $ let us consider the problem of approximating $V$  uniformly from above in the $L^1$ norm, 
\ifthenelse{\boolean{longver}}{}{\vspace{-0.05cm}}
\begin{align} \label{opt: L1}
	J_d^* \in &\arg \sup_{J_d \in \R_d[x]} \int_\Lambda |J_d(x)-V(x)| dx \\ \nonumber
	&	\text{such that } V(x)  \le J_d(x)  \text{ for all } x \in \Lambda.
\end{align}
Since the constraints of Opt.~\eqref{opt: L1} enforce $V(x) \le J_d(x)  $ it follows that $\int_\Lambda |J_d(x)-V(x)| dx= \int_\Lambda J_d(x) dx- \int_\Lambda V(x) dx$. Since $\int_\Lambda V(x) dx$ is a constant it is equivalent to minimize $ \int_\Lambda J_d(x) dx$ as it is to minimize $\int_\Lambda |J_d(x)-V(x)| dx$. Hence, rather than solving Opt.~\eqref{opt: L1} we solve, 
\ifthenelse{\boolean{longver}}{}{\vspace{-0.05cm}}
\begin{align} \label{opt: L2}
	J_d^* \in &\arg \inf_{J_d \in \R_d[x]} \int_\Lambda J_d(x) dx \\ \nonumber
	&	\text{such that } V(x) \le J_d(x)   \text{ for all } x \in \Lambda.
\end{align}
Interestingly, Opt.~\eqref{opt: L2} shares some similarities to the dual problem in~\cite{henrion2009approximate} for the approximation of integrals over sets. 


It is also useful to consider a similar optimization problem to Opt.~\eqref{opt: L2} that yields outer sublevel set approximations
\ifthenelse{\boolean{longver}}{}{\vspace{-0.05cm}}
\begin{align} \label{opt: L2 2}
	J_d^* \in &\arg \sup_{J_d \in \R_d[x]} \int_\Lambda J(x) dx \\ \nonumber
	&	\text{such that } J_d(x) \le V(x)   \text{ for all } x \in \Lambda.
\end{align}
\paragraph{SOS implementation of $L^\infty$ and $L^1$ approximation}
To solve Opts~\eqref{opt: Linf 2}~\eqref{opt: L2} and~\eqref{opt: L2 2} we tighten the problem by replacing all inequality constraints by SOS constraints.\ifthenelse{\boolean{longver}}{
	This tightening results in a SOS optimization problem that yields a single polynomial sublevel set approximation of several types of sets. Unfortunately, it is non-trivial to tighten Opts~\eqref{opt: Linf 2}~\eqref{opt: L2} and~\eqref{opt: L2 2} because, as we will see, for many set approximation problems the associated $V$ is non-polynomial and hence we cannot directly constrain $ V(x) \le J_d(x)$ or $ P_d(x) \le V(x)$ using SOS.     }{} In the subsequent subsections we consider the following cases,
\begin{itemize}
	\item $V(x)=\max_{1 \le i \le m} g_i(x)$ associated with semialgebraic sets (see Lemma~\ref{lem: writing semialg as a single sublevel}).
	\item $V(x):=\min_{1 \le i \le m} g_i(x)$ associated with unions of semialgebraic sets (see Lemma~\ref{lem: writing union semialg as a single sublevel}).
	\item $V(x):=\inf_{w  \in  \{z \in \Lambda: g_2(z) \le 0 \} } g_1(x-w)$ associated with Minkowski sums (see Lemma~\ref{lem: Mink sum of sublevel sets}).
	\item $V(x):=\sup_{w  \in  \{z \in \Lambda: g_2(z) \le 0 \} } g_1(x+w)$ associated with Pontryagin differences (see Lemma~\ref{lem: Pontryagin difference of sublevel sets}).
	\item $V(x):=1-\mathds{1}_{\{x_i\}_{i=1}^N}(x)$ associated with discrete points (see Lemma~\ref{lem: sublevel set pf discrete points}).
\end{itemize}
For brevity we next only discuss the case of using SOS to constrain $ V(x) \le J_d(x)$ or $ P_d(x) \le V(x)$ when $V(x)=\max_{1 \le i \le m} g_i(x)$ and $g_i \in \R[x]$ for each $1 \le i \le m$. Other forms of $V$ can also be constrained using a similar approach since, with the exception of discrete points, each $V$ involves max/min operators. 

For $V(x)=\max_{1 \le i \le m} g_i(x)$ it is straightforward to constrain $ V(x) \le J_d(x)$ by enforcing $g_i(x) \le J_d(x)$ for $1 \le i \le m$ which tightens to the SOS constraint $J_d-g_i \in \sum_{SOS}$ for $1 \le i \le m$. On the other hand it is slightly more difficult to constrain $P_d(x) \le V(x)$. This is because if we set $P_d(x) \le g_i(x)$ for all $1 \le i \le m$ then $P_d(x) \le \min_{1 \le i \le m} g_i(x) <\max_{1 \le i \le m} g_i(x)=V(x)$ implying that $P_d$ cannot be made arbitrarily close to $V$ since $|V(x)-P_d(x)| \ge \max_{1 \le i \le m} g_i(x)-\min_{1 \le i \le m} g_i(x)>0$.  To enforce $P_d(x) \le V(x)$ in a non-conservative manner we enforce
\begin{align} \label{non conservative constraint}
P_d(x) \le g_i(x) \text{ for } x \in \{y\in \Lambda: g_i(y)\ge g_j(y) \text{ for all } i \ne j  \}.
\end{align}
 \ifthenelse{\boolean{longver}}{
 	 This constraint is now a polynomial inequality over a semialgebraic set and therefore can be readily tightened to an SOS constraint.
 	    }{
 	    \vspace{-1cm}
 	    }
\subsection{Approximation of Semialgebraic Sets} \label{subsec: approx semialg set}
In this subsection we solve the problem of approximating semialgebraic sets, sets of the form $ X=\{x \in \R^n: g_i(x) < 0 \text{ for all } 1 \le i \le m \}$, where $g_i \in \R[x]$ for $1 \le i \le m$. Throughout this section we will assume $X \subset \R^n$ is a compact set and hence WLOG we assume 
 \begin{align}  \label{eq: semialg sets}
X=\{x \in \Lambda: g_i(x) < 0 \text{ for all } 1 \le i \le m \},
 \end{align}
 where $ \Lambda \subset \R^n$ is some sufficiently large compact set and $g_i \in \R[x]$ for $1 \le i \le m$. 
 
 In the following Lemma we show that semialgebraic sets can be written as a sublevel set of a single function.
\begin{lem}[Semialgebraic sets can be written as a single sublevel set] \label{lem: writing semialg as a single sublevel}
	Consider $X=\{x \in \Lambda: g_i(x) < 0 \text{ for all } 1 \le i \le m \}$ then 
\vspace{-0.2cm}	\begin{align*} 
		X=\{x \in \Lambda: V(x)<0 \},
	\end{align*}
where $V(x):=\max_{1 \le i \le m} g_i(x)$.
%
\end{lem}
\ifthenelse{\boolean{longver}}{%
\begin{proof}
Suppose $y \in X$ then $g_i(y) < 0 \text{ for all } 1 \le i \le m$ and hence $V(y)=\max_{1 \le i \le m} g_i(y)<0$. Hence, $y \in \{x \in \Lambda: V(x)<0 \}$ implying that $X \subseteq \{x \in \Lambda: V(x)<0 \}$. On the other hand suppose $y \in \{x \in \Lambda: V(x)<0 \}$. Then $V(y)=\max_{1 \le i \le m} g_i(y)<0$ implying $g_i(y) < 0 \text{ for all } 1 \le i \le m$ and hence $y \in X$. Therefore $\{x \in \Lambda: V(x)<0 \}  \subseteq  X$. Hence $\{x \in \Lambda: V(x)<0 \}  =	  X$. 
\end{proof}
	    }{%
	    \begin{proof}
Follows trivially~\cite{jones2023sublevel}.
	    \end{proof}
    }

It is now clear from Lemma~\ref{lem: writing semialg as a single sublevel} and Theorems~\ref{thm: uniform convegence mplies H convergence} and~\ref{thm: close in L1 implies close in V norm strict sublevel set} that in order to approximate the set in Eq.~\eqref{eq: semialg sets} we must solve Opt.~\eqref{opt: Linf 2} or Opt.~\eqref{opt: L2} for $V(x)=\max_{1 \le i \le m} g_i(x)$. 

WLOG we assume $\Lambda \subset \R^n$ is a ball with sufficiently large radius, that is $\Lambda=\{x \in \R^n: ||x||_2<r\}$. We now propose the following SOS tightening of Opt.~\eqref{opt: Linf 2} for $V(x):=\max_{1 \le i \le m} g_i(x)$,
\ifthenelse{\boolean{longver}}{
\begin{align} \label{opt: SOS Haus intersections}
&	(J_d^*,P_d^*,\gamma_d^*) \in \arg \inf_{J_d \in \R_d[x],
	P_d \in \R_d[x], \gamma \in \R } \gamma \qquad 	\text{ such that } \\ \nonumber
& (g_i(x)  \hspace{-0.05cm} - \hspace{-0.05cm} P_d(x)) \hspace{-0.05cm}  -  \hspace{-0.05cm} s_{i,1}(x)(r^2 \hspace{-0.05cm}  - \hspace{-0.05cm}  ||x||_2^2) \hspace{-0.05cm} \\ \nonumber 
& \hspace{1cm} - \hspace{-0.05cm} \sum_{j=1}^m \hspace{-0.05cm}  s_{i,j,2}(x)(g_i(x) \hspace{-0.05cm}  - \hspace{-0.05cm} g_j(x)) \hspace{-0.05cm}  \in \hspace{-0.05cm}  \sum_{SOS}^d  \text{ for all } 1 \le i \le m, \\ \nonumber
& J_d(x)-g_i(x) - s_{i,3}(x)(r^2 - ||x||_2^2)  \in \sum_{SOS}^d  \text{ for all } 1 \le i \le m,  \\ \nonumber
& \gamma -( J_d(x)-P_d(x)) -s_4(x)(r^2 - ||x||_2^2)  \in \sum_{SOS}^d,\\ \nonumber
& s_{i,1}(x)  \hspace{-0.05cm} \in  \hspace{-0.05cm} \sum_{SOS}^d, \text{ } s_{i,j,2}(x)  \hspace{-0.05cm}  \in  \hspace{-0.05cm} \sum_{SOS}^d,  \text{ } s_{i,3}(x)   \hspace{-0.05cm} \in  \hspace{-0.05cm} \sum_{SOS}^d,\\ \nonumber
& \hspace{3cm}  s_{4}(x)   \hspace{-0.05cm} \in  \hspace{-0.05cm} \sum_{SOS}^d   \text{for all } i,j   \hspace{-0.05cm} \in  \hspace{-0.05cm} \{1,...,m\}.
\end{align}
	    }{
	    \begin{align} \label{opt: SOS Haus intersections}
	    	&	(J_d^*,P_d^*,\gamma_d^*) \in \arg \inf_{J_d \in \R_d[x],
	    		P_d \in \R_d[x], \gamma \in \R,s_{i,1},s_{i,j,2},s_{i,3},s_4 \in \sum_{SOS}^d } \{\gamma\}  \\ \nonumber
	    	& \text{such that: } (g_i(x)  \hspace{-0.05cm} - \hspace{-0.05cm} P_d(x)) \hspace{-0.05cm}  -  \hspace{-0.05cm} s_{i,1}(x)(r^2 \hspace{-0.05cm}  - \hspace{-0.05cm}  ||x||_2^2) \hspace{-0.05cm} \\ \nonumber 
	    	& \hspace{1cm} - \hspace{-0.05cm} \sum_{j=1}^m \hspace{-0.05cm}  s_{i,j,2}(x)(g_i(x) \hspace{-0.05cm}  - \hspace{-0.05cm} g_j(x)) \hspace{-0.05cm}  \in \hspace{-0.05cm}  \sum_{SOS}^d  \text{ for all } 1 \le i \le m, \\ \nonumber
	    	& J_d(x)-g_i(x) - s_{i,3}(x)(r^2 - ||x||_2^2)  \in \sum_{SOS}^d  \text{ for all } 1 \le i \le m,  \\ \nonumber
	    	& \gamma -( J_d(x)-P_d(x)) -s_4(x)(r^2 - ||x||_2^2)  \in \sum_{SOS}^d.
	    \end{align}
}

In a similar way to how we tightened Opt.~\eqref{opt: Linf 2} to get Opt.~\eqref{opt: SOS Haus intersections} we next tighten Opt.~\eqref{opt: L2} for $V(x):=\max_{1 \le i \le m} g_i(x)$ and some $\Lambda \subseteq \{x \in \R^n: ||x||_2<r\}$ to get,
\ifthenelse{\boolean{longver}}{
\begin{align} \label{opt: SOS vol intersection}
	J_d^* \in & \arg \inf_{J_d \in \R_d[x]}  \int_\Lambda J_d(x) dx \qquad \text{ such that }\\ \nonumber
	& J_d(x) - g_i(x) - s_i(x)(r^2 - ||x||_2^2) \in \sum_{SOS}^d   \text{ for } 1 \le i \le m, \\ \nonumber
		&  s_i(x) \in \sum_{SOS}^d \text{ for } 1 \le i \le m.
\end{align}
	   }{
	   \begin{align} \label{opt: SOS vol intersection}
	   	J_d^* \in & \arg \inf_{J_d \in \R_d[x],s_i(x) \in \sum_{SOS}^d}  \int_\Lambda J_d(x) dx \qquad \text{ such that }\\ \nonumber
	   	& J_d(x) - g_i(x) - s_i(x)(r^2 - ||x||_2^2) \in \sum_{SOS}^d   \text{ for } 1 \le i \le m.
	   \end{align}
    }

\subsection{Approximation of Unions of Semialgebraic Sets} \label{subsec: approx of uniion of semialg}
In this subsection we solve the problem of approximating the union of semialgebraic sets, 
\begin{align} \label{set: union strict}
X= \cup_{i=1}^{m_1} \{x \in \Lambda: g_{i,j}(x) < 0 \text{ for all } 1 \le j \le m_2 \},
\end{align}
 where $g_{i,j} \in \R[x]$ for $1 \le i \le m_1$ and $1 \le j \le m_2$. For simplicity we will only consider the case $m_2=1$.

 %
 
  In the following lemma we show that unions of semialgebraic sets can be written as a sublevel set of a single function.
\begin{lem}[Unions of sublevel sets in single sublevel set form]  \label{lem: writing union semialg as a single sublevel}
	Consider $X=\cup_{i=1}^m \{x \in \Lambda: g_i(x) < 0 \}$ then 
	\begin{align*}
		X=\{x \in \Lambda: V(x)<0 \},
	\end{align*}
	where $V(x):=\min_{1 \le i \le m} g_i(x)$.
%
\end{lem}
\begin{proof}
	Follows by a similar argument as Lem.~\ref{lem: writing semialg as a single sublevel}.
\end{proof}
To approximate sets of the form given in Eq.~\eqref{set: union strict} we now propose the following SOS tightening of Opt.~\eqref{opt: Linf 2} for $\Lambda=\{x \in \R^n: ||x||_2<r\}$ and $V(x):=\min_{1 \le i \le m} g_i(x)$, 
\ifthenelse{\boolean{longver}}{
\begin{align} \label{opt: SOS Haus Unions}
	&	(J_d^*,P_d^*,\gamma_d^*) \in \arg \inf_{J_d \in \R_d[x],
		P_d \in \R_d[x], \gamma \in \R} \gamma \qquad 	\text{ such that } \\ \nonumber
	& (g_i(x)  \hspace{-0.05cm} - \hspace{-0.05cm} P_d(x)) \hspace{-0.05cm}  -  \hspace{-0.05cm} s_{i,1}(x)(r^2 \hspace{-0.05cm}  - \hspace{-0.05cm}  ||x||_2^2)   \in \hspace{-0.05cm}  \sum_{SOS}^d \text{ for all } 1 \le i \le m, \\ \nonumber
	& (J_d(x)-g_i(x)) - s_{i,2}(x)(r^2 - ||x||_2^2)  \hspace{-0.05cm}   \\ \nonumber
	&  \hspace{1cm} - \hspace{-0.05cm} \sum_{j=1}^m \hspace{-0.05cm}  s_{i,j,3}(x)(g_j(x) \hspace{-0.05cm}  - \hspace{-0.05cm} g_i(x)) \hspace{-0.05cm}  \in \sum_{SOS}^d  \text{ for all } 1 \le i \le m,  \\ \nonumber
	& \gamma -( J_d(x)-P_d(x)) -s_4(x)(r^2 - ||x||_2^2)  \in \sum_{SOS}^d,\\ \nonumber
	& s_{i,1}(x)  \hspace{-0.05cm} \in  \hspace{-0.05cm} \sum_{SOS}^d, \text{ } s_{i,2}(x)   \hspace{-0.05cm} \in  \hspace{-0.05cm} \sum_{SOS}^d, \text{ } s_{i,j,3}(x)  \hspace{-0.05cm}  \in  \hspace{-0.05cm} \sum_{SOS}^d,  \text{ } ,\\ \nonumber
	&   \hspace{3cm} s_{4}(x) \in  \hspace{-0.05cm} \sum_{SOS}^d   \text{for all } i,j   \hspace{-0.05cm} \in  \hspace{-0.05cm} \{1,...,m\}.
\end{align}
	    }{
	    \begin{align} \label{opt: SOS Haus Unions}
	    	&	(J_d^*,P_d^*,\gamma_d^*) \in \arg \inf_{J_d \in \R_d[x],
	    		P_d \in \R_d[x], \gamma \in \R,s_{i,1},s_{i,j,2},s_{i,3},s_4 \in \sum_{SOS}^d} \{\gamma\}  \\ \nonumber
	    	& \text{such that: } (J_d(x)-g_i(x)) - s_{i,2}(x)(r^2 - ||x||_2^2)  \hspace{-0.05cm}   \\ \nonumber
	    	&  \hspace{1cm} - \hspace{-0.05cm} \sum_{j=1}^m \hspace{-0.05cm}  s_{i,j,3}(x)(g_j(x) \hspace{-0.05cm}  - \hspace{-0.05cm} g_i(x)) \hspace{-0.05cm}  \in \sum_{SOS}^d  \text{ for all } 1 \le i \le m,  \\ \nonumber
	    	& (g_i(x)  \hspace{-0.05cm} - \hspace{-0.05cm} P_d(x)) \hspace{-0.05cm}  -  \hspace{-0.05cm} s_{i,1}(x)(r^2 \hspace{-0.05cm}  - \hspace{-0.05cm}  ||x||_2^2)   \in \hspace{-0.05cm}  \sum_{SOS}^d \text{ for all } 1 \le i \le m, \\ \nonumber
	    	& \gamma -( J_d(x)-P_d(x)) -s_4(x)(r^2 - ||x||_2^2)  \in \sum_{SOS}^d.
	    \end{align}
    }

We next propose the following SOS tightening of Opt.~\eqref{opt: L2 2} for $V(x):=\min_{1 \le i \le m} g_i(x)$ and some $\Lambda \subseteq \{x \in \R^n: ||x||_2<r\}$,
\ifthenelse{\boolean{longver}}{
	\begin{align} \label{opt: SOS vol union}
	J_d^* \in & \arg \sup_{J_d \in \R_d[x]}  \int_\Lambda J_d(x) dx \qquad \text{ such that }\\ \nonumber
	& g_i(x) - J_d(x) - s_i(x)(r^2 - ||x||_2^2) \in \sum_{SOS}^d   \text{ for } 1 \le i \le m, \\ \nonumber 
	&  s_i (x)\in \sum_{SOS}^d \text{ for } 1 \le i \le m.
\end{align}    }{
\begin{align} \label{opt: SOS vol union}
	J_d^* \in & \arg \sup_{J_d \in \R_d[x],s_i \in \sum_{SOS}^d}  \int_\Lambda J_d(x) dx \qquad \text{ such that }\\ \nonumber
	& g_i(x) - J_d(x) - s_i(x)(r^2 - ||x||_2^2) \in \sum_{SOS}^d   \text{ for } 1 \le i \le m.
\end{align}
    }

\ifthenelse{\boolean{longver}}{
 Note, solving Opt.~\eqref{opt: SOS Haus Unions} results in both a lower ($P_d$) and upper ($J_d$) approximation of $V(x):=\min_{1 \le i \le m} g_i(x)$, whereas, solving Opt.~\eqref{opt: SOS vol union} only results in a lower approximation. Unfortunately, Theorems~\ref{thm: uniform convegence mplies H convergence} and~\ref{thm: close in L1 implies close in V norm strict sublevel set} only apply for upper function approximations and hence we are unable to use these theorems to show Opt.~\eqref{opt: SOS vol union} yields an arbitrary accurate sublevel set approximation of the set given in Eq.~\eqref{set: union strict}. We could construct a similar SOS problem to Opt.~\eqref{opt: SOS vol union} that results in a upper approximation in a non-conservative way using a similar argument as in Eq.~\eqref{non conservative constraint} but this would result in more SOS decision variables. Alternatively, we will later show that Opt.~\eqref{opt: SOS vol union} yields a sublevel approximation to the non-strict/relaxed version the set given of Eq.~\eqref{set: union strict}. That is, Opt.~\eqref{opt: SOS vol union} can be used to approximate the following set,
 \begin{align} \label{set: union non-strict}
 	X= \cup_{i=1}^{m} \{x \in \Lambda: g_{i}(x) \le 0\},
 \end{align}
 where $g_i \in \R[x]$ for $1 \le i \le m$.
}{
We will also consider the problem of approximating unions of non-strict semialegebriac sets such as
 \begin{align} \label{set: union non-strict}
	X= \cup_{i=1}^{m} \{x \in \Lambda: g_{i}(x) \le 0\},
\end{align}
   }

%

\subsection{Approximation of Minkowski Sums} \label{subsec: Minkowski Sum}
\begin{defn}[Minkowski Sum]
Given two sets $A,B \subset \R^n$ their Minkowski sum is defined as
\begin{equation} \label{eq:mink sum}
	A \oplus B = \{a+b \in \R^n : a \in A \text{ and } b \in B \}.
\end{equation}
\end{defn}
We next consider the problem of approximating the set $X=A \oplus B$, where $A,B \subset \R^n$ can be written as sublevel sets. That is we consider the problem of approximating the following,
\begin{align} \label{set: mink sum}
	X=\{x \in \Lambda: g_1(x) \le 0 \} \oplus \{x \in \Lambda: g_2(x) \le 0 \},
\end{align}
where $g_1,g_2 \in C^\infty$. This is non-restrictive since \ifthenelse{\boolean{longver}}{Prop.~\ref{prop: exietence of smooth sublevel set} (found in the {appendix})}{Prop.~1 from~\cite{schlosser2021converging}} shows any compact set can be written as a sublevel set and Lemmas~\ref{lem: writing semialg as a single sublevel} and~\ref{lem: writing union semialg as a single sublevel} give this sublevel set in an analytical form for intersections or unions of semialgebraic sets. 

In the following lemma we show sets satisfying Eq.~\eqref{set: mink sum} can be written as the sublevel set of a single function. 
\begin{lem}[Minkowski sums in sublevel set form] \label{lem: Mink sum of sublevel sets}
	Consider $X:=\{x \in \Lambda: g_1(x) \le 0 \} \oplus \{x \in \Lambda: g_2(x) \le 0 \}  $, where $\Lambda \subset \R^n$ is a compact set and $g_1,g_2 \in C(\Lambda,\R)$, then 
	\begin{align*}
		X=\{x \in \Lambda: V(x) \le 0 \},
	\end{align*}
	where $V(x):=\inf_{w  \in  \{z \in \Lambda: g_2(z) \le 0 \} } g_1(x-w)$.
%
\end{lem}
\ifthenelse{\boolean{longver}}{
	\begin{proof}
		Suppose $x \in X$ then there exists $z_1 \in \{x \in \Lambda: g_1(x) \le 0 \} $ and $z_2 \in \{x \in \Lambda: g_2(x) \le 0 \}  $ such that $x =z_1+z_2$. Hence,
		\begin{align*}
			V(x)= \inf_{w  \in  \{z \in \Lambda: g_2(z) \le 0 \} } g_1(x-w) \le g_1(x-z_2)= g_1(z_1) \le 0,
		\end{align*}
		since $z_1 \in \{x \in \Lambda: g_1(x) \le 0 \}  $. Therefore, $x \in \{y \in \Lambda: V(y) \le 0 \}$ implying $X \subseteq  \{x \in \Lambda: V(x) \le 0 \}$.
		
		On the other hand suppose $x \in \{y \in \Lambda: V(y) \le 0 \}$. Note that by Lem.~\ref{lem: Continuous functions have compact sublevel sets} {(found in the appendix)} it follows that $\{z \in \Lambda: g_2(z) \le 0 \}$ is compact. Now, by the extreme value theorem, continuous functions attain their extreme values over compact sets implying that there exists $z_2 \in \arg \inf_{w  \in  \{z \in \Lambda: g_2(z) \le 0 \} } g_1(x-w)$ such that $g_1(x-z_2)=V(x)$. Let $z_1=x-z_2$. Because $V(x) \le 0$ it follows that $g_1(z_1)= g_1(x-z_2)=V(x) \le 0$. Hence, $z_1 \in \{x \in \Lambda: g_1(x)  \le 0 \} $. Therefore it follows that $x \in X$ and hence $X \subset \{y \in \Lambda: V(y) \le 0 \}$. Thus $X=\{x \in \Lambda: V(x)  \le 0 \}$.
	\end{proof}
	    }{
	    \begin{proof}
	    	Follows by a similar argument as Lemma~\ref{lem: writing semialg as a single sublevel} with full technical details given in the extended Arxiv paper~\cite{jones2023sublevel}.
	    \end{proof}
	    }
For simplicity we next consider the special case of approximating $A \oplus B$ where $A$ is a union of sublevel sets and $B$ is a semialgebraic set. That is we consider the problem of approximating the following set,
\begin{align} \label{set: mink union and int}
	X:=& \cup_{i=1}^{m_1} \{x \in \Lambda: g_{1,i}(x) \le 0 \} \\ \nonumber
	 & \qquad \qquad  \oplus \{x \in \Lambda: g_{2,j}(x) \le 0 \text{ for } 1 \le j \le m_2 \},
\end{align}
where $g_{1,i}, g_{2,j} \in \R[x]$ for all $1 \le i \le m_1$ and $1 \le j \le m_2$. Other cases where one or both of $A$ or $B$ is a union or intersection of semialgebraic sets can be handled by a similar methodology. Now, by similar arguments to Lemmas~\ref{lem: writing semialg as a single sublevel} and~\ref{lem: writing union semialg as a single sublevel} and the use of Lemma~\ref{lem: Mink sum of sublevel sets} it is clear that $X$, given in Eq.~\eqref{set: mink union and int}, is such that $X=\{x \in \Lambda: V(x) \le 0 \}$ where
\begin{align} \label{fun: V for mink}
V(x)=\inf_{w  \in  \{z \in \Lambda: \max_{1 \le j \le m_2} g_{2,j}(z) \le 0 \} } \min_{1 \le i \le m_1} g_{1,i}(x-w).
\end{align}

Unlike in Subsections~\ref{subsec: approx semialg set} and~\ref{subsec: approx of uniion of semialg}, in this subsection we do not provide a SOS tightening of Opt.~\eqref{opt: Linf 2} due to the complexity of using SOS to enforce the constraint $V(x)\le J_d(x)$ when $V$ is given in Eq.~\eqref{fun: V for mink}\ifthenelse{\boolean{longver}}{
	    }{ (this constraint can be found in the extended Arxiv version~\cite{jones2023converse}).} \ifthenelse{\boolean{longver}}{
	Indeed, we could increase the problems dimension and enforce,
\begin{align*}
& J_d(x) - g_{1,i}(x-w) - s_{i,0}(x,u,w)(g_{1,i}(x-u) - g_{1,i}(x-w)) \\
& -\sum_{j=1}^{m_2} s_{i,j,1}(x,w,u)(g_{1,j}(x-w) -g_{1,i}(x-w))  \\
& + \sum_{j=1}^{m_2} s_{i,j,2}(x,w,u)g_{2,j}(w)+ \sum_{j=1}^{m_2} s_{i,j,3}(x,w,u)g_{2,j}(u) \\
& - s_{i,4}(x,w,u)(r^2 - ||(x,u,w)||_2^2) \in \sum_{SOS} \text{ for } 1 \le i \le m_1,\\
& s_{i,0},s_{i,j,1}, s_{i,j,2}, s_{i,j,3}, s_{i,4} \in \sum_{SOS} \text{ for } 1 \le i \le m_1 \text{ \& } 1 \le j\le m_2
\end{align*} 
However, this approach results in a very large SOS optimization problem.}{
   }Fortunately, it is relatively easy to enforce the constraint $J_d(x) \le V(x)$. This is the only constraint that is required to be tightened in Opt.~\eqref{opt: L2 2}. Thus, to approximate sets of the form given in Eq.~\eqref{set: mink union and int} we now propose the following SOS tightening of Opt.~\eqref{opt: L2 2} for {$V(x):=\inf_{w  \in  \{z \in \Lambda: \max_{1 \le j \le m_2} g_{2,j}(z) \le 0 \} } \min_{1 \le i \le m_1} g_{1,i}(x-w)$} and some $\Lambda \subseteq \{x \in \R^n: ||x||_2<r\}$,
\ifthenelse{\boolean{longver}}{
\begin{align} \label{opt: SOS vol mink sum}
	J_d^* \in & \arg \sup_{J_d \in \R_d[x]}  \int_\Lambda J_d(x) dx \qquad \text{ such that }\\ \nonumber
	& g_{1,i}(x-w) - J_d(x) - s_{i,1}(x,w)(r^2 - ||(x,w)||_2^2)\\ \nonumber
	& \qquad +\sum_{j=1}^{m_2} s_{i,j,2}(x,w)g_{2,j}(w)  \in \sum_{SOS}^d   \text{ for } 1 \le i \le m_1, \\ \nonumber
		&  s_{i,1}(x) \in \sum_{SOS}^d \text{ for } 1 \le i \le m_1, \\ \nonumber
	& s_{i,j,2}(x) \in \sum_{SOS}^d \text{ for } 1 \le i \le m_1 \text{ and } 1 \le j \le m_2.
\end{align}	    }{
\begin{align} \label{opt: SOS vol mink sum}
	J_d^* \in & \arg \sup_{J_d \in \R_d[x],s_{i,1},s_{i,j,2}\in \sum_{SOS}^d}  \int_\Lambda J_d(x) dx \qquad \text{ such that }\\ \nonumber
	& g_{1,i}(x-w) - J_d(x) - s_{i,1}(x,w)(r^2 - ||(x,w)||_2^2)\\ \nonumber
	& \qquad +\sum_{j=1}^{m_2} s_{i,j,2}(x,w)g_{2,j}(w)  \in \sum_{SOS}^d   \text{ for } 1 \le i \le m_1.
\end{align}
    }

\subsection{Approximation of Pontryagin Differences} \label{subsec: Pontryagin}
\begin{defn}[Pontryagin Difference]
	Given two sets $A,B \subset \R^n$ their Pontryagin difference is defined as
	\begin{equation} \label{eq: Pontry diff}
		A \ominus B = \{a \in A: a+b \in A \text{ for all } b \in B   \}.
	\end{equation}
\end{defn}
We next show that the Pontryagin difference of two sublevel sets can be written as a sublevel set of a single function.
\begin{lem}[Pontryagin difference in sublevel set form] \label{lem: Pontryagin difference of sublevel sets}
	Consider $X:=\{x \in \Lambda: g_1(x) < 0 \} \ominus \{x \in \Lambda: g_2(x) \le 0 \}  $, where $\Lambda \subset \R^n$ is a compact set and $g_1,g_2 \in C(\Lambda,\R)$, then 
	\begin{align} \label{pfeq: strict sub mink }
		X=\{x \in \Lambda: V(x) < 0 \},
	\end{align}
	where $V(x):=\sup_{w  \in  \{z \in \Lambda: g_2(z) \le 0 \} } g_1(x+w)$.
	%
\end{lem}
\ifthenelse{\boolean{longver}}{
	\begin{proof}
		We first show Eq.~\eqref{pfeq: strict sub mink }. Suppose $x \in X$ then it must follow that $g_1(x+w)<0$ for all $w \in \{z \in \Lambda: g_2(z) \le 0 \} $. Note that by Lem.~\ref{lem: Continuous functions have compact sublevel sets} {(found in the appendix)} it follows that $\{z \in \Lambda: g_2(z) \le 0 \}$ is compact. Now, by the extreme value theorem continuous functions attain their extreme values over compact sets implying that there exists $w^*  \in  \{z \in \Lambda: g_2(z) \le 0 \} $ such that $g_1(x+w^*)=\sup_{w  \in  \{z \in \Lambda: g_2(z) \le 0 \} } g_1(x+w)$. Hence $V(x)=\sup_{w  \in  \{z \in \Lambda: g_2(z) \le 0 \} } g_1(x+w)=g_1(x+w^*)<0$ implying that $x \in \{z \in \Lambda: V(z) < 0 \}$ and hence $X \subseteq \{x \in \Lambda: V(x) < 0 \}$.
		
		On the other hand if $x \in \{z \in \Lambda: V(z) < 0 \}$ it follows that $\sup_{w  \in  \{z \in \Lambda: g_2(z) \le 0 \} } g_1(x+w)<0$. Hence, $g_1(x+w) \le\sup_{w  \in  \{z \in \Lambda: g_2(z) \le 0 \} } g_1(x+w) < 0$ for all $w  \in  \{z \in \Lambda: g_2(z) \le 0 \} $ implying $x+w \in \{x \in \Lambda: g_1(x) <	 0 \} $ for all $w  \in  \{z \in \Lambda: g_2(z) \le 0 \} $ . Thus it follows that $x \in X$.
		%
	\end{proof}
	    }{
	     \begin{proof}
	    	Follows by a similar argument as Lemma~\ref{lem: writing semialg as a single sublevel} with full technical details given in the extended Arxiv paper~\cite{jones2023sublevel}.
	    \end{proof}
	    }

For simplicity we consider the special case of approximating $A \ominus B$ where both $A$ and $B$ are semialgebraic sets. Other cases where one or both of $A$ or $B$ is a union or intersection of semialgebraic sets can be handled by a similar methodology. That is we consider the problem of approximating the following set,
\begin{align} \label{set: pont int and int}
	X:=&  \{x \in \Lambda: g_{1,i}(x) < 0 \text{ for } 1 \le i \le m_1 \} \\ \nonumber
	& \qquad \qquad  \ominus \{x \in \Lambda: g_{2,j}(x) \le 0 \text{ for } 1 \le j \le m_2 \},
\end{align} 
where $g_{1,i}, g_{2,j} \in \R[x]$ for all $1 \le i \le m_1$ and $1 \le j \le m_2$. Other cases where one or both of $A$ or $B$ is a union or intersection of semialgebraic sets can be handled by a similar methodology. Now, by Lemmas~\ref{lem: writing semialg as a single sublevel} and~\ref{lem: Pontryagin difference of sublevel sets} it is clear that $X$, given in Eq.~\eqref{set: pont int and int}, is such that $X=\{x \in \Lambda: V(x) < 0 \}$ where
\begin{align} \label{fun: V for pont} 
	V(x)=\sup_{w  \in  \{z \in \Lambda: \max_{1 \le j \le m_2} g_{2,j}(z) \le 0 \} } \max_{1 \le i \le m_1} g_{1,i}(x+w).
\end{align}
Similarly to Subsection~\ref{subsec: Minkowski Sum}, we do not provide a tightening of the Opt.~\eqref{opt: Linf 2} due to the complexity of using SOS to enforce the constraint $J_d(x) \le V(x)$ when $V$ is given in Eq.~\eqref{fun: V for pont} 

To approximate sets given in Eq.~\eqref{set: pont int and int} we now propose the following SOS tightening of Opt.~\eqref{opt: L2} for {$V(x):=\sup_{w  \in  \{z \in \Lambda: \max_{1 \le j \le m_2} g_{2,j}(z) \le 0 \} } \max_{1 \le i \le m_1} g_{1,i}(x+w)$ and some $\Lambda \subseteq \{x \in \R^n: ||x||_2<r\}$,
\ifthenelse{\boolean{longver}}{
		\begin{align} \label{opt: SOS vol Ponty diff}
		J_d^* \in & \arg \inf_{J_d \in \R_d[x]}  \int_\Lambda J_d(x) dx \qquad \text{ such that }\\ \nonumber
		& J_d(x)- g_{1,i}(x+w)  - s_{i,1}(x,w)(r^2 - ||(x,w)||_2^2)\\ \nonumber
		& \qquad +\sum_{j=1}^{m_2} s_{i,j,2}(x,w)g_{2,j}(w)  \in \sum_{SOS}^d   \text{ for } 1 \le i \le m_1, \\ \nonumber
		&  s_{i,1} \in \sum_{SOS}^d \text{ for } 1 \le i \le m_1, \\ \nonumber
		& s_{i,j,2} \in \sum_{SOS}^d \text{ for } 1 \le i \le m_1 \text{ and } 1 \le j \le m_2.
	\end{align}
    }{
    \begin{align} \label{opt: SOS vol Ponty diff}
    	J_d^* \in & \arg \inf_{J_d \in \R_d[x],s_{i,1},s_{i,j,2} \in \sum_{SOS}^d }  \int_\Lambda J_d(x) dx \qquad \text{ such that }\\ \nonumber
    	& J_d(x)- g_{1,i}(x+w)  - s_{i,1}(x,w)(r^2 - ||(x,w)||_2^2)\\ \nonumber
    	& \qquad +\sum_{j=1}^{m_2} s_{i,j,2}(x,w)g_{2,j}(w)  \in \sum_{SOS}^d   \text{ for } 1 \le i \le m_1.
    \end{align}
    }

\subsection{Approximation of Discrete Points} \label{subsec: discrete points}
In this section we consider the problem of approximating a set of discrete points, $\{x_i\}_{i=1}^N \subset \R^n$. \ifthenelse{\boolean{longver}}{
	We first note that it is possible to write a set of discrete points as a sublevel set of a polynomial without any computation by considering the function $V(x)= \prod_{i=1}^N (x_i - x)^2$. However, this polynomial function has a very large degree, double the number of data points. We next  show that sets of discrete points can be written as a sublevel set of a simpler function. This will allow us to derive an SOS optimization problem to approximate this simple function by a polynomial with relatively small degree, yielding a low degree polynomial sublevel approximation of $\{x_i\}_{i=1}^N \subset \R^n$.
	    }{
	  }
\begin{lem} \label{lem: sublevel set pf discrete points}
Consider $X=\{x_i\}_{i=1}^N \subset \Lambda$. Then
\begin{align}
	X=\{x \in \Lambda: V(x) \le 0 \},
\end{align}
where $V(x)=1-\mathds{1}_{\{x_i\}_{i=1}^N}(x)$.
\end{lem}
\begin{proof}
	Follows trivially.
\end{proof}
\ifthenelse{\boolean{longver}}{
	We do not provide an SOS tightening for Opt.~\eqref{opt: Linf 2} when $V(x)=1-\mathds{1}_{\{x_i\}_{i=1}^N}(x)$ because in order for us to prove that the resulting polynomial sublevel set approximation can be arbitrarily small with respect to the Hausdorff metric we rely on Thm.~\ref{thm: uniform convegence mplies H convergence}. Thm.~\ref{thm: uniform convegence mplies H convergence} requires that our polynomial approximations of $V$ converge in the $L^\infty$ norm. Unfortunately, this $V$ is not continuous so it is not possible to approximate it uniformly using smooth functions like polynomials. Fortunately, $V$ is an integrable function so it is possible to approximate it by a polynomial in the $L^1$ norm using the following SOS tightening of Opt.~\eqref{opt: L2 2} for some $\{x \in \R^n: ||x||_2<r \}\supseteq \Lambda \supseteq \{x_i\}_{i=1}^N$,
	    }{
	    We do not provide an SOS tightening for Opt.~\eqref{opt: Linf 2} when $V(x)=1-\mathds{1}_{\{x_i\}_{i=1}^N}(x)$ since we cannot approximate this discontinuous function in the $L^\infty$ norm by a polynomial. Fortunately, $V$ is an integrable function so it is possible to approximate it by a polynomial in the $L^1$ norm using the following SOS tightening of Opt.~\eqref{opt: L2 2} for some $\{x \in \R^n: ||x||_2<r \}\supseteq \Lambda \supseteq \{x_i\}_{i=1}^N$,
	    }
\ifthenelse{\boolean{longver}}{
\begin{align} \label{opt: SOS discrete points}
	J_d^* \in & \arg \sup_{J_d \in \R_d[x]}  \int_\Lambda J_d(x) dx \qquad \text{ such that }\\ \nonumber
	& J_d(x_i) <0  \text{ for } 1 \le i \le N, \\ \nonumber 
	&  1-J_d(x) - s_0(x)(r^2 -||x||_2^2) \in \sum_{SOS}^d, \quad s_0(x) \in \sum_{SOS}^d.
\end{align}
	    }{
	    \begin{align} \label{opt: SOS discrete points}
	    	J_d^* \in & \arg \sup_{J_d \in \R_d[x],s_0 \in \sum_{SOS}^d}  \int_\Lambda J_d(x) dx \qquad \text{ such that }\\ \nonumber
	    	& J_d(x_i) <0  \text{ for } 1 \le i \le N, \\ \nonumber 
	    	&  1-J_d(x) - s_0(x)(r^2 -||x||_2^2) \in \sum_{SOS}^d.
	    \end{align}
    }

\section{Convergence of Our Proposed SOS Programs} \label{sec: Convergence of Our Proposed SOS Programs}
\ifthenelse{\boolean{longver}}{
In the previous sections we proposed several SOS optimization problems to approximate various functions, $V$, whose sublevel sets yield several important classes of sets \ifthenelse{\boolean{longver}}{
	(intersections and unions of semialgebraic sets, Minkowski sums, Pontryagin differences and discrete points)
	    }{
	    }. We use Theorems~\ref{thm: uniform convegence mplies H convergence} and~\ref{thm: close in L1 implies close in V norm strict sublevel set} to prove that the sequences of solutions to our SOS problems produce sequences of sublevel sets that each converge to the associated target set in the Hausdorff or volume metric. \ifthenelse{\boolean{longver}}{
	         We begin this section with showing convergence in the Hausdorff metric. 
	    }{
    }
\begin{thm} \label{thm: SOS H convergence}
	It holds that,
		\begin{align} \label{eqn: convergence of SOS to intersection}
		& \lim_{d \to \infty}	D_H \bigg(\{x \in \Lambda: J_d^*(x)<0\} , X \bigg) =0,\\ \label{eq: J_d subset of X}
	&\qquad \qquad 	\{x \in \Lambda: J_d^*(x)<0\} \subseteq X,
	\end{align}
where $\Lambda:=\{x \in \R^n: ||x||_2 \le r\}$ and the set $X$ and the sequence of functions $\{J_d^*\}_{d \in \N} \subset \R[x]$ satisfy either of the following:
	\begin{itemize}
		\item  $X$ is a semialgebraic set, given by Eq.~\eqref{eq: semialg sets}, and $\{J_d^*\}_{d \in \N} \subset \R[x]$ solve the family of $d$-degree SOS optimization problems given in Eq.~\eqref{opt: SOS Haus intersections}.
		\item  $X$ is a union of semialgebraic sets given by Eq.~\eqref{set: union strict} and $\{J_d^*\}_{d \in \N} \subset \R[x]$ solve the family of $d$-degree SOS optimization problems given in Eq.~\eqref{opt: SOS Haus Unions}.
	\end{itemize}
\end{thm}
%
\begin{proof}
	\textbf{Convergence to semialgebraic sets:}	We first prove the first case where $X$ is a semialgebraic set, given by Eq.~\eqref{eq: semialg sets}, and $\{J_d^*\}_{d \in \N} \subset \R[x]$ solve the family of $d$-degree SOS optimization problems given in Eq.~\eqref{opt: SOS Haus intersections}.
	
Now, it follows that if we show that
	\begin{align} \label{pfeq: uniform bound}
		J_d^*(x) \ge V(x) \text{ for all } x \in \Lambda,\\ \label{pfeq: uniform convergence}
		\lim_{d \to \infty}||J_d^*-V||_{L^\infty(\Lambda,\R)}=0,
	\end{align}
	where $V(x):=\max_{1 \le i \le m} g_i(x)$, then Eqs~\eqref{eqn: convergence of SOS to intersection} and~\eqref{eq: J_d subset of X} hold and the proof is complete. This is because Lem.~\ref{lem: writing semialg as a single sublevel} shows that $X=\{x \in \Lambda: V(x)<0\}$ and therefore if Eq.~\eqref{pfeq: uniform bound} holds then clearly Eq.~\eqref{eq: J_d subset of X} must hold. Moreover, if Eqs~\eqref{pfeq: uniform bound} and~\eqref{pfeq: uniform convergence} hold then by Theorem~\ref{thm: uniform convegence mplies H convergence} it must follow that $\lim_{d \to \infty}	D_H (\{x \in \Lambda: J_d^*(x)<0\} , \{x \in \Lambda: V(x)<0\} )$. It is then clear Eq.~\eqref{eqn: convergence of SOS to intersection} holds.
	
	We first show Eq.~\eqref{pfeq: uniform bound}. Because $J_d^*$ solves Opt.~\eqref{opt: SOS Haus intersections} it follows that it must be feasible to Opt.~\eqref{opt: SOS Haus intersections} and hence satisfies each constraint. In particular there must exist SOS variables $s_{2,i} \in \sum_{SOS}^d$ such that 
	\begin{align} \label{pfeq: constraint}
		J_d^*(x) - g_i(x) -s_{2,i}(x)(r^2 - ||x||_2^2) \in  \sum_{SOS}^d \text{ for } 1 \le i \le m.
	\end{align}
	Hence, it follows by Eq.~\eqref{pfeq: constraint} and the positivity of SOS polynomials that $J_d^*(x) \ge g_i(x)$ for all $x \in \Lambda$ and $1\le i \le m$. Therefore $J_d^*(x) \ge \max_{1 \le i \le m} g_i(x)=V(x)$ and thus Eq.~\eqref{pfeq: uniform bound} holds. 
	
	We next show that Eq.~\eqref{pfeq: uniform convergence} holds. Let $\eps>0$, by Cor.~\ref{cor: existence of poly close to max of polys}, found in Appendix~\ref{subsec: appendix poly approx}, there exists $H_1,H_2 \in \R[x]$ such that, 
	\begin{align} \label{pfeq:close to V}
		& \sup_{ x \in \Lambda} |V(x) - H_j(x) | < \frac{\eps}{4} \text{ for } j \in \{1,2\}, \\ \nonumber
		& H_1(x) > g_i(x) \text{ for all } x \in B_r(0) \text{ and } 1 \le i \le m,\\ \nonumber
		& H_2(x) < g_i(x) \text{ for all } x \in B_r(0) \cap \bar{\mcl Y_i }\text{ and } 1 \le i \le m,
	\end{align}
	where  $\bar{\mcl Y_i } :=\{y \in \Lambda: g_i(y) \ge g_j(y) \text{ for } 1 \le j \le m  \}$. 
	
	Since $H_1(x)-g_i(x) >0$ for all $x \in B_r(0)$ and $1 \le i \le m$ and also since $B_r(0)$ is compact semialgebraic set it follows by Thm.~\ref{thm: Psatz} that there exists SOS polynomials $\{s_{2,i}\}_{i \in \{1,..,m\} } \subset \sum_{SOS}$ and $s_{0,1	} \in \sum_{SOS}$ such that,
	\begin{align} \label{pfeq:s1}
		H_1(x)-g_i(x)-s_{2,i}(x)(r^2 - ||x||_2^2)=s_{0,1	}(x).
	\end{align}
	Moreover, by a similar argument, since $g_i(x)- H_2(x) >0$ for all $x \in  B_r(0) \cap \bar{\mcl Y_i } $ and $1 \le i \le m$ and also since $B_r(0)\cap \bar{\mcl Y_i }$ is compact semialgebraic set it follows by Thm.~\ref{thm: Psatz} that there exists SOS polynomials $\{s_{1,i,j}\}_{(i,j) \in \{1,..,m\} \times \{1,..,m\} } \subset \sum_{SOS}$, $\{s_{1,i}\}_{i \in \{1,..,m\} } \subset \sum_{SOS}$  and $s_{0,2	} \in \sum_{SOS}$ such that,
	\begin{align} \label{pfeq:s2}
		g_i(x)-H_2(x) -s_{1,i}(x) & (r^2 - ||x||_2^2) \\ \nonumber
		& -\sum_{j=1}^m s_{1,i,j}(x)(g_i(x)-g_j(x))     =s_{0,2}(x).
	\end{align}
	Now clearly, $H_2(x)< \max_{1 \le i \le m} g_i(x) =V(x)< H_1(x)$ for all $x \in  B_r(0)$ implying
	\begin{align*}
		& \gamma_H -(H_2(x) - H_1(x)) \ge (H_2(x) - H_1(x)) \\
		& = (H_2(x)-V(x))+ (V(x) - H_1(x))>0 \text{ for all } x \in B_r(0),
	\end{align*}
	where $\gamma_H :=2 \sup_{x \in B_r(0)} \{H_2(x) - H_1(x)\} $. Therefore by Thm.~\ref{thm: Psatz} there exists SOS polynomials $s_{3}   \in \sum_{SOS}$ and and $s_{0,3	} \in \sum_{SOS}$ such that,
	\begin{align} \label{pfeq:s3}
		\gamma_H -(H_2(x) - H_1(x)) -s_3(x)(r^2 - ||x||_2^2) =s_{0,3}(x).
	\end{align}
	Furthermore, by Eq.~\eqref{pfeq:close to V} it follows that
	\begin{align} \label{pfeeq:eps}
		\gamma_H\le  2\sup_{x \in B_r(0)}|H_2(x)-V(x)| +   2\sup_{x \in B_r(0)}|V(x)-H_1(x)| < \eps.
	\end{align}
	
	Now, let $d_0:=\max \bigg\{ deg(H_1),deg(H_2), \max_{i,j \in \{1,...,m\} }\{s_{1,i,j} \},$ $\max_{i \in \{1,...,m\}} \{s_{2,i}\}, s_3, \max_{i \in \{1,2,3\}} \{s_{0,i}\}    \bigg\}$. It then follows by Eqs~\eqref{pfeq:s1}, \eqref{pfeq:s2} and~\eqref{pfeq:s3} that $(H_2,H_1,\gamma_H)$ is a feasible solution to Opt.~\eqref{opt: SOS Haus intersections} for degree $d \ge d_0$. Hence, the value of the objective resulting from the optimal solutions, $(J_{d}^*,P_{d}^*, \gamma_{d}^*)$ for $d \ge d_0$, will be less than or equal to the value of the objective function resulting from the feasible solution $(H_2,H_1,\gamma_H)$. Therefore by Eq.~\eqref{pfeeq:eps} we have that,
	\begin{align} \label{pdeq:gam}
		\gamma^*_{d} \le \gamma_H<\eps \text{ for all } d \ge d_0.
	\end{align}
	
	Now, since $(J_{d}^*,P_{d}^*, \gamma_{d}^*)$ is the optimal solution to Opt.~\eqref{opt: SOS Haus intersections} it satisfies all of the constraints of Opt.~\eqref{opt: SOS Haus intersections}. Since SOS polynomials are non-negative it then follows
	\begin{align*}
		J^*_{d}(x)&  \ge g_i(x) \text{ for all } x \in B_r(0) \text{ and } 1 \le i \le m, \\
		P^*_{d}(x) & \le g_i(x) \text{ for all } x \in B_r(0) \cap \bar{\mcl Y_i }\text{ and } 1 \le i \le m,\\
		J^*_{d}(x)- P^*_{d}(x) &  \le \gamma_{d}^* \text{ for all } x \in B_r(0),
	\end{align*}
	implying $P^*_{d}(x) \le V(x) \le J^*_{d}(x)$ for all $B_r(0)$ and hence 
	\begin{align} \label{pfeq: less than gam}
		|J^*_{d}(x) - V(x)| & = J^*_{d}(x) - V(x) \le J^*_{d}(x) - P^*_{d}(x)\\ \nonumber
		& \le \gamma_{d}^* \text{ for all } x \in B_r(0).
	\end{align}
	Therefore, by Eqs~\eqref{pdeq:gam} and~\eqref{pfeq: less than gam} it follows that
	\begin{align} \label{pfeq:1}
		\sup_{x \in B_r(0)}	|J^*_{d}(x) - V(x)|\le \gamma_{d}^*< \eps \text{ for } d \ge d_0.
	\end{align}
	Now, $\eps>0$ was arbitrarily chosen and thus Eq.~\eqref{pfeq:1} shows Eq.~\eqref{pfeq: uniform convergence}, completing the proof.
	
		\textbf{Convergence to unions of semialgebraic sets:} Consider the case when $X$ is given by Eq.~\eqref{set: union strict} and $\{J_d^*\}_{d \in \N} \subset \R[x]$ is given in Eq.~\eqref{opt: SOS Haus Unions}. Then  Eqs~\eqref{eqn: convergence of SOS to intersection} and~\eqref{eq: J_d subset of X} hold by a similar argument to the proof of convergence to to semialgebraic sets, where instead of showing $\lim_{d \to \infty}||J_d^*-\max_{1 \le i \le m} g_i||_{L^\infty(\Lambda,\R)}=0$ (as in Eq.~\eqref{pfeq: uniform convergence}) we show $\lim_{d \to \infty}||J_d^*-\min_{1 \le i \le m} g_i||_{L^\infty(\Lambda,\R)}=0$.
\end{proof}

Note, the solutions, $\{J_d\}_{d \in \N}$, to Opts~\eqref{opt: SOS Haus intersections} and~\eqref{opt: SOS Haus Unions} provide inner sublevel set approximations as shown in Eq.~\eqref{eq: J_d subset of X}. However, it is still possible to construct an outer sublevel set approximation using $P_d^*$ since in both cases $P_d^*(x) \le V(x)$ and $\lim_{d \to \infty} ||P_d- V||_{L^\infty(\Lambda,\R)}=0$. Unfortunately, as shown in Counterexample~\ref{cex: lower bound does not tend in Haus} (found in the appendix), it is not necessarily true that $\{x \in \Lambda: P^*_d(x)<\gamma \} \to \{x \in \Lambda: V(x)<\gamma \}$ in the Hausdorff metric. Fortunately, it is fairly straight forward to show $\{x \in \Lambda: P^*_d(x) \le \gamma \} \to \{x \in \Lambda: V(x) \le \gamma \}$ in the volume metric using Prop.~\ref{prop: close in L1 implies close in V norm}.


We next use Theorem~\ref{thm: close in L1 implies close in V norm strict sublevel set} to show that the SOS programs we proposed in Section~\ref{sec: SOS programs for set approx}, that approximate in the $L^1$ norm from above, yield polynomial sublevel sets that converge in the volume metric to semialgebraic sets or Pontryagin differences.
\begin{thm} \label{thm: SOS converge in V to strict sublevel sets}
	It holds that,
	\begin{align} \label{eqn: convergence of L1 SOS to intersection}
	& \lim_{d \to \infty}	D_V \bigg(\{x \in \Lambda: J_d^*(x) < 0\} , X \bigg) =0, \\ \label{eq: J_d subset of X 2}
	&\qquad \qquad 	\{x \in \Lambda: J_d^*(x)<0\} \subseteq X,
\end{align}
	where $\Lambda:=\{x \in \R^n: ||x||_2 \le r\}$ and the set $X$ and the sequence of solutions $\{J_d^*\}_{d \in \N} \subset \R[x]$ satisfy either of the following:
	\begin{itemize}
		\item $X$ is a semialgebraic set, given by Eq.~\eqref{eq: semialg sets}, and $\{J_d^*\}_{d \in \N} \subset \R[x]$ is given in Eq.~\eqref{opt: SOS vol intersection}.
		\item $X$ is a Pontryagin difference, given by Eq.~\eqref{set: pont int and int}, and $\{J_d^*\}_{d \in \N} \subset \R[x]$ is given in Eq.~\eqref{opt: SOS vol Ponty diff}.
	\end{itemize}
\end{thm}
\begin{proof}
\textbf{Convergence to semialgebraic sets:}	We first show Eqs~\eqref{eqn: convergence of L1 SOS to intersection} and~\eqref{eq: J_d subset of X 2} hols in the case $X$ is a semialgebraic set, given by Eq.~\eqref{eq: semialg sets}, and $\{J_d^*\}_{d \in \N} \subset \R[x]$ is given in Eq.~\eqref{opt: SOS vol intersection}.

 Now, it follows that if we show that
	\begin{align} \label{pfeq: uniform bound 2}
		J_d^*(x) \ge V(x) \text{ for all } x \in \Lambda,\\ \label{pfeq: L1 convergence}
		\lim_{d \to \infty}||J_d^*-V||_{L^1(\Lambda,\R)}=0,
	\end{align}
	where $V(x):=\max_{1 \le i \le m} g_i(x)$, then Eqs~\eqref{eqn: convergence of L1 SOS to intersection} and~\eqref{eq: J_d subset of X 2} hold and the proof is complete. This is because by Lem.~\ref{lem: writing semialg as a single sublevel} we have that $X=\{x \in \Lambda: V(x)<0\}$. Then if Eqs~\eqref{pfeq: uniform bound 2} and~\eqref{pfeq: L1 convergence} hold then clearly Eq.~\eqref{eq: J_d subset of X 2} is satisfied and, by Thm.~\ref{thm: close in L1 implies close in V norm strict sublevel set}, it must follow that $\lim_{d \to \infty}	D_V(\{x \in \Lambda: J_d^*(x)<0\} , \{x \in \Lambda: V(x)<0\} )$. It is then clear Eq.~\eqref{eqn: convergence of L1 SOS to intersection} holds.
	
	We first show Eq.~\eqref{pfeq: uniform bound 2}. Because $J_d^*$ solves Opt.~\eqref{opt: SOS vol intersection} it follows that it must be feasible to Opt.~\eqref{opt: SOS vol intersection} and hence satisfies each constraint. In particular there must exist SOS variables $s_{i} \in \sum_{SOS}^d$ such that 
	\begin{align} \label{pfeq: constraint 2}
		J_d^*(x) - g_i(x) -s_{i}(x)(r^2 - ||x||_2^2) \in  \sum_{SOS}^d \text{ for } 1 \le i \le m.
	\end{align}
	Hence, it follows by Eq.~\eqref{pfeq: constraint 2} and the positivity of SOS polynomials that $J_d^*(x) \ge g_i(x)$ for all $x \in \Lambda$ and $1\le i \le m$. Therefore $J_d^*(x) \ge \max_{1 \le i \le m} g_i(x)=V(x)$ and thus Eq.~\eqref{pfeq: uniform bound 2} holds. 
	
	We next show Eq.~\eqref{pfeq: L1 convergence}. Consider the sequence of functions, $\{G_d^*\}_{d \in \N} \subset \R[x]$, that solve the family of $d$-degree SOS optimization problems given in Eq.~\eqref{opt: SOS Haus intersections}. This sequence of functions is also feasible to Opt.~\eqref{opt: SOS vol intersection}. Therefore the value of the objective function of Opt.~\eqref{opt: SOS vol intersection} resulting from $G_d^*$ will be greater than the optimal solution $J_d^*$. That is
	\begin{align*}
		\int_\Lambda J_d^*(x) dx \le  \int_\Lambda G_d^*(x) dx.
	\end{align*}
	Moreover, by Eq.~\eqref{pfeq: uniform convergence}, in the proof of Thm.~\ref{thm: SOS H convergence}, we have that 
	\begin{align*}
		\lim_{d \to \infty}||G_d^*-V||_{L^\infty(B_r(0),\R)}=0.
	\end{align*}
	Hence,
	\begin{align*}
		&	\lim_{d \to \infty} \int_{\Lambda} |V(x) - J_d^*(x)| dx = \lim_{d \to \infty} \int_{\Lambda} V(x) - J_d^*(x) dx\\
		& \le \lim_{d \to \infty} \int_{\Lambda} V(x) - G_d^*(x) dx \le \lim_{d \to \infty} \mu(\Lambda) ||V-G_d||_{L^1(B_r(0),\R)}=0.
	\end{align*}
	Therefore, Eq.~\eqref{pfeq: L1 convergence} is shown completing the proof.

					\textbf{Convergence to Pontryagin differences:} 	Consider the case where $X$ is given by Eq.~\eqref{set: pont int and int}, and $\{J_d^*\}_{d \in \N} \subset \R[x]$ is given in Eq.~\eqref{opt: SOS vol Ponty diff}. Note, in Eq.~\eqref{fun: V for pont} that it was shown that $X=\{x \in \Lambda: V(x) < 0 \}$, where \[V(x):=\sup_{w  \in  \{z \in \Lambda: \max_{1 \le i \le m_2} g_{2,i}(z) \le 0 \} } \max_{1 \le i \le m_1} g_{1,i}(x-w).\]
					
					The rest of the proof follows by a similar argument to the convergence to semialgebraic sets proof, where we use Cor.~\ref{cor: function approx ponty diff} {(found in the appendix)} to approximate $V$ uniformly from above by a polynomial. We then show that this polynomial approximation is feasible to Opt.~\eqref{opt: SOS vol Ponty diff} for sufficiently large $d \in \N$ using Theorem~\ref{thm: Psatz} {(found in the appendix)}. Because this feasible solution to Opt.~\eqref{opt: SOS vol Ponty diff} is arbitrarily close to $V$, it can then be shown that the true solution of Opt.~\eqref{opt: SOS vol Ponty diff} approximates $V$ from above with arbitrary precision with respect to the $L^1$ norm. Hence, Eq.~\eqref{eqn: convergence of L1 SOS to intersection} follows by Thm.~\ref{thm: close in L1 implies close in V norm strict sublevel set}. It is also clear that Eq.~\eqref{eq: J_d subset of X 2} since $J_d^*(x) \ge V(x)$.
\end{proof}

Thm.~\ref{thm: SOS converge in V to strict sublevel sets} shows that the SOS programs we proposed in Section~\ref{sec: SOS programs for set approx}, that approximate in the $L^1$ norm from \textbf{above}, converge to various sets defined by \textbf{strict} sublevel sets. We next show that the SOS programs we proposed in Section~\ref{sec: SOS programs for set approx}, that approximate in the $L^1$ norm from \textbf{below}, yield polynomial sublevel sets that converge in the volume metric to unions of semialgebraic sets or Minkowski sums or discrete points defined by  \textbf{non-strict} sublevel sets.
\begin{thm}
It holds that,
\begin{align} \label{eqn: convergence of L1 SOS non-strict}
	& \lim_{d \to \infty}	D_V \bigg(\{x \in \Lambda: J_d^*(x)  \le 0\} , X \bigg) =0,\\ \label{eq: X subset of J_d}
	&\qquad \qquad 	 X \subseteq \{x \in \Lambda: J_d^*(x) \le 0\},
\end{align}
where $\Lambda:=\{x \in \R^n: ||x||_2 \le r\}$ and the set $X$ and the sequence of solutions $\{J_d^*\}_{d \in \N} \subset \R[x]$ satisfy either of the following:
\begin{itemize}
	\item $X$ is a union of semialgebraic sets, given by Eq.~\eqref{set: union non-strict}, and $\{J_d^*\}_{d \in \N} \subset \R[x]$ is given in Eq.~\eqref{opt: SOS vol union}.
	\item $X$ is a Minkowski sum, given by Eq.~\eqref{set: mink union and int}, and $\{J_d^*\}_{d \in \N} \subset \R[x]$ is given in Eq.~\eqref{opt: SOS vol mink sum}.
	\item $X$ is a set of discrete points, given by $X=\{x_i\}_{i=1}^N\subset \Lambda$, and $\{J_d^*\}_{d \in \N} \subset \R[x]$ is given in Eq.~\eqref{opt: SOS discrete points}.
\end{itemize}
\end{thm}
\begin{proof}
		\textbf{Convergence to unions of semialgebraic sets:}  Consider the case when $X$ is given by Eq.~\eqref{set: union non-strict} and $\{J_d^*\}_{d \in \N} \subset \R[x]$ is given in Eq.~\eqref{opt: SOS vol union}.	By a similar argument to the proof of Thm.~\ref{thm: SOS converge in V to strict sublevel sets} that showed Eqs~\eqref{pfeq: uniform bound 2} and~\eqref{pfeq: L1 convergence} hold, it is possible to use Prop.~\ref{prop: existence of poly close to min of polys} to show that
	\begin{align*}
		V(x) \ge J_d^*(x)  \text{ for all } x \in \Lambda,\\ 
		\lim_{d \to \infty}||J_d^*-V||_{L^1(\Lambda,\R)}=0,
	\end{align*}
	for $V(x):=\min_{1 \le i \le m} g_i(x)$. Then by Prop.~\ref{prop: close in L1 implies close in V norm} it follows that $\lim_{d \to \infty} D_V(\{x \in \Lambda: V(x) \le 0\}, \{x \in \Lambda: J_d(x) \le 0 \} )=0$. By a similar argument to Lem.~\ref{lem: writing union semialg as a single sublevel} it follows that $X=\{x \in \Lambda: V(x) \le 0\}$, where $X$ is given by Eq.~\eqref{set: union non-strict}.  Therefore it is clear Eq.~\eqref{eqn: convergence of L1 SOS non-strict} holds. Moreover, since $V(x) \ge J_d^*(x)$ it is clear that Eq.~\eqref{eq: X subset of J_d} holds.
	
	\textbf{Convergence to Minkowski sums:} Consider the case where $X$ is given by Eq.~\eqref{set: mink union and int}, and $\{J_d^*\}_{d \in \N} \subset \R[x]$ is given in Eq.~\eqref{opt: SOS vol mink sum}. Note in Eq.~\eqref{fun: V for mink} that it was shown that $X=\{x \in \Lambda: V(x) \le 0 \}$, where \[V(x):=\inf_{w  \in  \{z \in \Lambda: \max_{1 \le i \le m_2} g_{2,i}(z) \le 0 \} } \min_{1 \le i \le m_1} g_{1,i}(x-w).\]
	
	The rest of the proof follows by a similar argument to the convergence to unions of semialgebraic sets proof, where we use Lem.~\ref{lem: function approx mink sum} {(found in the appendix)} to approximate $V$ uniformly from below by a polynomial. We then show this polynomial approximation is feasible to Opt.~\eqref{opt: SOS vol mink sum} for sufficiently large $d \in \N$ using Theorem~\ref{thm: Psatz} {(found in the appendix)}. Because this feasible solution to Opt.~\eqref{opt: SOS vol mink sum} is arbitrarily close to $V$, it can then be shown that the true solution of Opt.~\eqref{opt: SOS vol mink sum} approximates $V$ from below with arbitrary precision with respect to the $L^1$ norm. Hence, Eq.~\eqref{eqn: convergence of L1 SOS non-strict} follows by Prop.~\ref{prop: close in L1 implies close in V norm}. Moreover, since $V(x) \ge J_d^*(x)$ it is clear that Eq.~\eqref{eq: X subset of J_d} holds.
	
		\textbf{Convergence to Discrete Points:} Consider the case where $X=\{x_i\}_{i=1}^N$, and $\{J_d^*\}_{d \in \N} \subset \R[x]$ is given in Eq.~\eqref{opt: SOS discrete points}. Note in Lem.~\ref{lem: sublevel set pf discrete points} it was shown that $X=\{x \in \Lambda: V(x) \le 0 \}$, where \[V(x):=1 - \mathds{1}_{\{x_i\}_{i=1}^N}(x).\]
		
		The rest of this proof again follows by a similar argument to proof of convergence to unions of semialgebraic sets or Minkowski sums. This time we use Prop.~\ref{prop: L1 poly approx of indicator} to approximate $V$ uniformly from below by a polynomial. We then show that this polynomial is feasible to Opt.~\eqref{opt: SOS discrete points} for sufficiently large enough degree using Theorem~\ref{thm: Psatz}. Because this feasible solution to Opt.~\eqref{opt: SOS discrete points} is arbitrarily close to $V$, it can then be shown that the true solution of Opt.~\eqref{opt: SOS discrete points} approximates $V$ from below with arbitrary precision with respect to the $L^1$ norm. Hence, Eq.~\eqref{eqn: convergence of L1 SOS non-strict} follows by Prop.~\ref{prop: close in L1 implies close in V norm}. Moreover, since $V(x) \ge J_d^*(x)$ it is clear that Eq.~\eqref{eq: X subset of J_d} holds.
\end{proof}
}{
\begin{thm} \label{thm: SOS H convergence}
	It holds that,
	\begin{align} \label{eqn: convergence of SOS to intersection}
		& \lim_{d \to \infty}	D_H \bigg(\{x \in \Lambda: J_d^*(x)<0\} , X \bigg) =0,\\ \label{eq: J_d subset of X}
		&\qquad \qquad 	\{x \in \Lambda: J_d^*(x)<0\} \subseteq X,
	\end{align}
	where $\Lambda:=\{x \in \R^n: ||x||_2 \le r\}$ and the set $X$ and the sequence of functions $\{J_d^*\}_{d \in \N} \subset \R[x]$ satisfy either of the following:
	\begin{itemize}
		\item  $X$ is a semialgebraic set, given by Eq.~\eqref{eq: semialg sets}, and $\{J_d^*\}_{d \in \N} \subset \R[x]$ solve the family of $d$-degree SOS optimization problems given in Eq.~\eqref{opt: SOS Haus intersections}.
		\item  $X$ is a union of semialgebraic sets given by Eq.~\eqref{set: union strict} and $\{J_d^*\}_{d \in \N} \subset \R[x]$ solve the family of $d$-degree SOS optimization problems given in Eq.~\eqref{opt: SOS Haus Unions}.
	\end{itemize}
\begin{align} \label{eqn: convergence of L1 SOS to intersection}
	& \lim_{d \to \infty}	D_V \bigg(\{x \in \Lambda: J_d^*(x) < 0\} , X \bigg) =0, \\ \label{eq: J_d subset of X 2}
	&\qquad \qquad 	\{x \in \Lambda: J_d^*(x)<0\} \subseteq X,
\end{align}
where $\Lambda:=\{x \in \R^n: ||x||_2 \le r\}$ and the set $X$ and the sequence of solutions $\{J_d^*\}_{d \in \N} \subset \R[x]$ satisfy either of the following:
\begin{itemize}
	\item $X$ is a semialgebraic set, given by Eq.~\eqref{eq: semialg sets}, and $\{J_d^*\}_{d \in \N} \subset \R[x]$ is given in Eq.~\eqref{opt: SOS vol intersection}.
	\item $X$ is a Pontryagin difference, given by Eq.~\eqref{set: pont int and int}, and $\{J_d^*\}_{d \in \N} \subset \R[x]$ is given in Eq.~\eqref{opt: SOS vol Ponty diff}.
\end{itemize}
\begin{align} \label{eqn: convergence of L1 SOS non-strict}
	& \lim_{d \to \infty}	D_V \bigg(\{x \in \Lambda: J_d^*(x)  \le 0\} , X \bigg) =0,\\ \label{eq: X subset of J_d}
	&\qquad \qquad 	 X \subseteq \{x \in \Lambda: J_d^*(x) \le 0\},
\end{align}
where $\Lambda:=\{x \in \R^n: ||x||_2 \le r\}$ and the set $X$ and the sequence of solutions $\{J_d^*\}_{d \in \N} \subset \R[x]$ satisfy either of the following:
\begin{itemize}
	\item $X$ is a union of semialgebraic sets, given by Eq.~\eqref{set: union non-strict}, and $\{J_d^*\}_{d \in \N} \subset \R[x]$ is given in Eq.~\eqref{opt: SOS vol union}.
	\item $X$ is a Minkowski sum, given by Eq.~\eqref{set: mink union and int}, and $\{J_d^*\}_{d \in \N} \subset \R[x]$ is given in Eq.~\eqref{opt: SOS vol mink sum}.
	\item $X$ is a set of discrete points, given by $X=\{x_i\}_{i=1}^N\subset \Lambda$, and $\{J_d^*\}_{d \in \N} \subset \R[x] $ is given in Eq.~\eqref{opt: SOS discrete points}.
\end{itemize}
\end{thm}
\begin{proof}
	Convergence is shown using the following steps:
\begin{enumerate}
	\item Use the Weierstrass theorem and a small correction term (or mollification in the case of the discontinuous function $V(x):=1-\mathds{1}_{\{x_i\}_{i=1}^N}(x)$) to show that there exists a polynomial that is arbitrary close to $V$ and also satisfies the inequality constraints to the associated SOS problem strictly. Since only standard arguments are required we have left the full technical details for this step to the extended Arxiv paper~\cite{jones2023sublevel}.
	\item Apply Putinar’s Positivstellesatz~\cite{putinar1993positive} to show that this polynomial is feasible to the associated SOS problem for some sufficiently large degree.
	\item Deduce that the true solution to the SOS problem must tend to $V$ as $d \to \infty$ since by optimality it must be closer than the feasible polynomial that we found in the previous steps of the proof.
	\item Apply Thm.~\ref{thm: uniform convegence mplies H convergence} or Thm.~\ref{thm: close in L1 implies close in V norm strict sublevel set} to show that the sublevel sets of $V$ and the optimal solution to the SOS program converge in some set metric as $d \to \infty$.
\end{enumerate}
\end{proof}}
\section{Numerical Examples}
\ifthenelse{\boolean{longver}}{
	In this section we use the SOS programs we proposed in Sec.~\ref{sec: SOS programs for set approx} to approximate various sets. We first approximate unions of semialgebraic sets in the Hausdorff and volume metric. We then approximate Minkowski sums and Pontryagin differences in the volume metric. We finish the section with three numerical examples that have practical motivations, the approximation of Region of Attractions (ROAs) and attractor sets of nonlinear systems via the outer bounding sets of discrete points, and the Minkowski sum of obstacles and vehicle shape sets to construct a C-space in which collision free path planning can be computed in. All SOS programs are solved using Yalmip~\cite{lofberg2004yalmip} with SDP solver Mosek~\cite{aps2019mosek}.
	    }{
All SOS programs are solved using Yalmip~\cite{lofberg2004yalmip} with SDP solver Mosek~\cite{aps2019mosek}. Further implementation details and more numerical examples can be found in~\cite{jones2023sublevel}.
	    }

%


\begin{ex}[Approximation of unions of semialgebraic sets] \label{ex: Approximation of unions of semialgebraic sets}
	Consider the union of semialgebraic sets
	\begin{align} \label{numerical union of semialg}
		& \cup_{i=1}^3 X_i= \cup_{i=1}^3 \{x \in \R^2: g_i(x)<0 \} \\ \nonumber
		\text{where } \\ \nonumber
	g_1(x) & = x_1^2 + x_2^2 -0.075 , \\ \nonumber
	g_2(x) & = (x_1-0.15)^2 + (x_2-0.15)^2 -0.025, \\ \nonumber
	g_3(x) & = (x_1+0.25)^2 + (x_2 + 0.25)^2  -0.001.
	\end{align}
	In Fig.~\ref{fig: union of semialg} we have plotted $\cup_{i=1}^3 X_i$ as the green region along with several approximations of the form $\{x\in \R^n: P(x)<0 \}$. For our Hausdorff approximation $P=P_d$ and $P_d$ is found by solving SOS Opt.~\eqref{opt: SOS Haus Unions}. For our volume approximation $P=J_d$ and $J_d$ is found by solving SOS Opt.~\eqref{opt: SOS vol union}. Fig.~\ref{fig: union of semialg} shows several outer approximations for $r=0.57$, $\Lambda=[-0.4,0.4]^2$ and $d=2,10$ $\&$ $18$. As expected from the convergence proofs given in Sec.~\ref{sec: Convergence of Our Proposed SOS Programs}, in both cases we see as the degree increases our set approximation improves. Interestingly the degree 18 Hausdorff approximation seems to give a better representation of the topology than the degree 18 volume approximation by showing some disconnection between $X_2$ and $X_3$. 
	
	\begin{figure} 
		\centering
		\subfloat[\scriptsize Hausdorff approximation]
		{
			\includegraphics[width=0.46 \linewidth,  trim = {0cm 2cm 3cm 0cm},, clip]{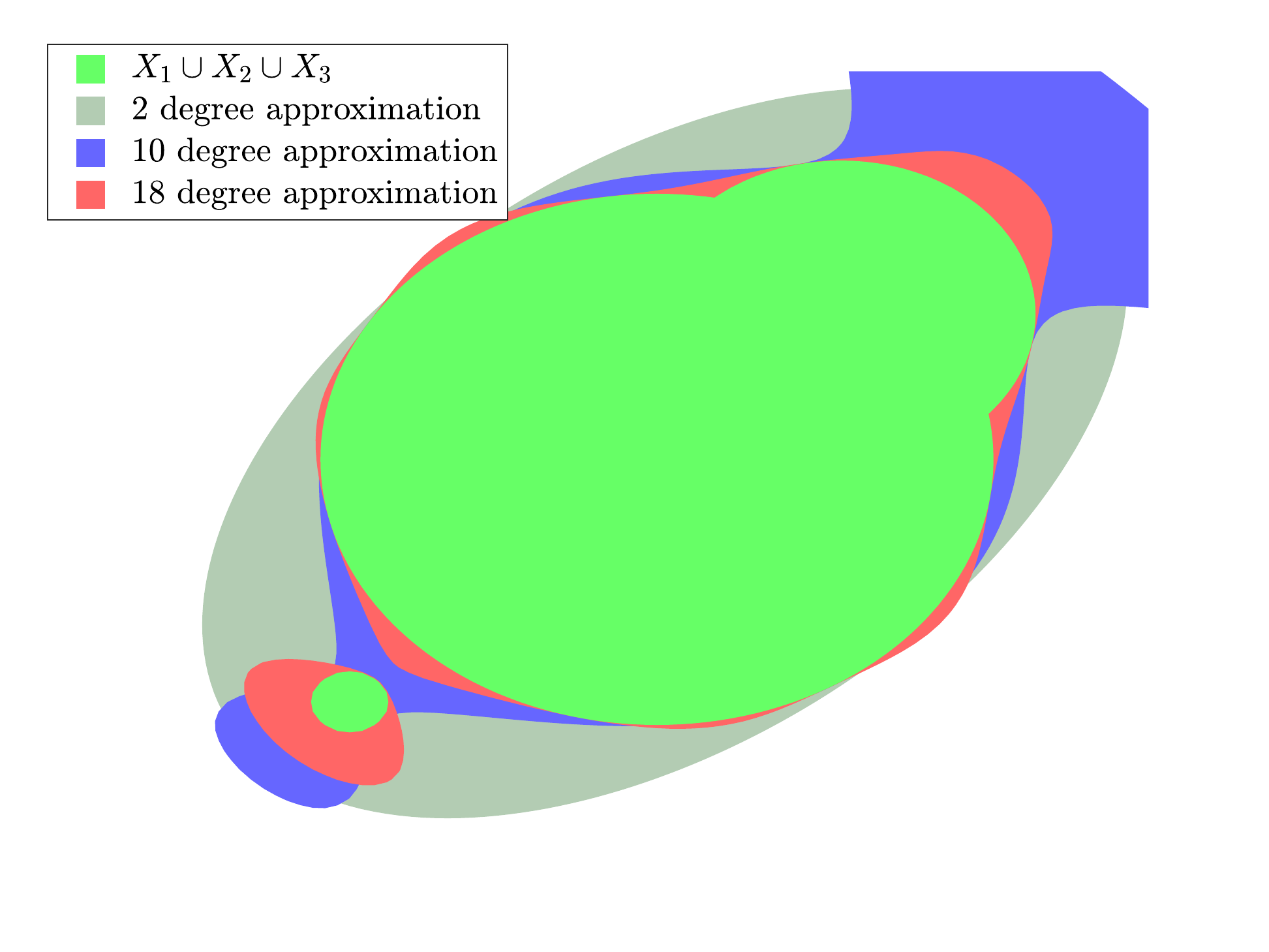}
			\label{fig:Haus union}
		}
		\subfloat[ \scriptsize Volume approximation]
		{
			\includegraphics[width=0.46 \linewidth, trim = {0cm 2cm 3cm 0cm}, clip]{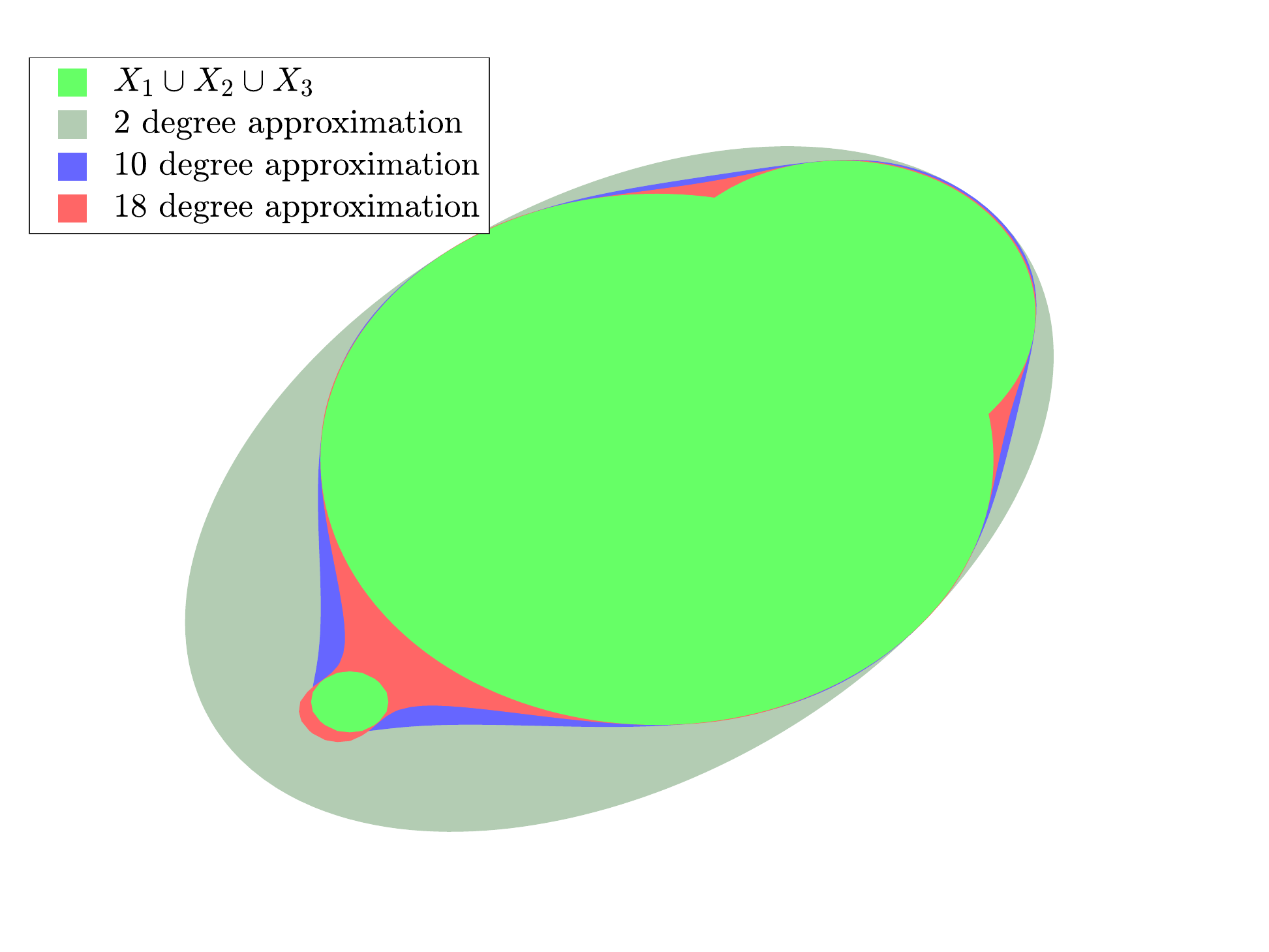}
			\label{fig:Vol union}
		}
			\vspace{-5pt}
		\caption{Plot associated with Example~\ref{ex: Approximation of unions of semialgebraic sets} showing the approximation of the union of semialgebraic sets, $X_1 \cup X_2 \cup X_3$ given in Eq.~\eqref{numerical union of semialg}, shown as the green region.} \label{fig: union of semialg} 	\vspace{-15pt}
	\end{figure}
\end{ex}

\ifthenelse{\boolean{longver}}{%

\begin{ex}[Approximation of Minkowski sums and Pontryagin differences] \label{ex: mink and pont} In~\cite{guthrie2022closed} the Minkowski sum of the following sets was heuristically approximated,
	\begin{align} \label{mink and pont numerical}
		X_1 & =\{x \in \R^2:x_1^2+x_2^2 - 0.25^2<0  \} \\ \nonumber
		X_2 & = \{x \in \R^2: x_1-0.5<0, -0.5-x_1<0,\\ \nonumber 
		& \quad  -x_2-0.5<0, x_2-x_1-0.5<0,x_1+x_2-0.5<0  \}
	\end{align}

In Fig~\ref{fig:Minkoski and Pont approx} we have plotted both the sets $X_1$ and $X_2$ as the gold and purple regions respectively (note that there is an axis scale change between Sub-figures~\ref{subfig: Mink} and~\ref{subfig: Pont}).

In Fig.~\ref{subfig: Mink} we have plotted $X_1 \oplus X_2$ as the green region, which was found by discretizing both $X_1$ and $X_2$ and adding each element together. We have also plotted our outer approximations of $X_1 \oplus X_2$ that take the form $\{x\in \R^2: J_d(x) \le 0\}$ where $J_d$ is found by solving SOS Opt.~\eqref{opt: SOS vol mink sum} for $r=1.77$, $\Lambda=[-1.25,1.25]^2$, $d=2,6$ $\&$ $12$. 

In Fig.~\ref{subfig: Pont} we have plotted $X_1 \ominus X_2$ as the green region, which was found by discretizing both $X_1$ and $X_2$. We have also plotted our inner approximations of $X_1 \ominus X_2$ that take the form $\{x\in \R^2: J_d(x) < 0\}$ where $J_d$ is found by solving SOS Opt.~\eqref{opt: SOS vol Ponty diff} for $r=1.06$, $\Lambda=[-0.75,0.75]^2$, $d=8,10$ $\&$ $16$.  As expected from the convergence proofs given in Sec.~\ref{sec: Convergence of Our Proposed SOS Programs}, in both cases we see as the degree increases our set approximation improves.

\begin{figure}
	\centering
	\subfloat[\scriptsize Minkoski sum approximation]
	{
		\includegraphics[width=0.46 \linewidth, trim = {4cm 2cm 2.25cm 1cm}, clip]{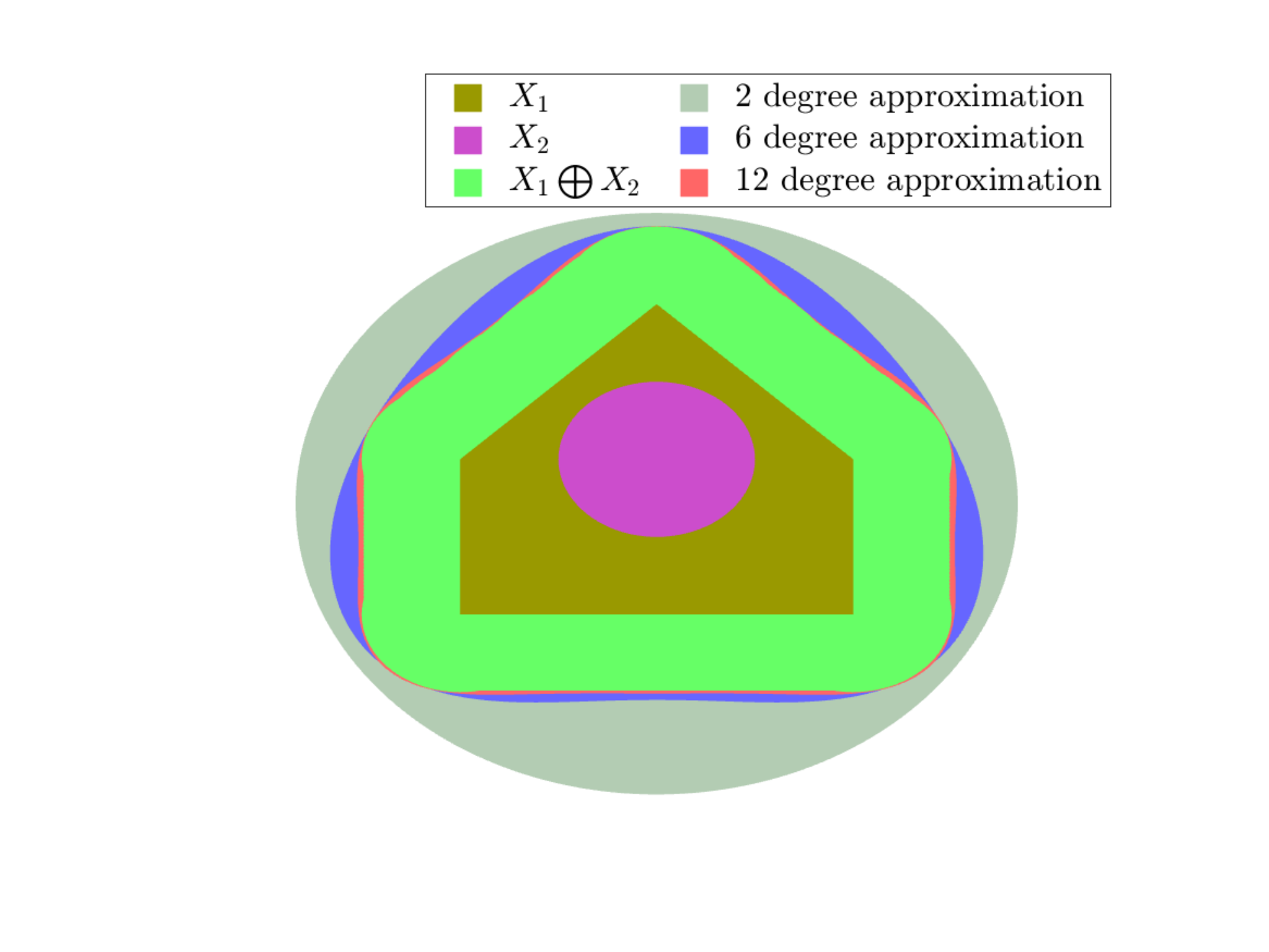}
		\label{subfig: Mink}
	}
	\subfloat[\scriptsize Pontryagin difference approximation]
	{
		\includegraphics[width=0.46 \linewidth,  trim = {4cm 3cm 3cm 0cm}, clip]{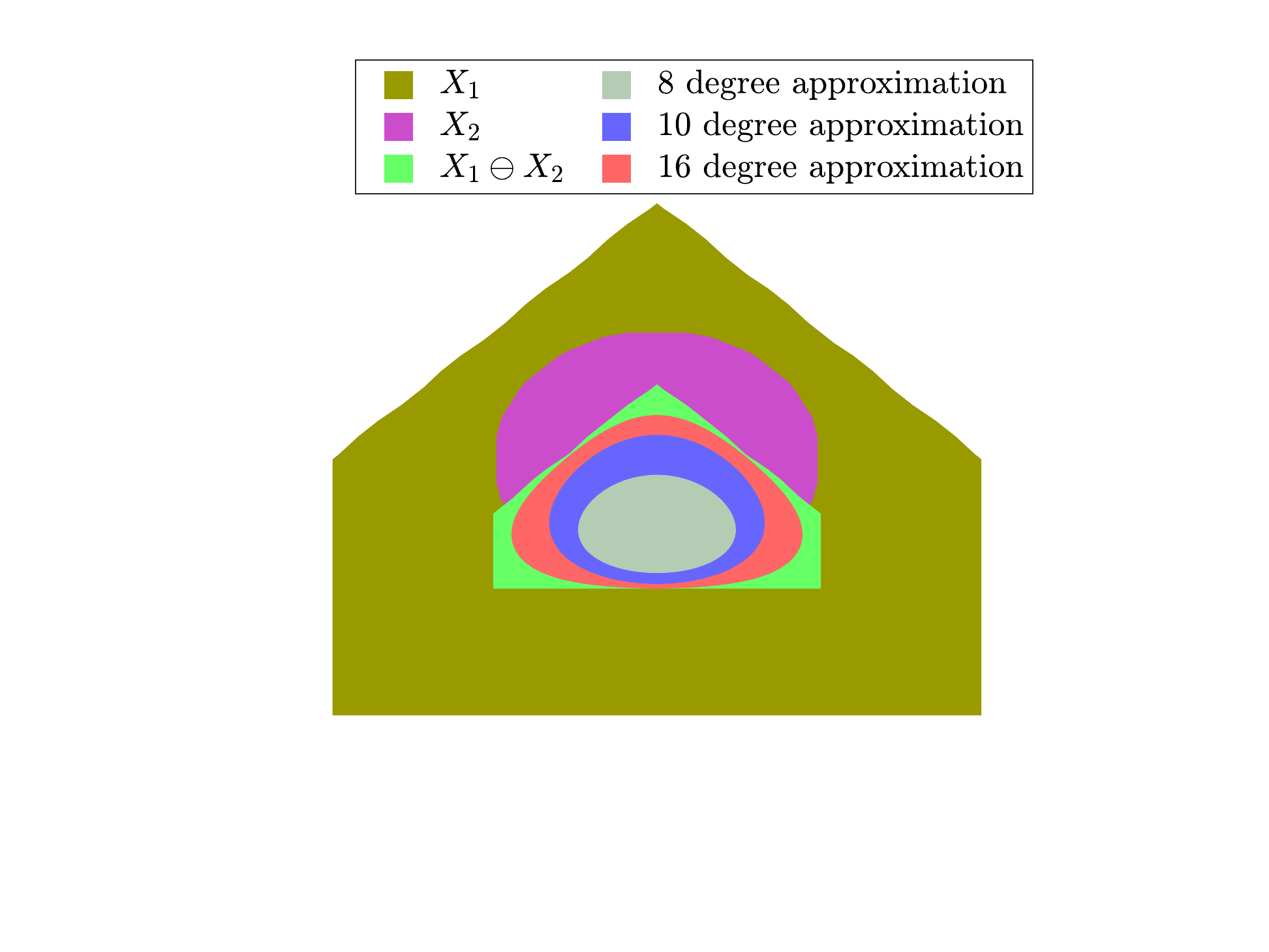}
		\label{subfig: Pont}
	}
	\caption{Plot associated with Example~\ref{ex: mink and pont} showing approximations of $X_1 \oplus X_2$ and $X_1 \ominus X_2$, where $X_1$ and $X_2$ are given in Eq.~\eqref{mink and pont numerical}.}
	\label{fig:Minkoski and Pont approx}
	\vspace{-10pt}
\end{figure}
\end{ex}

\begin{ex}[Approximation of ROAs] \label{ex: ROA}
The ROA is defined as the set of initial conditions for which a systems trajectories tend to an equilibrium point. We next consider the problem of approximating the ROA of the single machine infinite bus system given by the following nonlinear Ordinary Differential Equation (ODE):
\begin{align} \label{ODE: SMIb}
		\begin{bmatrix}	\dot{x}_1(t) \\ 	\dot{x}_2(t) \end{bmatrix}=\begin{bmatrix}
		x_2(t)\\ (1/2H)(P_m-P_e \sin(x_1(t)+\delta_{ep})-Dx_2(t))
	\end{bmatrix},
\end{align}
	where $H=0.0106$, $X_t=0.28$, $P_m=1$, $E_s=1.21$, $V=1$, $P_e=(E_s V)/(P_m X_t)$, $D=0.03$ and $\delta_{ep}=\sin^{-1}(1/P_e)$.
	
	Using a similar method to~\cite{fernandes2022combining} we simulate ODE~\eqref{ODE: SMIb} for various initial conditions, $x(0)=x_0 \in \R^2$. Using these simulations we construct a labelled data set where each data point represents an initial condition that is either an element of the ROA or not. To approximate the ROA we must then solve this machine learning binary classification problem by computing the decision boundary of our labelled data set. We solve this problem by only considering the stable data points, $\{x_i\}_{i=1}^N$. We then use our proposed SOS algorithm to compute an outer approximation of $\{x_i\}_{i=1}^N$. In Fig.~\ref{fig:ROA} we have plotted  our estimation of the ROA as the red region, that is of the form $\{x\in \R^2: J_d(x)  \le 0\}$ where $J_d$ is found by solving SOS Opt.~\eqref{opt: SOS discrete points} for $d=14$.
	\begin{figure}
		\centering
		{
			\includegraphics[width=0.9 \linewidth, trim = {2.5cm 2cm 2cm 2cm}, clip]{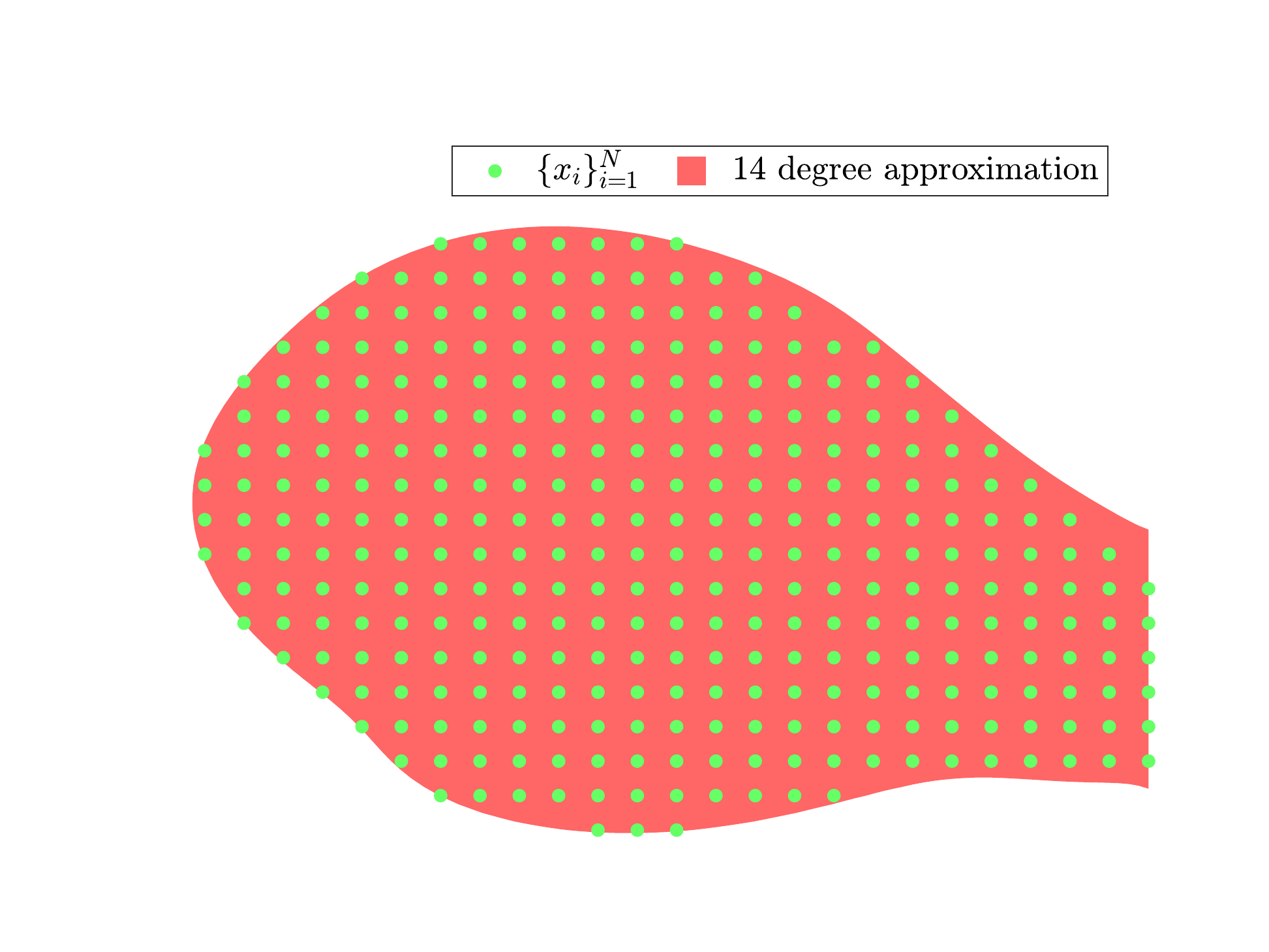}
		
		}
		\caption{Plot associated with Example~\ref{ex: ROA} showing the approximation of the ROA of the single machine infinite bus system from trajectory data.} 	\label{fig:ROA}
	\end{figure}
\end{ex}
	    }{%

}

\begin{ex}[Approximation of attractor sets] \label{ex: lorenz} \ifthenelse{\boolean{longver}}{
		    The Lorenz system can be modelled as the following nonlinear ODE,
\begin{align} \label{ODE: lorenz}
	\begin{bmatrix}	\dot{x}_1(t) \\ 	\dot{x}_2(t) \\ 	\dot{x}_3(t) \end{bmatrix}=\begin{bmatrix}
		\sigma( x_2(t)-x_1(t) )\\ \rho x_1(t) -x_2(t) -x_1(t) x_3(t) \\ x_1(t) x_2(t) - \beta x_3(t)
	\end{bmatrix},
\end{align}
	where $(\sigma, \rho,\beta)=(10,28,\frac{8}{3})$. It is well known that the Lorenz system exhibits a global attractor set in which all trajectories converge towards. }{} The problem of approximating the Lorenz attractor from data can be posed as a machine learning classification problem~\cite{shena2021approximation}. One way to approach this classification problem is by collecting discrete points, $\{x_i\}_{i=1}^N$, of terminal points of trajectories simulated for large amounts of time. Assuming our simulation time is sufficiently large, each of the discrete points, $\{x_i\}_{i=1}^N$, will be inside the attractor set. In Fig.~\ref{fig:Attractor set} we have plotted our approximation of $\{x_i\}_{i=1}^N$ as the red region given by $\{x\in \R^3: J_d(x)  \le 0\}$ where $J_d$ is found by solving SOS Opt.~\eqref{opt: SOS discrete points} for $d=15$.
	\ifthenelse{\boolean{longver}}{%
	\begin{figure}
		\centering
		\subfloat[]
		{
			\includegraphics[width=0.46 \linewidth, trim = {2cm 1cm 3cm 0cm}, clip]{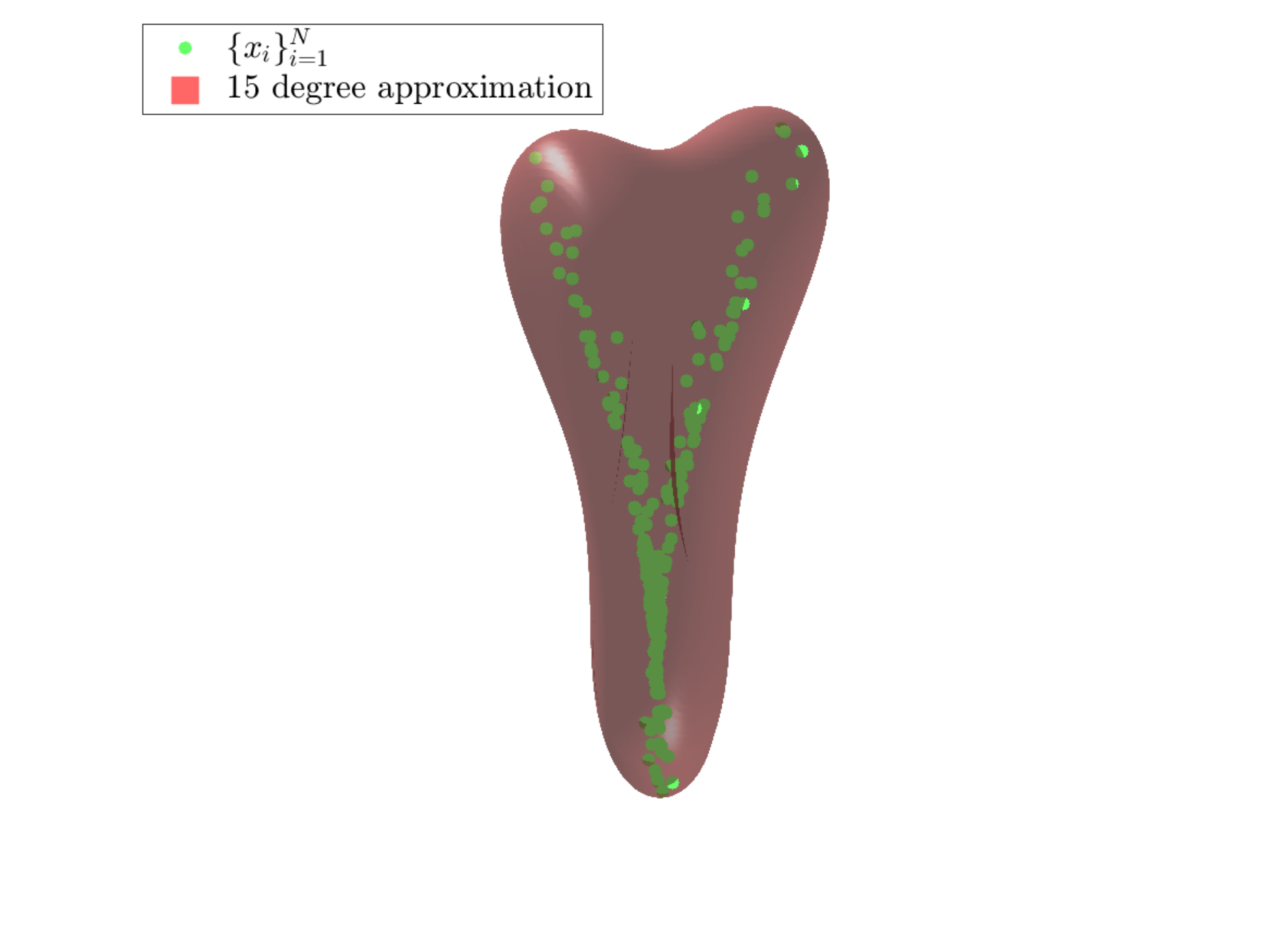}
		}
		\subfloat[]
		{
			\includegraphics[width=0.46 \linewidth,  trim = {4cm 3cm 2cm 0cm}, clip]{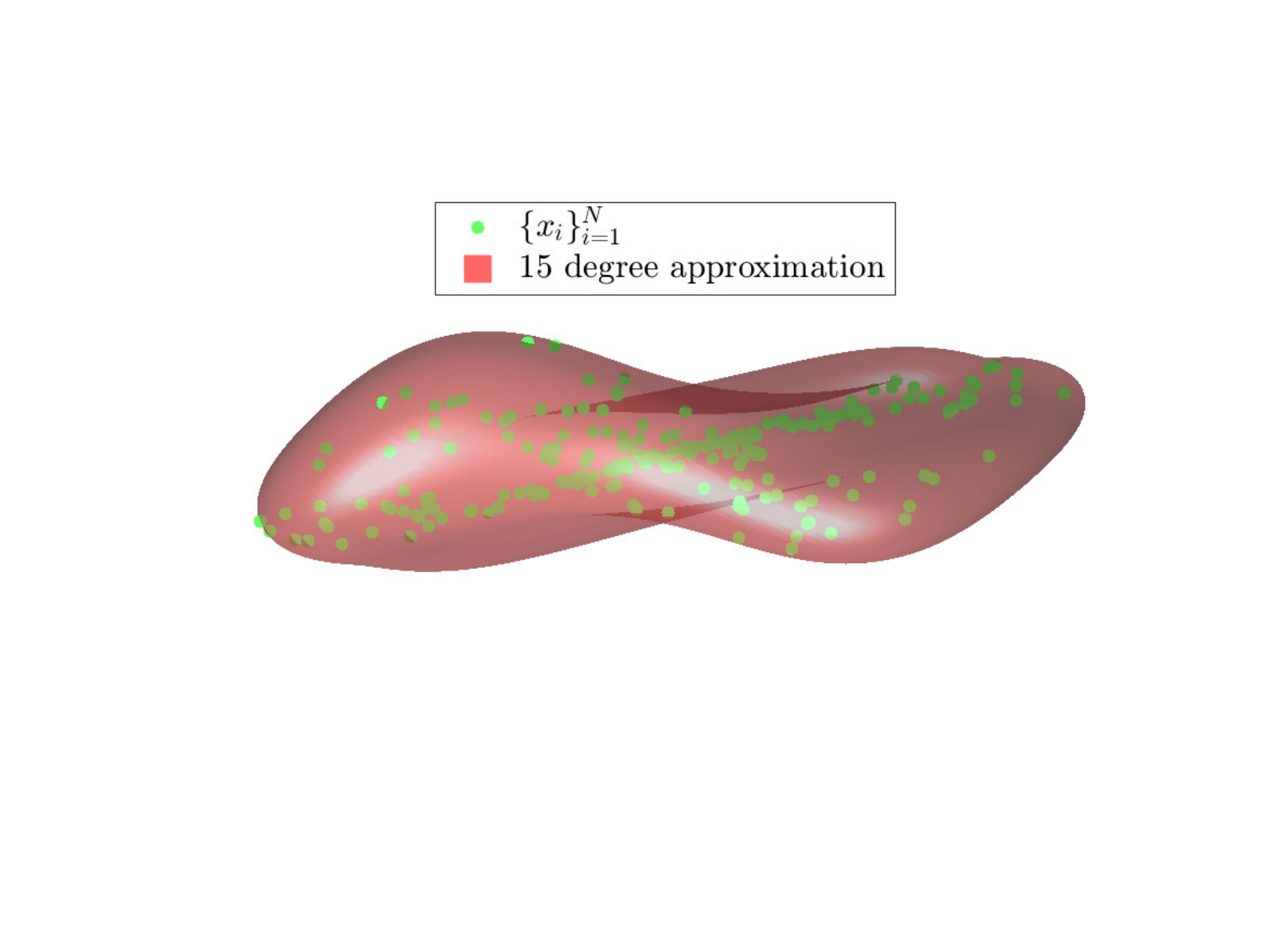}
		}\\
		\subfloat[]
		{
			\includegraphics[width=0.46 \linewidth,  trim = {3cm 1cm 2cm 0cm}, clip]{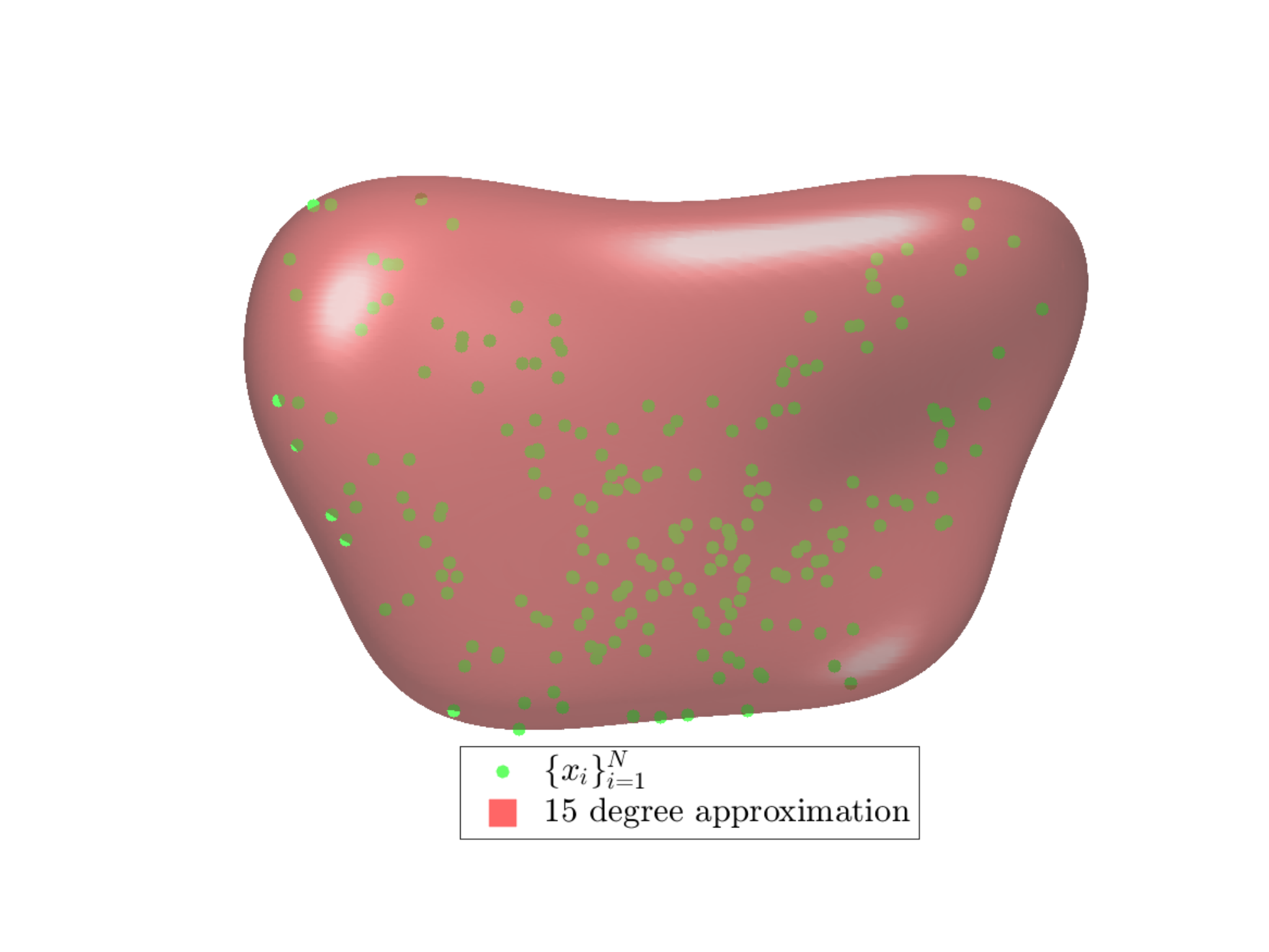}
		}
		\caption{Plot associated with Example~\ref{ex: lorenz} showing various angles of our approximation the Lorenz attractor from trajectory data.}  	\label{fig:Attractor set}
	\end{figure}
    }{%
\begin{figure}
	\centering
	\subfloat[]
	{
		\includegraphics[width=0.46 \linewidth, trim = {2cm 1cm 3cm 0cm}, clip]{fig/angle_1}
	}
	\subfloat[]
	{
		\includegraphics[width=0.46 \linewidth,  trim = {4cm 3cm 2cm 0cm}, clip]{fig/angle_2}
	}\\
	\vspace{-5pt}
	\caption{Plot associated with Example~\ref{ex: lorenz} showing various angles of our approximation the Lorenz attractor from trajectory data.}  	\label{fig:Attractor set} 	\vspace{-15pt}
\end{figure}
    }
\end{ex}

\begin{ex}[Approximation of C-space for collision free path planning] \label{ex: path planning} 



	Dubin's car is the name given to the following discrete time system,
	\begin{align} \label{eq: dubins car}
		x(t+1)=\begin{bmatrix} x_1(t) + \nu \cos(x_3(t)) \\ x_2(t) + \nu \sin(x_3(t)) \\ x_3(t) + \frac{\nu}{L} \tan(u(t)) \end{bmatrix},
	\end{align}
	where $(x_1(t), x_2(t)) \in \R^2$ is the position of the car at time $t \in \N$, $x_3(t) \in  \R$ denotes the
	angle that the car is pointing towards, $u(t) \in \R$ is the steering angle input, $\nu \in \R$ is the fixed speed of the car, and $L>0$ is a parameter that determines the turning radius of the car.
	
	In Fig.~\ref{subfig: workspace} Dubin's car at various time stages is described by the set coloured purple and is given by
	\begin{align}
		X_1=\{x \in \R^2: x_1^2 + x_2^2 -0.1^2<0 \}.
	\end{align}
\ifthenelse{\boolean{longver}}{
	Furthermore, several obstacles are described by golden coloured sublevel sets $X_2:=\cap_{i=1}^6 \{x \in \R^2: g_i(x)<0 \}$, where
	\begin{align*}
		&g_{star}(x) =0.1-2.5 x_1^2 x_2^2-0.05(x_1+x_2)^2,\\
		& g_1(x)  = g_{star}(5x+[0.2, -0.2]^\top ) \\
& g_{2}(x)  = \begin{bmatrix} 0.75+x_1 \\ -x_1 \\-x_2+0.85 \\ -((0.85-0.7)/(0.75))x_1-0.85+x_2 \end{bmatrix} \\
& g_3(x) = g_{2} \left(\begin{bmatrix} \cos((\pi+1)/2) & -\sin((\pi+1)/2) \\ \sin((\pi+1)/2) & \cos((\pi+1)/2) \end{bmatrix}^{-1} x \right) \\
& g_4(x)  = \begin{bmatrix} x_1 \\ 0.25-x_1 \\ 0.85+ x_2 \\-0.8-x_2 \end{bmatrix}, \text{ }  g_5(x)  = \begin{bmatrix} x_1-0.1 \\ 0.5-x_1 \\ 0.4+ x_2 \\-0.2-x_2 \end{bmatrix} , \\
&  g_6(x)  = \begin{bmatrix} x_1-0.4 \\ 0.75-x_1 \\ 0.4+ x_2 \\ 0.8-x_2 \end{bmatrix}.
	\end{align*}
Note, some of the sets that describe our obstacles were taken from the previous works of~\cite{cotorruelo2022sum} and~\cite{guthrie2022closed}.	   }{
	Furthermore, several obstacles are described by golden coloured sublevel sets $X_2:=\cap_{i=1}^6 \{x \in \R^2: g_i(x)<0 \}$ from ~\cite{cotorruelo2022sum,guthrie2022closed} and described in~\cite{jones2023sublevel}.
}

In Fig.~\ref{subfig: cspace} we have plotted our outer sublevel set approximation of $X_1 \oplus X_2$ as the red region. This approximation was obtained by solving Opt.~\eqref{opt: SOS vol mink sum} for $d=12$. Based on this approximation of  $X_1 \oplus X_2$ we then applied the Dynamic Programming (DP) algorithm proposed in~\cite{jones2021generalization} to compute the optimal path collision free path. That is we derived a sequence of inputs $u(0), u(1), \dots,  u(T)$ that drives the system described in Eq.~\eqref{eq: dubins car} from an initial condition, $x(0)=x_0$, to the target set, given by the blue square, in the minimum number of steps while avoiding the enlarged obstacles in the C-space, given by our approximation of $X_1 \oplus X_2$. As shown in Fig.~\ref{subfig: workspace}, solving the path planning problem in C-space for obstacles $X_1 \oplus X_2$, ensures that there is no collisions in the workspace when the shape of Dubin's car is accounted for. 
	
	\ifthenelse{\boolean{longver}}{%
		%
	\begin{figure*}
		%
		\subfloat[Workspace \label{subfig: workspace} ]{\includegraphics[width=0.47 \linewidth, trim = {1cm 1cm 1cm 0cm}, clip]{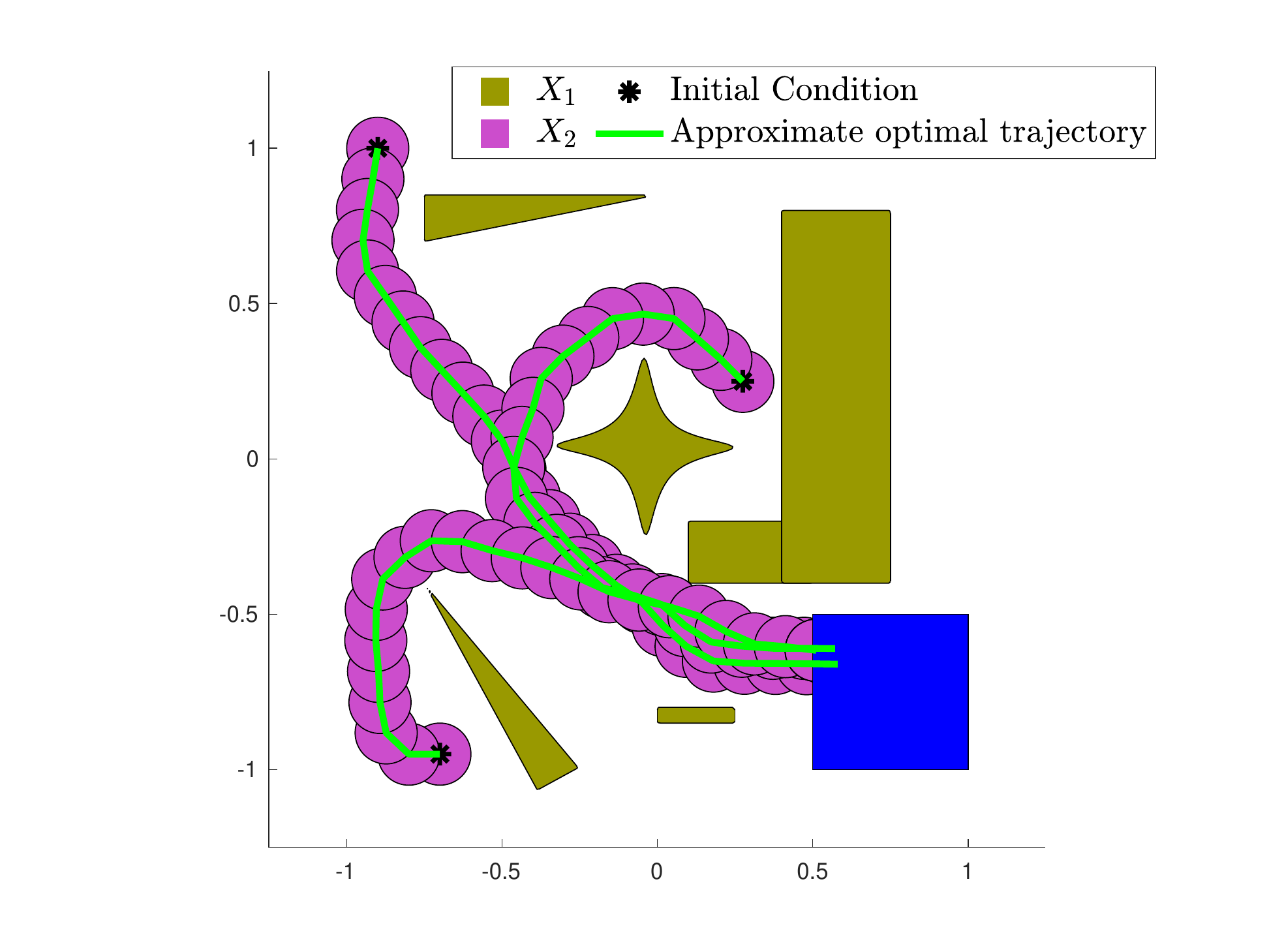}} \hfill 
			\subfloat[C-space \label{subfig: cspace}]{\includegraphics[width=0.47 \linewidth, trim = {1cm 1cm 1cm 0cm}, clip]{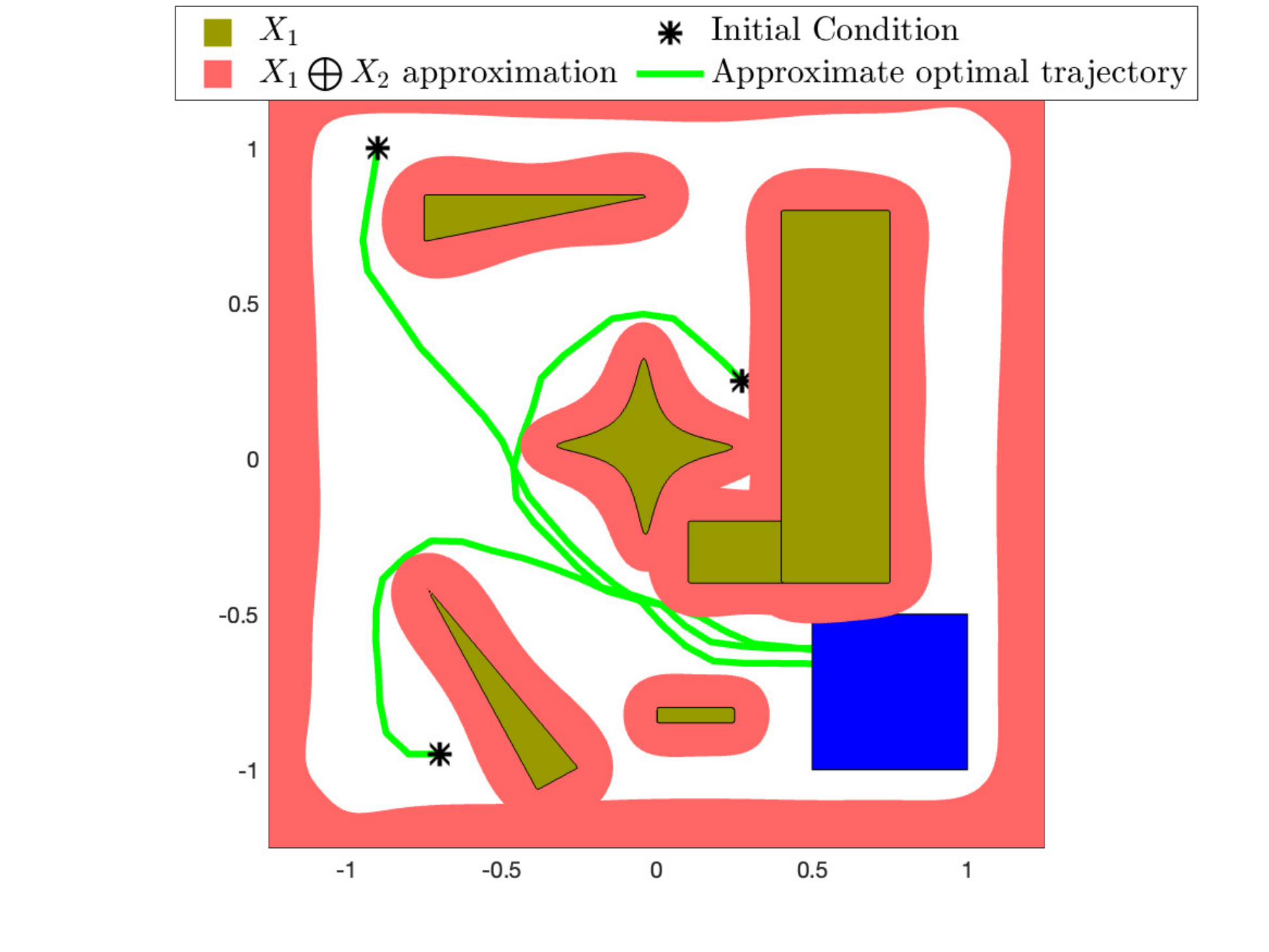}}\hfill 
			
			\caption{Plot associated with Example~\ref{ex: path planning} showing collision free trajectories of Dubin's car in both the workspace and C-space.} \label{fig: path planning}
			\vspace{-15pt}
	\end{figure*}
}{%
	\begin{figure}
	\subfloat[Workspace \label{subfig: workspace} ]{\includegraphics[width=0.5 \linewidth, trim = {2cm 1cm 1.5cm 0cm}, clip]{fig/path_planning_plot2}} \hfill 
	\subfloat[C-space \label{subfig: cspace}]{\includegraphics[width=0.5 \linewidth, trim = {2cm 1cm 1cm 0cm}, clip]{fig/path_planning_plot}}\hfill 
		\vspace{-5pt}
	\caption{Plot associated with Example~\ref{ex: path planning} showing collision free trajectories of Dubin's car in both the workspace and C-space.} \label{fig: path planning}
	\vspace{-15pt}
\end{figure}
    }
\end{ex}

\vspace{-0.5cm}
\section{Conclusion}
We have established a link between the $L^\infty$ and $L^1$ function metrics and the Hausdorff and volume set metrics, respectively, allowing us to construct SOS programs for accurately approximating sets encountered throughout control theory. \ifthenelse{\boolean{longver}}{More specifically, we have shown that if functions are close in the $L^\infty$ norm and one uniformly bounds the other, their sublevel sets are close in the Hausdorff metric. Likewise, if we change the function metric to the $L^1$ norm, the respective sublevel sets are close in the volume metric.}{} By applying our methodology to approximating sets of discrete points, we have proposed a new machine learning one-class classification algorithm that accurately finds decision boundaries for problems with low-dimensional, error-free, and dense data sets. \ifthenelse{\boolean{longver}}{
	    We have applied this new classification algorithm to the problem of approximating ROAs or attractor sets of nonlinear ODEs.}{} Furthermore, our set approximation approach allows us to numerically approximate Minkowski sums, which can be used to compute optimal collision-free paths.
\vspace{-0.25cm}

\ifthenelse{\boolean{longver}}{%
\begin{IEEEbiography}[{\includegraphics[width=1in,height=1.25in,clip,keepaspectratio]{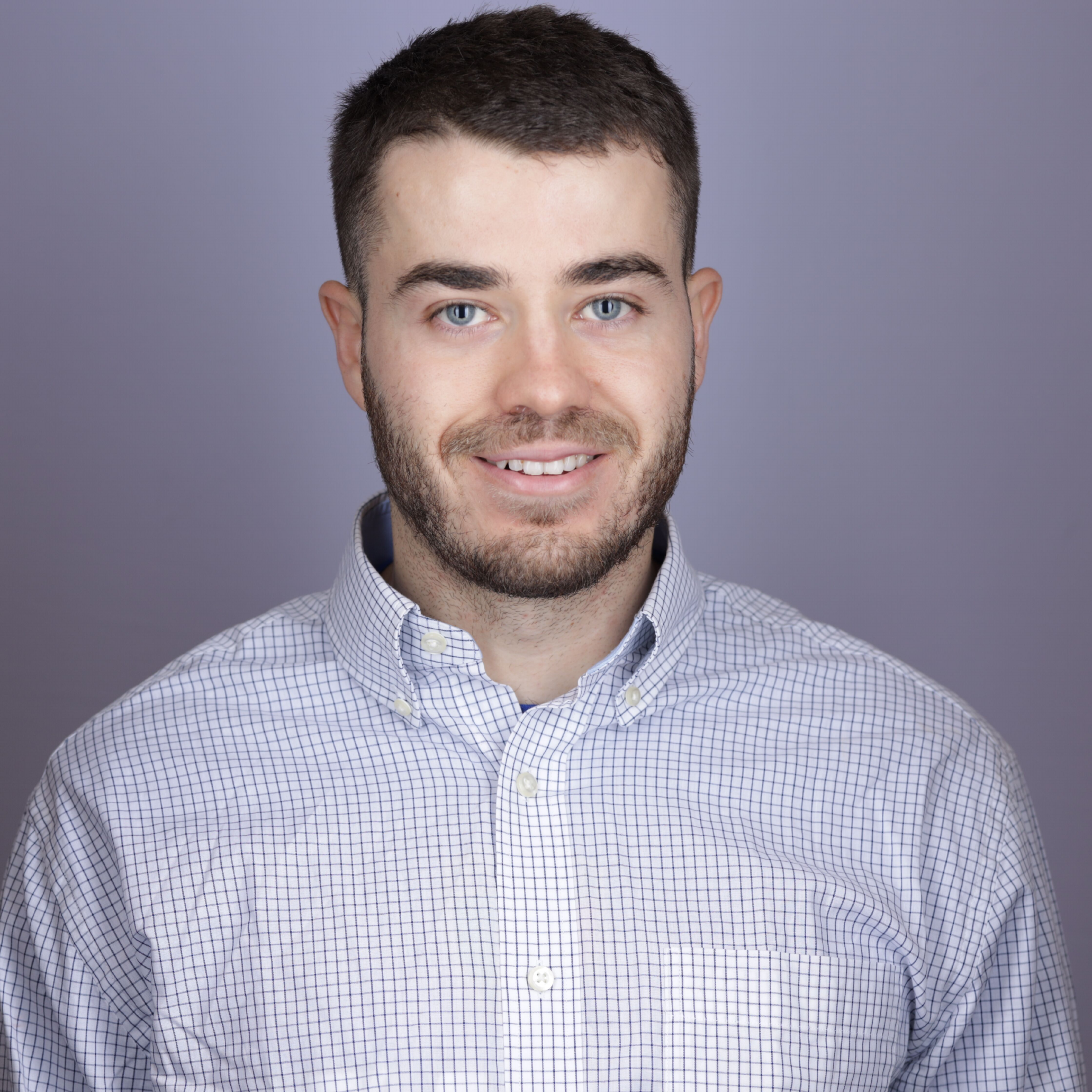}}]{Morgan Jones} \small
	received the MMath degree in
	mathematics from The University of Oxford, England in 2016 and PhD degree from Arizona State University, USA in 2021.
	Since 2022 he has been a lecturer in the department of Automatic Control and Systems Engineering at the University of Sheffield. His research primarily
	focuses on the estimation of reachable sets, attractors and regions of attraction for nonlinear ODEs. Furthermore,
	he has an interest in extensions of the dynamic programing framework to non-separable cost functions.
\end{IEEEbiography}
	    }{
}

\bibliographystyle{ieeetr}
\bibliography{bib_Approx_sublevel_set}

\ifthenelse{\boolean{longver}}{%
\section{Appendix}

\subsection{The Volume Metric} \label{sec: appendix 2}
Recall from Section~\ref{subsec: vol approx} that
\begin{align*} 
	D_V(A,B):=\mu( (A/B) \cup (B/A) ),
\end{align*} where $\mu(A)$ is the Lebesgue measure of $A \subset \R^n$.

\begin{defn} \label{def:metric}
	$D: X \times X \to \R$ is a \textit{metric} if the following is satisfied for all $x,y \in X$,
		\begin{itemize}
			\item $D(x,y) \ge 0$,
			\item $D(x,y)=0$ iff $x=y$,
			\item $D(x,y)=D(y,x)$,
			\item $D(x,z) \le D(x,y) + D(y,z)$.
		\end{itemize}
\end{defn}

The sublevel approximation results presented in this appendix are required in the proof of Theorem~\ref{thm: close in L1 implies close in V norm strict sublevel set}. 

\begin{lem}[\cite{jones2019using}] \label{lem: Dv is metric}
	Consider the quotient space,
		\[
		X:= \mcl B \pmod {\{X \subset \R^n : X \ne \emptyset, \mu(X) =0 \}},
		\]  where $\mcl B $ is the set of all Lebesgue measurable sets. Then $D_V: X \times X \to \R$ is a metric.
\end{lem}


\begin{lem}[\cite{jones2019using}] \label{lem: D_V is related to vol}
	Consider Lebesgue measurable sets $A,B \subset \R^n$. Suppose $A$ and $B$ have finite Lebesgue measure and $B \subseteq A$, then
	\begin{align*}
		D_V(A,B)& =\mu(A/B)= \mu(A)- \mu (B).
	\end{align*}
\end{lem}


\begin{prop}[\cite{jones2020polynomial}] \label{prop: close in L1 implies close in V norm}
	Consider a Lebesgue measurable set $\Lambda \subset \R^n$ with finite Lebesgue measure, a function $V \in L^1(\Lambda, \R)$, and a family of functions $\{J_d \in L^1(\Lambda, \R): d \in \N\}$ that satisfies the following properties:
	\begin{enumerate}
		\item For any $ d \in \N$ we have $J_d(x) \le V(x)$ for all $x \in \Lambda$.
		\item $\lim_{d \to \infty} ||V -J_d||_{L^1(\Lambda, \R)} =0$.
	\end{enumerate}
	Then for all $\gamma \in \R$
	\begin{align} \label{sublevel sets close}
		\lim_{d \to \infty}	D_V \bigg(\{x \in \Lambda : V(x) \le \gamma\}, \{x \in \Lambda : J_d(x) \le \gamma\} \bigg) =0.
	\end{align}
\end{prop} }{}

\ifthenelse{\boolean{longver}}{
\subsection{Counterexamples: When Close Functions Have Distant Sublevel Sets} \label{subsec: appendix counter examples}

As pointed out by Remark~\ref{rem: counterexamples}, relaxing the statement of Thm.~\ref{thm: uniform convegence mplies H convergence} in any way may result in the loss of sublevel set convergence. We first show that if we slightly change the conditions of Thm.~\ref{thm: uniform convegence mplies H convergence} to have $J_d(x)\le V(x)$ (rather than $ V(x) \le J_d(x)$) then we can no longer establish that the sublevel sets of $V$ and $J_d$ will be close.
\begin{cex}[Upper functional approximation is important] \label{cex: lower bound does not tend in Haus}
We show there exists $\gamma \in \R$, $\Lambda \subset \R$, $V \in L^1(\Lambda,\R)$ and $\{J_d\}_{d \in \N} \subset L^1(\Lambda,\R)$ such that $J_d(x)\le V(x)$ for all $x \in \Lambda$ and $\lim_{d \to \infty}  ||V(x) - J_d(x)||_{L^\infty(\Lambda,\R)} dx=0$ but
	\begin{align*}
			\lim_{d \to \infty} D_H\bigg(\{ x \in \Lambda : V(x) < \gamma\}, \{x \in \Lambda : J_d(x) < \gamma\} \bigg)  \ne 0.
		\end{align*}
	Let $\Lambda:=[-10,10]$, $\gamma:=1$, $V(x):=1-\mathds{1}_{[1,2]}(x)$ and $J_d(x)= 1-\mathds{1}_{[1,2]}(x) - \frac{1}{d}$. The functions, along with their corresponding sublevel sets highlighted, have been graphically represented in Fig.~\ref{fig:cx less than}. Clearly, $\lim_{d \to \infty}  ||V(x) - J_d(x)||_{L^\infty(\Lambda,\R)} dx=\lim_{d \to \infty}  \frac{1}{d}=0$. However,
	\begin{align*}
\{ x \in \Lambda : V(x) < 1 \}&= [1,2].\\
\{ x \in \Lambda : J_d(x) < 1 \}&= \Lambda.
	\end{align*}
\end{cex}
Note, Counterexample~\ref{cex: lower bound does not tend in Haus} does not contradict Prop.~\ref{prop: close in L1 implies close in V norm} that deals with the same case where $J_d$ lower bounds $V$. This is because Prop.~\ref{prop: close in L1 implies close in V norm} shows that the non-strict sublevel sets are close in the volume metric. Indeed, $\{ x \in \Lambda : V(x) \le  1 \}= \{ x \in \Lambda : J_d(x) \le  1 \}$ so there is no contradiction in this case.
}{}

\ifthenelse{\boolean{longver}}{
We next consider what happens if we change the other condition of Thm.~\ref{thm: uniform convegence mplies H convergence} where instead of having $||J_d - V||_{L^\infty(\Lambda,\R)} \to 0$ we only have $||J_d - V||_{L^1(\Lambda,\R)} \to 0$.
\begin{cex}[$L^\infty$ functional approximation is important] \label{cex: close in L1 no H}
	We show there exists $\gamma \in \R$, $\Lambda \subset \R$, $V \in L^1(\Lambda,\R)$ and $\{J_d\}_{d \in \N} \subset L^1(\Lambda,\R)$ such that $V(x) \le J_d(x)$ for all $x \in \Lambda$ and $\lim_{d \to \infty}  ||V(x) - J_d(x)||_{L^1(\Lambda,\R)} =0$ but
	\begin{align*}
		\lim_{d \to \infty} D_H\bigg(\{ x \in \Lambda : V(x) < \gamma\}, \{x \in \Lambda : J_d(x) < \gamma\} \bigg)  \ne 0.
	\end{align*}
	Let $\Lambda:=[-10,10]$, $\gamma:=1$, $V(x): =1-\mathds{1}_{\{0.5\}}(x)-\mathds{1}_{[1,2]}(x)$ and $J_d(x):=1- \mathds{1}_{[1,2-\sfrac{1}{d}]}(x)$. The functions, along with their corresponding sublevel sets highlighted, have been graphically represented in Fig.~\ref{fig:cx L1 but not Hauss}. Clearly, $\lim_{d \to \infty}  ||V(x) - J_d(x)||_{L^\infty(\Lambda,\R)} =\lim_{d \to \infty} \sup_{x \in \Lambda}\{ \mathds{1}_{\{0.5\}}(x)+\mathds{1}_{[2-\sfrac{1}{d},2]}(x)\} =1 \ne 0$ and $\lim_{d \to \infty}  ||V(x) - J_d(x)||_{L^1(\Lambda,\R)}= \lim_{d \to \infty}   \int_{\Lambda} \mathds{1}_{\{0.5\}}(x)+\mathds{1}_{[2-\sfrac{1}{d},2]}(x)dx = \lim_{d \to \infty}   \int_{2-\sfrac{1}{d}}^{2} 1dx =\lim_{d \to \infty} \frac{1}{d}=0$. However,
	\begin{align*}
		\{ x \in \Lambda : V(x) < 1 \}&= [1,2]\cup\{0.5\}  .\\
		\{ x \in \Lambda : J_d(x) < 1 \}&= [1,2-\sfrac{1}{d}].
	\end{align*}
Hence 
\begin{align*}
	&\lim_{d \to \infty} D_H\bigg(\{ x \in \Lambda : V(x) < \gamma\}, \{x \in \Lambda : J_d(x) < \gamma\} \bigg)\\
	& = \lim_{d \to \infty} D_H( [1,2]\cup\{0.5\}, [1,2-\sfrac{1}{d}])=0.5 \ne 0.
\end{align*}
\end{cex}
Although Counterexample~\ref{cex: close in L1 no H} shows that if functions are close in the $L^1$ norm then their sublevel sets may not be close in the Hausdorff metric this does not contradict Theorem~\ref{thm: close in L1 implies close in V norm strict sublevel set}, that shows that these sublevel sets must be close in the volume metric. This holds true in the case of Counterexample~\ref{cex: close in L1 no H} since 
\begin{align*}
	&\lim_{d \to \infty} D_V\bigg(\{ x \in \Lambda : V(x) < \gamma\}, \{x \in \Lambda : J_d(x) < \gamma\} \bigg)\\
	& = \lim_{d \to \infty} D_V( [1,2]\cup\{0.5\}, [1,2-\sfrac{1}{d}])\\
	&=  \lim_{d \to \infty} \mu([2-\sfrac{1}{d},2] \cup \{0.5\})=   0.
\end{align*}
Note that in the case of approximating discrete points we approximate $V(x)=1-\mathds{1}_{\{x_i\}_{i=1}^N}(x)$ using Opt.~\eqref{opt: SOS discrete points}. Since, in this case, $V$ is discontinuous we cannot approximate it in the $L^\infty$ norm by a smooth function (like a polynomial). Counterexample~\ref{cex: close in L1 no H} shows that this $L^1$ approximation may not be sufficient to approximate discrete points in the Hausdorff metric but Theorem~\ref{thm: close in L1 implies close in V norm strict sublevel set} shows that we can still use Opt.~\eqref{opt: SOS discrete points} to approximate discrete points in the volume metric.

	    }{
    }

\ifthenelse{\boolean{longver}}{
\begin{figure}
	\centering
	\begin{tikzpicture}[scale=1.5]
		\draw[->] (0,0) -- (3,0) node[right] {$x$};
		\draw[->] (0,-0.5) -- (0,1.5) node[above] {$y$};

		\draw[thick,blue,domain=0:1] plot (\x, {1});
		\draw[thick,blue,domain=1:2] plot (\x, {0});
		\draw[thick,blue,domain=2:3] plot (\x, {1});
		\draw[thick,blue] (1,0) -- (1,1);
		\draw[thick,blue] (2,0) -- (2,1);
		
		\draw[thick,red,dashed,domain=0:1] plot (\x, {1-1/5});
		\draw[thick,red,dashed,domain=1:2] plot (\x, {-1/5});
		\draw[thick,red,dashed,domain=2:3] plot (\x, {1-1/5});
		\draw[thick,red,dashed] (1,0-1/5) -- (1,1-1/5);
		\draw[thick,red,dashed] (2,0-1/5) -- (2,1-1/5);
		
		\fill[blue, opacity=0.2] (1,-0.1) rectangle (2,0.1);
		\fill[red, opacity=0.3] (0,-0.05) rectangle (3,0.05);
		
								\node at (-0.1,1) {$1$};
		\node at (1,-0.3) {$1$};
		\node at (2,-0.3) {$2$};
		
		\node[blue] at (3,1.2) {$V(x)=1-\mathds{1}_{[1,2]}(x)$};
\node[red] at (3.3,0.55) {$J_d(x)=1-\mathds{1}_{[1,2]}(x)-\sfrac{1}{d}$};
		
		\fill[blue, opacity=0.2]  (1,-1.) rectangle (1.25,-1.2) node[right, opacity=1] {$=\{x \in \Lambda: V(x)<1\}$};
		\fill[red, opacity=0.3] (1,-0.65) rectangle (1.25,-0.75) node[right, opacity=1] {$=\{x \in \Lambda: J_d(x)<1\}$};
		
		\draw[ draw=black]  (0.9,-0.5) rectangle (3.5,-1.5) ;
		
	\end{tikzpicture}
\caption{Figure associated with Counterexample~\ref{cex: lower bound does not tend in Haus} showing $\{x \in \Lambda: J_d(x)<1\} \not\to \{x \in \Lambda: V(x)<1\}$ even when $J_d \to V$ in the $L^\infty$ norm from below. } \label{fig:cx less than}
\end{figure}
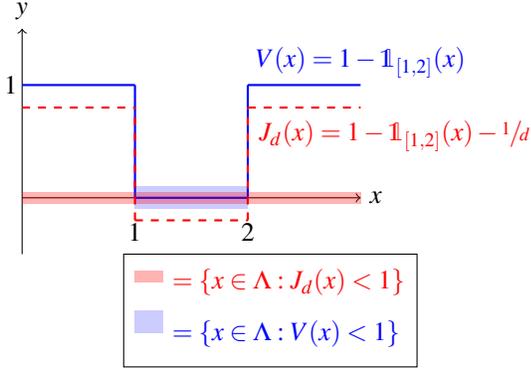
	    }{
    }

\ifthenelse{\boolean{longver}}{
	\begin{figure}
		\centering
	\begin{tikzpicture}[scale=1.5]
		\draw[->] (0,0) -- (3,0) node[right] {$x$};
		\draw[->] (0,-0.5) -- (0,1.5) node[above] {$y$};

		\draw[thick,blue,domain=0:0.499] plot (\x, {1});
	\draw[thick,blue,domain=0.499:0.5001] plot (\x, {0});
	\draw[thick,blue,domain=0.5001:1] plot (\x, {1});
		\draw[thick,blue,domain=1:2] plot (\x, {0});
		\draw[thick,blue,domain=2:3] plot (\x, {1});
		\draw[thick,blue] (0.499,0) -- (0.499,1);
		\draw[thick,blue] (2,0) -- (2,1);
				\draw[thick,blue] (1,0) -- (1,1);

			\draw[thick,red,dashed,domain=0:1] plot (\x, {1});
		\draw[thick,red,dashed,domain=1:2-1/5] plot (\x, {0});
		\draw[thick,red,dashed,domain=2-1/5:3] plot (\x, {1});
		\draw[thick,red,dashed] (2-1/5,0) -- (2-1/5,1);
				\draw[thick,red,dashed] (1,0) -- (1,1);

		\fill[blue, opacity=0.2] (0.45,-0.1) rectangle (0.55,0.1);
				\fill[blue, opacity=0.2] (1,-0.1) rectangle (2,0.1);
		\fill[red, opacity=0.3] (1,-0.05) rectangle (2-1/5,0.05);
		
								\node at (-0.1,1) {$1$};
		\node at (1,-0.3) {$1$};
		\node at (2,-0.3) {$2$};
		\node at (0.5,-0.3) {$0.5$};
		
		\node[blue] at (3,1.2) {$V(x)=1-\mathds{1}_{\{0.5\}}(x)-\mathds{1}_{[1,2]}(x)$};
		\node[red] at (3.6,0.75) {$J_d(x)=1-\mathds{1}_{[1,2-\sfrac{1}{d}]}(x)$};
		
		\fill[blue, opacity=0.2]  (1,-1.) rectangle (1.25,-1.2) node[right, opacity=1] {$=\{x \in \Lambda: V(x)<1\}$};
		\fill[red, opacity=0.3] (1,-0.65) rectangle (1.25,-0.75) node[right, opacity=1] {$=\{x \in \Lambda: J_d(x)<1\}$};
		
		\draw[ draw=black]  (0.9,-0.5) rectangle (3.5,-1.5) ;
		
	\end{tikzpicture}
		\caption{Figure associated with Counterexample~\ref{cex: close in L1 no H} showing $\{x \in \Lambda: J_d(x)<1\} \not\to \{x \in \Lambda: V(x)<1\}$ in the Hausdorff metric even when $J_d \to V$ in the $L^1$ norm from {above}. } \label{fig:cx L1 but not Hauss}
	\end{figure}
}{
}

\ifthenelse{\boolean{longver}}{
	We next consider what happens if we modify the statements of Theorems~\ref{thm: uniform convegence mplies H convergence} and~\ref{thm: close in L1 implies close in V norm strict sublevel set} that  the non-strict sublevel sets converge, rather than the strict sublevel sets converge.
	\begin{cex}[Strictness of the sublevel set is important] \label{cex: close L inf but non-strict far}
		We show there exists $\gamma \in \R$, $\Lambda \subset \R$, $V \in L^1(\Lambda,\R)$ and $\{J_d\}_{d \in \N} \subset L^1(\Lambda,\R)$ such that $V(x) \le J_d(x)$ for all $x \in \Lambda$ and $\lim_{d \to \infty}  ||V(x) - J_d(x)||_{L^\infty(\Lambda,\R)} =0$ but
		\begin{align*}
		&	\lim_{d \to \infty} D_H\bigg(\{ x \in \Lambda : V(x) \le \gamma\}, \{x \in \Lambda : J_d(x) \le \gamma\} \bigg)  \ne 0,\\
			&\lim_{d \to \infty} D_V\bigg(\{ x \in \Lambda : V(x) \le \gamma\}, \{x \in \Lambda : J_d(x) \le \gamma\} \bigg)  \ne 0.
		\end{align*}
		Let $\Lambda:=[0,3]$, $\gamma:=1$, $V(x): =1-\mathds{1}_{[1,2]}(x)$ and $J_d(x):=1- \mathds{1}_{[1,2]}(x)+\sfrac{1}{d}\mathds{1}_{[2,2.5]}(x)$. The functions, along with their corresponding sublevel sets highlighted, have been graphically represented in Fig.~\ref{fig:cx non strict}. Clearly, $\lim_{d \to \infty}  ||V(x) - J_d(x)||_{L^\infty(\Lambda,\R)} =\lim_{d \to \infty} \sup_{x \in \Lambda}\{ \sfrac{1}{d}\mathds{1}_{[2,2.5]}(x))\} =\lim_{d \to \infty} \sfrac{1}{d} = 0$. However,
		\begin{align*}
			\{ x \in \Lambda : V(x) \le 1 \}&= \Lambda=[0,3]  .\\
			\{ x \in \Lambda : J_d(x) \le 1 \}&= [0,2) \cup (2.5,3].
		\end{align*}
		Hence,
		\begin{align*}
			&\lim_{d \to \infty} D_H\bigg(\{ x \in \Lambda : V(x) \le \gamma\}, \{x \in \Lambda : J_d(x) \le \gamma\} \bigg)\\
			& = \lim_{d \to \infty} D_H( [0,3], [0,2) \cup (2.5,3])=0.25 \ne 0,
		\end{align*}
	Moroever,
			\begin{align*}
		&\lim_{d \to \infty} D_V\bigg(\{ x \in \Lambda : V(x) \le \gamma\}, \{x \in \Lambda : J_d(x) \le \gamma\} \bigg)\\
		& = \lim_{d \to \infty} D_V( [0,3], [0,2) \cup (2.5,3])=\mu([2,2.5])=0.5 \ne 0.
	\end{align*}
	\end{cex}
}{}

\ifthenelse{\boolean{longver}}{
	\begin{figure}
		\centering
		\begin{tikzpicture}[scale=1.5]
			\draw[->] (0,0) -- (3,0) node[right] {$x$};
			\draw[->] (0,-0.5) -- (0,1.5) node[above] {$y$};

			\draw[thick,blue,domain=0:1] plot (\x, {1});
			\draw[thick,blue,domain=1:2] plot (\x, {0});
			\draw[thick,blue,domain=2:3] plot (\x, {1});
			\draw[thick,blue] (2,0) -- (2,1);
			\draw[thick,blue] (1,0) -- (1,1);

			\draw[thick,red,dashed,domain=0:1] plot (\x, {1});
			\draw[thick,red,dashed,domain=1:2] plot (\x, {0});
						\draw[thick,red,dashed,domain=2:2.5] plot (\x, {1+1/5});
			\draw[thick,red,dashed,domain=2.5:3] plot (\x, {1});
			\draw[thick,red,dashed] (2,0) -- (2,1+1/5);
						\draw[thick,red,dashed] (2.5,1) -- (2.5,1+1/5);
			\draw[thick,red,dashed] (1,0) -- (1,1);

			\fill[blue, opacity=0.2] (0,-0.1) rectangle (3,0.1);
			\fill[red, opacity=0.3] (0,-0.05) rectangle (2,0.05);
					\fill[red, opacity=0.3] (2.5,-0.05) rectangle (3,0.05);

			\node at (1,-0.3) {$1$};
						\node at (-0.1,1) {$1$};
			\node at (2,-0.3) {$2$};

			\node[blue] at (3.5,1.2) {$V(x)=1-\mathds{1}_{[1,2]}(x)$};
			\node[red] at (3.8,0.75) {$J_d(x)=1-\mathds{1}_{[1,2]}(x)+\sfrac{1}{d}\mathds{1}_{[2,2.5]}(x)$};
			
			\fill[blue, opacity=0.2]  (1,-1.) rectangle (1.25,-1.2) node[right, opacity=1] {$=\{x \in \Lambda: V(x)\le1\}$};
			\fill[red, opacity=0.3] (1,-0.65) rectangle (1.25,-0.75) node[right, opacity=1] {$=\{x \in \Lambda: J_d(x)\le1\}$};
			
			\draw[ draw=black]  (0.9,-0.5) rectangle (3.5,-1.5) ;
			
		\end{tikzpicture}
		\caption{Figure associated with Counterexample~\ref{cex: close L inf but non-strict far} showing $\{x \in \Lambda: J_d(x)\le1\} \not\to \{x \in \Lambda: V(x)\le1\}$ in the Hausdorff and volume metric even when $J_d \to V$ in the $L^\infty$ norm from {above}. } \label{fig:cx non strict}
	\end{figure}
}{
}

\ifthenelse{\boolean{longver}}{
	\begin{figure}
		\centering
		\begin{tikzpicture}[scale=0.5 ]
			\begin{axis}[
				xticklabels={},
				yticklabels={},
				axis lines=middle,
				domain=0:7,
				samples=100,
				smooth,
				scale=1.5,
				]
				\addplot[blue,thick] {-exp(-x)};
				\addplot[red,dashed,thick] {-exp(-x)+1/20};
				\fill[blue, opacity=0.2] (0,-0.05) rectangle (10,0.05);
				\fill[red, opacity=0.3] (0,-0.02) rectangle (2.996,0.02);
				\fill[blue, opacity=0.2]  (2,-0.3) rectangle (2.25,-0.4); 
				\node[blue] at (4,-0.35) {$=\{x \in \Lambda: V(x)<1\}$};
				\fill[red, opacity=0.3] (2,-0.25) rectangle (2.25,-0.27); 
				\node[red]at (4,-0.26)  {$=\{x \in \Lambda: J_d(x)<1\}$};
				
				\draw[ draw=black]  (1.9,-0.2) rectangle (6,-0.45) ;
				\node[blue] at (3.5,1.2) {$V(x)=1-\mathds{1}_{[1,2]}(x)$};
				\node[red] at (3.8,0.75) {$J_d(x)=1-\mathds{1}_{[1,2]}(x)+\sfrac{1}{d}\mathds{1}_{[2,2.5]}(x)$};
				
				\node[blue] at (1.5,-0.5) {$V(x)=-e^{-x}$};
				\node[red] at (1,-0.1) {$J_d(x)=\sfrac{1}{d}-e^{-x}$};
				\node[black] at (2.99,-0.026) {$\ln(d)$};
			\end{axis}
		\end{tikzpicture}
		\caption{Figure associated with Counterexample~\ref{cex: Compact lambda} showing $\{x \in \Lambda: J_d(x) < 1\} \not\to \{x \in \Lambda: V(x) <1\}$ in the Hausdorff or volume metric even when $J_d \to V$ in the $L^\infty$ norm from {above} if $\Lambda$ is not compact. } \label{fig:cx compact Lambda}
	\end{figure}
}{
}

\ifthenelse{\boolean{longver}}{
	We next show that if we relax the condition that $\Lambda\subset \R^n$ is compact  in Theorem~\ref{thm: uniform convegence mplies H convergence}, or the condition $\Lambda \subset \R^n$ has finite Lebesque measure Theorem~\ref{thm: close in L1 implies close in V norm strict sublevel set}, then we may not get sublevel set convergence. 
	\begin{cex}[Compactness of sublevel set domain is important] \label{cex: Compact lambda}
	We show there exists $\gamma \in \R$, \textbf{non-compact} $\Lambda \subset \R$, $V \in L^1(\Lambda,\R)$ and $\{J_d\}_{d \in \N} \subset L^1(\Lambda,\R)$ such that $V(x) \le J_d(x)$ for all $x \in \Lambda$ and $\lim_{d \to \infty}  ||V(x) - J_d(x)||_{L^\infty(\Lambda,\R)} =0$ but
\begin{align*}
	&	\lim_{d \to \infty} D_H\bigg(\{ x \in \Lambda : V(x) < \gamma\}, \{x \in \Lambda : J_d(x) < \gamma\} \bigg)  \ne 0.
\end{align*}
Let $\Lambda=[0,\infty)$, $\gamma=0$, $V(x)=-e^{-x}$ and $J_d(x)=\sfrac{1}{d}-e^{-x}$. 
The functions, along with their corresponding sublevel sets highlighted, have been graphically represented in Fig.~\ref{fig:cx compact Lambda}. Clearly, $\lim_{d \to \infty}  ||V(x) - J_d(x)||_{L^\infty(\Lambda,\R)} =\lim_{d \to \infty} \sfrac{1}{d} = 0$. However,
\begin{align*}
	\{ x \in \Lambda : V(x) < 0 \}&= \Lambda=[0,\infty)  .\\
	\{ x \in \Lambda : J_d(x) <0 \}&= [0,\ln(d)).
\end{align*}
Hence,
\begin{align*}
	&\lim_{d \to \infty} D_H\bigg(\{ x \in \Lambda : V(x) < \gamma\}, \{x \in \Lambda : J_d(x) < \gamma\} \bigg)\\
	& = \lim_{d \to \infty} D_H( [0,\infty), [0,\ln(d)) )=\infty \ne 0,
\end{align*}
	\end{cex}
}{}

\ifthenelse{\boolean{longver}}{
\subsection{Polynomial Approximation} \label{subsec: appendix poly approx}
In Sec.~\ref{sec: SOS programs for set approx} we characterized several sets (intersections and unions of semialgebraic sets, Minkowski sums, Pontryagin differences and discrete points) by sublevel sets of various functions. We now show that we can approximate these functions arbitrarily well by polynomials that are also feasible to our associated SOS optimization problems. In order to approximate these functions we use the Weierstrass approximation theorem.


\begin{thm}[Weierstrass approximation theorem \cite{o1981five}] \label{thm: Weierstrass approx thm}
	Let $E \subset \R^n$ be an open set and $f \in C^1(E, \R)$. For any compact set $K \subseteq E$ and $\eps>0$ there exists  $g \in \R[x]$ such that
	\begin{align*}
		\sup_{x \in K}| f(x) -  g(x)| < \eps.
	\end{align*}
\end{thm}

We next show that there exists a polynomial that is feasible to Opt.~\eqref{opt: SOS Haus Unions} and that arbitrarily approximates the function, $V(x):=\min_{1 \le i \le m} g_i(x)$, whose sublevel set characterizes the set given in Eq.~\eqref{set: union strict}.
\begin{prop} \label{prop: existence of poly close to min of polys}
	Consider a compact set $\Lambda \subset \R^n$, functions $g_i \in LocLip(\Lambda,\R)$ for $1 \le i \le m$, a scalar $r>0$, and $V(x):= \min_{1 \le i \le m}\{ g_i(x) \}$. Then for any $\eps>0$ there exists $H_1,H_2 \in\R[x]$ such that
	\begin{align} \label{eq: H close to min of g}
		& \sup_{ x \in \Lambda} |V(x) - H_j(x) | < \eps \text{ for } j \in \{1,2\}, \\ \nonumber
		& H_1(x) < g_i(x) \text{ for all } x \in B_r(0) \text{ and } 1 \le i \le m,\\ \nonumber 
		& H_2 (x)> g_i(x) \text{ for all } x \in B_r(0)  \cap \underline{ \mcl Y_i } \text{ and } 1 \le i \le m,
	\end{align}
	where $ \underline{\mcl Y_i } :=\{y \in \Lambda: g_i(y) \le g_j(y) \text{ for } 1 \le j \le m  \}$.
\end{prop}

\begin{proof} We first show the existence of $H_1 \in \R[x]$ that satisfies Eq.~\eqref{eq: H close to min of g}. Since $\Lambda \subset \R^n$ and {$B_r(0)$} are compact sets it follows that there exists $R>0$ such that $\Lambda \cup B_r(0) \subset B_R(0)$. Let $\eps>0$. Since $V$ is continuous by Lem.~\ref{lem: max minimum of Lipschitz functions is Lipschitz} it follows by {Thm.~\ref{thm: Weierstrass approx thm}} that there exists $P \in \R[x]$ such that $\sup_{ x \in B_R(0)} |V(x) - P(x) | <\eps/4$. Let $H_1(x):=P(x) - \eps/4$. Then
	\begin{align*}
	|V(x) - H_1(x) | & = |V(x) - P(x) + \eps/2| < |V(x) - P(x)| + \eps/4\\
	& < \eps/2 \text{ for all } x \in B_R(0),
	\end{align*}
	and hence $\sup_{ x \in \Lambda} |V(x) - H_1(x) | \le \eps/2 < \eps$. 
	
	Moreover, since $|V(x) - P(x)|<\eps/4$ for all $x \in B_R(0)$ we have that $P(x)<V(x) + \eps/4$ for all $x \in B_R(0)$ . Hence $H_1(x)<V(x) + \eps/4-\eps/4= V(x)= \min_{1 \le j \le m}\{ g_j(x) \} \le g_i(x)$ for all $1 \le i \le m$ and $x \in B_r(0) \subset B_R(0)$.
	
	We next show the existence of $H_2 \in \R[x]$ that satisfies Eq.~\eqref{eq: H close to min of g}. This time let $H_2(x):=P(x) +\eps/4$. Then
	\begin{align*}
		|V(x) - H_2(x) | & = |V(x) - P(x) - \eps/2| < |V(x) - P(x)| + \eps/4\\
		& < \eps/2 \text{ for all } x \in B_R(0),
	\end{align*}
	and hence $\sup_{ x \in \Lambda} |V(x) - H_2(x) | < \eps$. 
	
	Moreover, since $|V(x) - P(x)|<\eps/4$ for all $x \in B_R(0)$ we have that $P(x) > V(x) -  \eps/4$ for all $x \in B_R(0)$. Hence $H_2(x)>V(x) - \eps/4 + \eps/4=V(x)$ for all $x \in B_R(0)$. Now, when $x \in \underline{ \mcl Y_i}$ we have $V(x)=g_i(x)$. Therefore, $H_2(x)>V(x)=g_i(x)$ for all $x \in \underline{ \mcl Y_i } \cap B_r(0)$.

%
\end{proof}

We next show that there exists a polynomial that is feasible to Opt.~\eqref{opt: SOS Haus intersections} and that arbitrarily approximates the function, $V(x):=\max_{1 \le i \le m} g_i(x)$, whose sublevel set characterizes the set given in Eq.~\eqref{eq: semialg sets}.
\begin{cor} \label{cor: existence of poly close to max of polys}
	Consider a compact set $\Lambda \subset \R^n$, $g_i \in LocLip(\Lambda,\R)$ for $1 \le i \le m$ and $V(x):= \max_{1 \le i \le m}\{ g_i(x) \}$. Then for any $\eps>0$ there exists $H_1,H_2 \in \R[x]$ such that
	\begin{align} \label{pfeq:1 }
	& \sup_{ x \in \Lambda} |V(x) - H_j(x) | < \eps \text{ for } j \in \{1,2\}, \\ \nonumber
	& H_1(x) > g_i(x) \text{ for all } x \in B_r(0) \text{ and } 1 \le i \le m,\\ \nonumber
	& H_2(x) < g_i(x) \text{ for all } x \in B_r(0) \cap \bar{\mcl Y_i }\text{ and } 1 \le i \le m,
	\end{align}
where $\bar{\mcl Y_i } :=\{y \in \Lambda: g_i(y) \ge g_j(y) \text{ for } 1 \le j \le m  \}$.
\end{cor}
\begin{proof}
Note that $\max_{1 \le i \le m}\{ g_i(x) \}= -\min_{1 \le i \le m}\{ -g_i(x) \}$. Let $\tilde{V}(x):=\min_{1 \le i \le m}\{ -g_i(x) \}= -V(x)$. By Prop.~\ref{prop: existence of poly close to min of polys} there exists $\tilde{H_1}, \tilde{H_2} \in \R[x]$ such that 
	\begin{align*}
	& \sup_{ x \in \Lambda} |\tilde{V}(x) - \tilde{H_j}(x) | < \eps \text{ for } j \in \{1,2\}, \\
	& \tilde{H_1}(x) < -g_i(x) \text{ for all } x \in B_r(0) \text{ and } 1 \le i \le m, \\
	& \tilde{H_2}(x) > -g_i(x)  \text{ for all } x \in B_r(0) \cap \tilde{\underline{\mcl Y_i} }\text{ and } 1 \le i \le m,
\end{align*}
where $\tilde{\underline{\mcl Y_i }} :=\{y \in \Lambda: -g_i(y) \le- g_j(y) \text{ for } 1 \le j \le m  \} = \{y \in \Lambda: g_i(y) \ge g_j(y) \text{ for } 1 \le j \le m  \} = \bar{\mcl Y_i } $.

Let $H_1(x):=-\tilde{H_1}(x)$ and $H_2(x):=-\tilde{H_2}(x)$. Clearly $H_1$ and $H_2$ satisfies Eq.~\eqref{pfeq:1 }, completing the proof.
\end{proof}

We next show that there exists a polynomial that is feasible to Opt.~\eqref{opt: SOS vol mink sum} and that arbitrarily approximates the function, $V(x):=\inf_{w  \in  \{z \in \Lambda: g_{2,i}(z) \le 0 \text{ for } 1 \le i \le m_2 \} } \min_{1 \le i \le m_1} g_{1,i}(x-w)$, whose sublevel set characterizes the set given in Eq.~\eqref{set: mink union and int}.

\begin{lem} \label{lem: function approx mink sum}
	Consider a compact set $\Lambda \subset \R^n$, $g_i \in LocLip(\Lambda,\R)$ for $1 \le i \le m$ and $V(x):= \inf_{w  \in  \{z \in \Lambda: g_{2,i}(z) \le 0 \text{ for } 1 \le i \le m_2 \} } \min_{1 \le i \le m_1} g_{1,i}(x-w)$. Then for any $\eps>0$ there exists $H \in \R[x]$ such that
\begin{align} \label{pfeq:111}
	& \sup_{ x \in \Lambda} |V(x) - H(x) | < \eps, \\ \nonumber
	& H(x) < g_{1,i}(x-w)  \text{ for all } x \in B_r(0), \\ \nonumber 
	& \qquad w \in \{z \in \Lambda: g_{2,j}(z) \le 0 \text{ for } 1 \le j \le m_2 \}  \text{ and } 1 \le i \le m_1.
\end{align}
\end{lem}
\begin{proof}
	Since $\Lambda \subset \R^n$ and $B_r(0)$ are compact sets there exists $R>0$ such that $\Lambda \cup B_r(0) \subset B_R(0)$. Now, by {Lem.~\ref{lem: max minimum of Lipschitz functions is Lipschitz}} it follows that $V$ is a continuous function. Hence, by {Thm.~\ref{thm: Weierstrass approx thm}}, for any $\eps>0$ there exists a polynomial $P \in \R[x]$ such that
	\begin{align*}
	\sup_{ x \in B_R(0)} |V(x) - P(x) | < \frac{\eps}{4}.
	\end{align*}
Let us consider $H(x):=P(x) -  \frac{\eps}{4}$. Then
\begin{align*}
	|V(x) - H(x)|  \le 	|V(x) - P(x)|  + \frac{\eps}{4} < \frac{\eps}{2} \text{ for } x \in B_R(0).
\end{align*}
Hence, 	$\sup_{ x \in \Lambda} |V(x) - H(x) | < \eps$.

Also note that since $|V(x) - P(x)| < \frac{\eps}{4}$ for all $x \in B_R(0)$ it follows that $P(x)< V(x) + \frac{\eps}{4}$ and hence $H(x)= P(x)- \frac{\eps}{4}< V(x)$ for all $x \in B_r(0) \subset B_R(0)$. Therefore
\begin{align*}
	P(x)& < V(x)= \inf_{u  \in  \{z \in \Lambda: g_{2,j}(z) \le 0 \text{ for } 1 \le j \le m_2 \} } \min_{1 \le j \le m_1} g_{1,j}(x-u)\\
	& \le  \inf_{u  \in  \{z \in \Lambda: g_{2,j}(z) \le 0 \text{ for } 1 \le j \le m_2 \} }  g_{1,i}(x-u) \\
	& \le g_{1,i}(x-w) \text{ for all } x \in B_r(0), 1 \le i \le m_1,  \text{ and } \\
	& \qquad \qquad  w \in \{z \in \Lambda: g_{2,j}(z) \le 0 \text{ for } 1 \le j \le m_2 \}.
\end{align*}
Thus we have shown Eq.~\eqref{pfeq:111} completing the proof.
\end{proof}

We next show that there exists a polynomial that is feasible to Opt.~\eqref{opt: SOS vol Ponty diff} and that arbitrarily approximates the function, $V(x):=\sup_{w  \in  \{z \in \Lambda: g_{2,i}(z) \le 0 \text{ for } 1 \le i \le m_2 \} } \max_{1 \le i \le m_1} g_{1,i}(x+w)$, whose sublevel set characterizes the set given in Eq.~\eqref{set: pont int and int}.

\begin{cor} \label{cor: function approx ponty diff}
	Consider a compact set $\Lambda \subset \R^n$, $g_i \in LocLip(\Lambda,\R)$ for $1 \le i \le m$ and $V(x):= \sup_{w  \in  \{z \in \Lambda: g_{2,i}(z) \le 0 \text{ for } 1 \le i \le m_2 \} } \max_{1 \le i \le m_1} g_{1,i}(x+w)$. Then for any $\eps>0$ there exists $H \in \R[x]$ such that
	\begin{align} \label{pfeq:111}
		& \sup_{ x \in \Lambda} |V(x) - H(x) | < \eps, \\ \nonumber
		& H(x) > g_{1,i}(x+w)  \text{ for all } x \in B_r(0), \\ \nonumber 
		& \qquad w \in \{z \in \Lambda: g_{2,j}(z) \le 0 \text{ for } 1 \le j \le m_2 \}  \text{ and } 1 \le i \le m_1.
	\end{align}
\end{cor}
\begin{proof}
	Follows by a similar argument to Lem.~\ref{lem: function approx mink sum}.
\end{proof}

We next show that there exists a polynomial that is feasible to Opt.~\eqref{opt: SOS discrete points} and that arbitrarily approximates the function, $V(x):=1-\mathds{1}_{\{x_i\}_{i=1}^N}(x)$, whose sublevel set characterizes the set given by $X=\{x_i\}_{i=1}^N$. 

\begin{prop} \label{prop: L1 poly approx of indicator}
Consider a compact set $\Lambda \subset \R^n$ and some discrete points $\{x_i\}_{i=1}^N \subset \Lambda$. Then, for any $\eps>0$ there exists a polynomial $H \in \R[x]$ such that
\begin{align} \label{eq: poly close to indicator function}
||H(x) - V(x)||_{L^1(\Lambda,\R)}< \eps, \\ \nonumber
	H(x_i)<0 \text{ for all } i \in \{1,...,N\}, \\ \nonumber
	H(x) < 1 \text{ for all } x \in \Lambda,
\end{align}
where $V(x)=1-\mathds{1}_{\{x_i\}_{i=1}^N}(x)$. 
\end{prop}
\begin{proof}
	We first show that there exists a smooth function that satisfies Eq.~\eqref{eq: poly close to indicator function}. We then approximate this function by a polynomial.
	
	Let $F(x):=1-\sum_{i=1}^N \eta \left( \frac{x-x_i}{\delta} \right)$, where $\eta \in C^\infty(\R^n, [0,\infty))$  is the bump function given by
	\begin{align*}
		\eta(x) = \begin{cases} 2e \exp \left( \frac{1}{||x||_2^2 -1}\right) \text{ for } ||x||_2<1\\
			0 \text{ otherwise.} \end{cases}
	\end{align*}
For more information on bump functions see~\cite{evans1998partial}.

Let $C:=\int_{\R^n} \eta(x) dx < \infty$ and $0<\delta< \left( \frac{\eps}{2 NC} \right)^{\frac{1}{n}}$, then
\begin{align} \label{F L1}
&||V(x) - F(x)||_{L^1(\Lambda,\R)} \le \sum_{i=1}^N \int_{\R^n} \eta \left( \frac{x-x_i}{\delta} \right) dx\\ \nonumber
& =\delta^n \sum_{i=1}^N \int_{\R^n} \eta(x) dx=\delta^n NC<\frac{\eps}{2}.
\end{align}
Moreover,
\begin{align} \label{F 1}
	F(x_i)=1 - \eta(0) - \sum_{j\ne u}\eta \left( \frac{x_i-x_j}{\delta} \right) \le 1 - \eta(0)= -1<0.
\end{align}
Furthermore, since $\eta(x) \ge 0$ for all $x \in \R^n$ it is clear that
\begin{align}  \label{F 2}
	F(x) = 1-\sum_{i=1}^N \eta \left( \frac{x-x_i}{\delta} \right) \le 1 \text{ for all } x \in \R^n.
\end{align}

Since $\Lambda$ is a compact set there exists $R>0$ such that $\Lambda \subset B_R(0)$. Now, $\eta \in C^\infty(\R^n, [0,\infty))$ and therefore $F \in C^\infty(\R^n, \R)$. Hence, by {Thm.~\ref{thm: Weierstrass approx thm}}, there exists a polynomial $P \in \R[x]$ such that
\begin{align} \label{F P}
	\sup_{ x \in B_R(0)} |F(x) - P(x) | < \frac{\eps}{4 (1 + \mu(\Lambda))}.
\end{align}
Let us consider $H(x):=P(x) -  \frac{\eps}{4(1 + \mu(\Lambda))}$. Then
\begin{align} \label{ F H }
	|F(x) - H(x)|  & \le 	|F(x) - P(x)|  + \frac{\eps}{4(1 + \mu(\Lambda))} \\ \nonumber
	& < \frac{\eps}{2(1 + \mu(\Lambda))} \text{ for } x \in B_R(0).
\end{align}
Eqs~\eqref{F L1} and~\eqref{ F H } imply that
\begin{align*}
|| V- H||_{L^1(\Lambda, \R)} & \le || V- F||_{L^1(\Lambda, \R)}  + || F- H||_{L^1(\Lambda, \R)} \\
& \le \frac{\eps}{2} + \mu(\Lambda)  ||F- H||_{L^\infty(\Lambda, \R)}\\
& < \eps.
\end{align*}
Moreover, by Eq.~\eqref{F P} we have $P(x) < F(x) +\frac{\eps}{4(1 + \mu(\Lambda))}$ for $x \in B_R(0)$ and hence $H(x)=P(x) - \frac{\eps}{4(1 + \mu(\Lambda))}< F(x)$. Therefore by Eqs~\eqref{F 1} and~\eqref{F 2} it follows that 
\begin{align*} 
	H(x_i)& <F(x_i)<0 \text{ for all } i \in \{1,...,N\}, \\ \nonumber
	H(x) & <F(x)< 1 \text{ for all } x \in \Lambda.
\end{align*}
\end{proof}

\subsection{ Miscellaneous Results} \label{sec: appendix miscaleneous}
\begin{thm}[The Bolzano Weierstrass Theorem~\cite{oman2017short}] \label{thm: Bolzano}
	Consider a sequence $\{x_n\}_{n \in \N} \subset \R^n$ that is bounded, that is there exists $M>0$ such that $x_n < M$ for all $n \in \N$. Then there exists a convergent subsequence $\{y_n\}_{n \in \N} \subset \{x_n\}_{n \in \N}$. 
\end{thm}

\begin{thm}[Putinar's Positivstellesatz \cite{putinar1993positive}] \label{thm: Psatz}
	Consider the semialgebriac set $X = \{x \in \R^n: g_i(x) \ge 0 \text{ for } i=1,...,k\}$. Further suppose $\{x  \in \R^n : g_i(x) \ge 0 \}$ is compact for some $i \in \{1,..,k\}$. If the polynomial $f: \R^n \to \R$ satisfies $f(x)>0$ for all $x \in X$, then there exists SOS polynomials $\{s_i\}_{i \in \{1,..,m\}} \subset \sum_{SOS}$ such that,
	\vspace{-0.4cm}\begin{equation*}
		f - \sum_{i=1}^m s_ig_i \in \sum_{SOS}.
	\end{equation*}
\end{thm}


\begin{lem}[\cite{clarke1975generalized}] \label{lem: max minimum of Lipschitz functions is Lipschitz}
Consider some compact set $X \subset \R^n$ and polynomial functions $\{g_i\}_{i=1}^m \in \R[x]$. Suppose
	\begin{align*}
	V_1(x) & := \inf_{w \in X} \min_{1 \le i \le m}g_i(x-w), \quad 
	V_2(x)  := \sup_{w \in X}  \max_{1 \le i \le m} g_i(x+w) ,
	\end{align*}
	then $V_1$ and $V_2$ are continuous functions.
\end{lem}


\begin{prop}[\cite{schlosser2021converging}] \label{prop: exietence of smooth sublevel set}
For each compact set $X \subset \R^n$ there exists a bounded function $p\in C^\infty(\R^n,\R)$ such that 
\begin{align*}
X=	\{x \in \R^n: p(x) \le 0\}.
\end{align*}
\end{prop}

\begin{lem}[Continuous functions have compact sublevel sets] \label{lem: Continuous functions have compact sublevel sets}
	If $f:\R^n \to \R$ is a continuous function then $\{x \in \R^n: f(x) \le \gamma \}$, where $\gamma \in \R$, is a closed set. Furthermore, if $\Lambda \subset \R^n$ is compact then $\{x \in \Lambda: f(x) \le \gamma \}$ is a compact set. 
\end{lem}
\begin{proof}
	Consider a converging subsequence $\{x_k\}_{k=1}^\infty \subset  \{x \in \Lambda: f(x) \le \gamma \}$ such that $x_k \to x^*$. By continuity we have $f(x_k) \to f(x^*)$. Since $f(x_k) \le \gamma$ for all $k$ it follows $f(x^*) \le \gamma$. Hence, $x^* \in \{x \in \Lambda: f(x) \le \gamma \}$ implying that $\{x \in \Lambda: f(x) \le \gamma \}$ is closed. Now, $\{x \in \Lambda: f(x) \le \gamma \} \subset \Lambda$ and $\Lambda$ is bounded. Since $\{x \in \Lambda: f(x) \le \gamma \}$ is closed and bounded it follows that it is compact.
\end{proof}


	    }{
}

\end{document}